\documentclass{article}
\usepackage[utf8]{inputenc}
\usepackage[T1]{fontenc}
\usepackage[english]{babel}
\usepackage[all,cmtip]{xy}
\usepackage{xfrac}
\usepackage{xcolor}
\usepackage{hyperref}
\usepackage[makeroom]{cancel}
\usepackage{tikz-cd}
\usepackage{appendix}
\usepackage{caption}
\usetikzlibrary{arrows}

\usepackage{amsmath}
\usepackage{amssymb}
\usepackage{amsthm} 

\newcommand{\calL}{\mathcal{L}}
\newcommand{\R}{\mathbb{R}}

\newcommand{\C}{\mathbb{C}}
\newcommand{\K}{\mathbb{K}}
\newcommand{\Z}{\mathbb{Z}}

\newcommand{\N}{\mathbb{N}}

\newcommand{\Oo}{\mathcal{O}}
\newcommand{\B}{\mathcal{B}}
\newcommand{\lan}{\langle}
\newcommand{\ran}{\rangle}

\newcommand{\CC}{\mathcal{C}}

\newcommand{\g}{\mathfrak{g}}
 \newcommand{\ttop}{{\overline{\top}}}

\newcommand{\Id}{\rm{Id}}
 
 \def\shuffle{{\sqcup\mathchoice{\mkern-3mu}{\mkern-3mu}{\mkern-3.2mu}{\mkern-3.8mu}\sqcup}} 

\newcommand{\cy}[1]{{ #1}}  
\newcommand{\sy}[1]{{ #1}} 
 \newcommand{\lf}[1]{{ #1}} 
\newcommand{\di}[1]{{ #1}} 

\newcommand\restr[2]{{
  \left.\kern-\nulldelimiterspace 
  #1 
  \vphantom{\big|} 
  \right|_{#2} 
  }}

\newtheorem{theorem}{Theorem}[section]
\newtheorem{prop}[theorem]{Proposition} 
\newtheorem{defnprop}[theorem]{Definition-Proposition}
\theoremstyle{definition}
\newtheorem{rk}[theorem]{Remark}
\newtheorem{defn}[theorem]{Definition}
\newtheorem{ex}[theorem]{Example}
\newtheorem{coex}[theorem]{Counter-example}
\newtheorem{lem}[theorem]{Lemma}
\newtheorem{cor}[theorem]{Corollary}
\newtheorem*{conj*}{Conjecture}
\newtheorem{conj}[theorem]{Conjectural statement}
\newtheorem{thm}[theorem]{Theorem}

\theoremstyle{remark}

\newtheorem{notation}{Notations}[subsection]

\usepackage{tikz} 
\usepackage{xparse} 
\usepackage{colordvi}
\usepackage{multicol}
\usepackage{enumitem}
\usepackage[normalem]{ulem}
\usepackage{epsfig}
\usepackage{caption}
\usepackage{ifpdf}
\ifpdf

\title{Tensor products   and the Milnor-Moore theorem\\
		 in the locality setup }
\author{Pierre J. Clavier${}^{1}$, Loic Foissy${}^{2}$,  Diego A. L\'opez${}^3$, Sylvie Paycha${}^{3,4}$\\
	~\\
	{\small \it $^1$ Université de Haute Alsace, IRIMAS,}\\
	{\small \it 12 rue des Frères Lumière,}\\
	{\small \it 68 093 MULHOUSE Cedex, France}\\ 
	~\\
	{\small \it $^2$ 
		LMPA, Centre Universitaire de la Mi-Voix,} \\
	{\small \it  50, rue Ferdinand Buisson,}\\
	{\small \it 62228 Calais Cedex - France }\\ 
	~\\
	{\small \it $^3$ Universit\"at Potsdam, Institut f\"ur Mathematik,} \\
	{\small \it Campus II - Golm, Haus 9} \\
	{\small \it Karl-Liebknecht-Stra\ss e 24-25} \\
	{\small \it D-14476 Potsdam, Germany}\\ 
	~\\
	{\small \it $^4$ On leave from the University of Clermont Auvergne, Clermont-Ferrand, France}\\
	~\\
	{\small emails: \it pierre.clavier@uha.fr, diealopezval@unal.edu.co, paycha@math.uni-potsdam.de}}

\date{}

\begin{document}

\maketitle

\begin{abstract}
The present exploratory paper deals with    tensor products  in the locality framework   {developed in previous work}, a natural setting  for an algebraic formulation of the locality principle in quantum field theory.  Locality  tensor products of locality vector spaces  raise challenging questions,  such as    whether the  locality   tensor
product of two locality vector spaces is a locality vector space. A related question is whether the quotient of locality vector spaces is a locality vector space,  which we first reinterpret in a group theoretic language and then in terms of  short exact sequences.   We prove a universal property for  the locality tensor algebra   and for  the   locality enveloping algebra, the analogs in  the locality  framework  of the tensor algebra   and of the   enveloping algebra. These universal properties hold under   compatibility assumptions between the locality and the multilinearity underlying the construction of tensor products which  we formulate   in the form of conjectural statements.  Assuming they hold true, we generalise the Milnor-Moore theorem  to the locality setup and discuss some of its consequences.
    \end{abstract}
    
    \noindent
	{\bf Classification:} 08A55, 16T05, 15A72

\tableofcontents

\section*{Introduction}

\addcontentsline{toc}{section}{Introduction}

\subsection*{Locality: state of the art}
 
 \addcontentsline{toc}{subsection}{Locality: state of the art}

 The notion of locality is ubiquitous in the physics and  the mathematics literature. The physical concept of locality states that an object 
or an event can only   interact with objects or events in its vicinity. Depending on the  approach one adopts, the concept of locality    can be expressed in various ways including commutativity of observables, disjointedness of supports or (partial) multiplicativity. We focus on the latter which is relevant in the Hopf algebraic approach to renormalisation and led to  
an abstract mathematical formulation of the physical notion of 
locality proposed    in \cite{CGPZ1}. Roughly speaking,   the "locality" framework  consists in enhancing the mathematician workman's categories such as that of sets, vector spaces, monoids, algebras, coalgebras, { Hopf algebras}  to locality categories by adjoining to a set, vector space, monoid, coalgebras {algebras or Hopf} algebra a symmetric binary relation $\top$  that we call a locality. This locality should satisfy some compatibility properties with the underlying structure. Modulo some small adjustments, this  "locality" framework can also host  the related concept of causality  in quantum field theory \cite{Rej}.   This locality setup tailored to ensure locality while renormalising was then implemented and discussed in various contexts  
\cite{CGPZ2,CGPZ3,CGPZ4}.

Tensor products are known to play an important role in renormalisation procedures e.g., in  Hopf algebras  used in  Connes and Kreimer's approach to renormalisation \cite{CK2}. This prompted the  notion of locality tensor product {and locality Hopf algebra} introduced in \cite{CGPZ1}.  In the  locality setup, tensor products require a special treatment  which gives rise to challenging questions. Whether the locality tensor product of two locality vector spaces is a locality vector space, a property which we expect to hold, raises the  question    whether the quotient of locality vector spaces 
is a locality vector space. Both issues are discussed in this paper, leading to interesting open questions.

In the present exploratory paper, we  formulate    conjectural statements  related to these questions and explore their consequences. Our main results are generalisations to the locality setup, of the  universality properties of tensor algebras and universal enveloping algebras as well as of the  Milnor-Moore theorem. The locality Milnor-Moore theorem provides new information which is not accessible in the ordinary non-locality set up.

\subsection*{Objectives and main results}

\addcontentsline{toc}{subsection}{Objectives and main results}

A first objective of this paper is to explore paths towards a full-fledged locality setup for tensor products. For this purpose we introduce {\bf pre-locality vector spaces}, namely vector spaces equipped with a set locality structure, the locality not being required to be compatible with the linear structure. \cy{We show in Theorem \ref{thm:equiv_categories} that the relation $\top_E$ on pre-locality vector space $(E;\top_E)$ can be  extended to a relation $\top_E^\ell$ such that $(E,\top_E^\ell)$ is a locality vector space. }
\lf{This defines a non essentially surjective functor} \sy{ from the category  of pre-locality spaces to the category  of  locality spaces, whose morphisms are locality linear maps in both cases, namely linear maps that map the graph of a locality relation on the source space to the graph of the locality relation on the target space (Definition \ref{defn:basic_loc_defn}).}

We then proceed to show that many of the important theorems which hold for usual (i.e. non pre-locality nor locality) tensor products still hold true in the pre-locality setup. We first show  the universality of the pre-locality tensor product (Theorem \ref{thm:univpptyloctensprod-nolocalityrel}). We then refine this result by formulating it in the pre-locality category (Theorem \ref{thm:univpropltp-preloc}).

We further prove the universality of the locality tensor algebra over a pre-locality vector space (Theorem \ref{thm:univpptyloctenalg-preloc}), and of the universal locality enveloping algebra in the category of pre-locality (Lie) algebras (Theorem \ref{thm:univproplocunivenvalg-preloc}). Whereas the universal property of the locality tensor product for pre-locality  vector spaces turns out to be equivalent to the universal property of the usual  tensor product of two vector spaces,  the universality of the  universal locality enveloping algebra only implies the universality of the universal enveloping algebra of a pre-locality algebra.    Along the way,      we prove various  useful properties and in particular the associativity of the  locality tensor product (Theorem \ref{thm:asso_pre_loc_tensor_prod}).

\vspace{0.5cm}

The second objective of the paper is to formulate precise conditions implying {the expected stability property, namely that} locality tensor products of locality vector spaces are  locality vector spaces. Like ordinary tensor products,  locality tensor products are defined as quotient spaces,  which brings us to the study of quotients of locality vector spaces and raises the question  (formulated in Question \ref{question}))  whether  the quotient of a locality vector space by a linear subspace is a locality vector space. We first reformulate this question  in group-theoretic terms (Theorem \ref{thm:cond_locality_quotient}) and then (see Proposition \ref{prop:simple_consequence_question})  in terms of locality short exact sequences     introduced in Definition \ref{defn:loc_short_ex_sequence}. By means of  some  counterexamples, we further show that Question \eqref{question} does not  always have a positive answer.

    Whether the quotient of a locality vector space by a locality subspace is also a locality space, relates  in turn 
to  the existence  of a {\bf strong locality complement} (Definition-Proposition \ref{defnprop:stronglocality}) of the locality subspace.   Roughly speaking, a strong locality complement of a subspace of a locality vector space is an algebraic complement of the given vector space with some locality conditions on the canonical projections induced by the direct sum of the subspace and its complement.
 We  show in Corollary \ref{cor:split_exact_seq} that  the quotient of a locality space by a   subspace with a strong locality complement is a locality space, which yields a sufficient condition  under which  Question \ref{question} has a  positive answer.

We then introduce  a second  sufficient condition -- the {\bf locality compatibility} condition (Definition \ref{defn:loccompatible}) -- on a linear subspace, which also ensures that the quotient of a locality space by this subspace is a locality space (Theorem \ref{thm:quotientlocality-vs}). It is  less stringent than the existence of a locality complement (see Proposition \ref{prop:strongloccomp}) and    turns out to be easier to work with. Theorem \ref{thm:quotientlocality-vs} underlies most of the forthcoming constructions, in particular the locality Milnor-Moore theorem.

We are now ready to  formulate  two main conjectural statements in terms of the (weaker) locality compatibility condition: 
	\begin{enumerate}
		\item Conjectural statement  \ref{conj:main1},  which ensures that the tensor product of  locality vector spaces is a locality vector space (Proposition \ref{prop:Conjecturefornbigger}) 
		\item   Conjectural statement \ref{conj:main2}, which ensures that the locality tensor algebra of a locality vector space is indeed a locality algebra ( {Theorem}  \ref{prop:tensor_alg_graded_loc}).
		\end{enumerate}

Assuming these conjectural properties hold true, we  can enhance the results of the first part {from the pre-locality} to the locality setup. Assuming they hold true, we prove the  universal {property} of the tensor power of a locality vector space (Theorem   \ref{thm:univpropltp}), which is the locality version of Theorem \ref{thm:univpropltp-preloc}. We also prove the universal property of {the} locality tensor algebra (Theorem \ref{thm:univpptyloctenalg}), which is the locality version of Theorem \ref{thm:univpptyloctenalg-preloc}. {   Assuming  that Conjectural statement \ref{conj:main2}  holds,, and under one further assumption (Conjectural statement \ref{conj:univenvalg}), we moreover}   prove the locality version of Theorem \ref{thm:univproplocunivenvalg-preloc}, namely the universal property of the locality universal enveloping algebra in the category of locality (Lie) algebras. Its universal property is then proven in Theorem \ref{thm:univproplocunivenvalg} under the assumption that the conjectural statements (\ref{conj:univenvalg}) and   \ref{conj:main2} both hold.

\vspace{0.5cm}

As a third and final objective of the paper, we prove  a locality version of the celebrated Milnor-Moore Theorem (also called Cartier-Quillen-Milnor-Moore Theorem, \cite{Car2,MM,Qui,Cha}), namely Theorem  \ref{thm:CQMMTheorem} which holds asuming the statements \ref{conj:main2} and \ref{conj:univenvalg} hold true. We hope that this is a first step towards  a  classification of locality Hopf algebras,  which was one of our original motivations.

The traditional proof of this theorem relies on a corollary of the Poincaré-Birkhoff-Witt theorem. Poincaré-Birkhoff-Witt theorem {is not yet available} in the locality setup mainly due to the fact that, unlike in the ordinary (non-locality) setup, one cannot reconstruct a locality vector space from a locality basis (whose generating vectors are mutually independent). Instead we prove the  required   intermediate step (Proposition \ref{prop:pbwcor}) using Zorn's lemma. To our knowledge, our proof is also new in the usual (non-locality) set up.

The article closes on some consequences of the Milnor-Moore Theorem inherent to the locality framework. Interestingly, it turns out that one cannot simply ``switch on'' locality, at least when the locality Milnor-Moore Theorem applies (Corollary \ref{cor:Nogocorollary}). 

The known correspondence between a class of sub-Hopf algebras and sub-Lie algebras of the set of their primitive elements extends to the locality setup, see Corollary \ref{cor:subalgsubHop} and the notations therein: 
\begin{equation*}
 \left((H, \top)\supset (H', \top')\right)\quad \longleftrightarrow \quad \left((\g', \top')\subset (\g, \top)\right).
\end{equation*}
However it follows from  Corollary \ref{cor:Nogocorollary} that  $H=H'$ and ${\mathfrak g}={\mathfrak g}'$  implies $\top=\top'$. We give an example (Proposition \ref{prop:locality_GL}) for which 
 $\mathfrak g=\mathfrak g'$ but $H\neq H'$, that is specific to the locality setup, since it cannot occur in the non-locality setup due to the Milnor-Moore Theorem.

 \subsection*{Organisation of the paper}
 
 \addcontentsline{toc}{subsection}{Organisation of the paper}

 The paper is divided into three parts (and one appendix), each corresponding to one of the  objectives described above. 
 The properties proved in the pre-locality setup in this first part shed light on  the challenges of the locality structures that we address in the second part of the paper. They also serve as a preparation to establish the  corresponding properties in the full-fledged locality setup when assuming that the conjectural statements    hold true.
 This preliminary study in the pre-locality framework  which is interesting  in its own right,   serves to identify   obstructions preventing tensor products of locality structures from  being locality vector spaces.  
 
 The first part  introduces and proves universality results in the pre-locality setup. In the second part,  we question the locality property  of   the quotient of a locality vector space by a locality subspace,  formulate conjectural statements that play a role when dealing with locality tensor products and   further investigate useful consequences of these conjectural properties. 
 
The various conjectural statements discussed in  this second part are formulated in a way that suggests a strategy for proving them. Although such proofs are not within the scope of the current work, we expect to tackle them in future work. The fact that these conjectural statements    allow to extend important results of the theory of Hopf algebras in the locality setup serves as a motivation. Hopf algebraic results are  the object of the third part of this paper. This final third part is dedicated to more sophisticated consequences of these conjectural statements, namely the locality version of the Milnor-Moore Theorem. Each of these parts is divided into sections whose contents we now present in more detail.
 
 The first section of the first part, namely Section \ref{section:one} introduces the basic concepts we use throughout this paper: locality vector spaces, pre-locality vector spaces and locality tensor products of these structures. Section \ref{sec:localityrelationontp} introduces locality on quotients of locality vector spaces, applying them to locality tensor products of pre-locality spaces. Section \ref{section:three} deals with products of more than two pre-locality vector spaces. Finally, Section \ref{sec:Tensoranduniversalpla} discusses properties of the locality tensor algebra and the locality universal enveloping algebra in the pre-locality setup.

  The second part of the paper starts with Section \ref{section:five} which discusses 
 when the   quotient of a locality vector space with a linear subspace, inherits a locality vector space structure. 
  	 These considerations on quotients allow us to give in Section \ref{sec:main_conjectures} two sufficient conditions, one stronger than the other, for quotients of locality vector spaces to be locality vector spaces. The remaining part of Section \ref{sec:main_conjectures} is dedicated to the formulation and  implications of two conjectural statements under which  locality tensor products and locality tensor algebras are locality vector spaces and locality algebra respectively. In Section \ref{sec:Firstconsequences} we infer that assuming the conjectural properties to hold true,     the universal properties shown in the first part of the paper in the pre-locality setup still hold in the locality one. One of these universal properties is proved under the assumption of a third conjectural property, also formulated in Section \ref{sec:Firstconsequences}.
  
 The third and final part of the paper focuses on the locality version of the Milnor-Moore Theorem. It starts in Section \ref{section:seven} by recalling (locality version of) the structures the Milnor-Moore Theorem implies, namely graded connected Hopf algebras, and establishing some useful preliminary results. In the final section \ref{section:eight} we state and prove, assuming the validity of the conjectural statements in the second part, the Milnor-Moore Theorem
 in the locality setup  and state some of its consequences specific to the locality framework. 
 
  For the sake of completeness, we  dedicate a short appendix to the description of an alternative locality tensor product which unlike the one we use throughout the paper, is not a subspace of the usual tensor product.

 	\subsection*{Openings}
 	
 	\addcontentsline{toc}{subsection}{Openings}
 	
This paper  provides
 	first steps toward  a full-fledge theory of locality tensor products. Thanks  to the notion of pre-locality vector space, we were able to identify the   challenges and formulate precise conditions under which we could merge locality and tensors products. 
We believe that discussing  these issues paves to way to a  complete theory of locality Hopf algebras. Beyond its relevance as a generalisation of the theory of ordinary Hopf algebras, it  is of importance in the context of  renormalisation theory, which in turn was one of the original motivations for introducing locality structures.
 	
Let us hint to what could be the next steps of our program. The three conjectural statements discussed in this paper raise challenging questions which we feel are interesting for their own sake. As shown in the paper, they are related to open questions in group theory. Yet other approaches might also be fruitful, such as rewriting systems and Gröbner bases  which seem natural candidates to tackle these questions.
 	
 	Another natural goal is  to  establish a locality version of  the Poincaré-Birkhoff-Witt Theorem. Such a result would simplify our proof of the Milnor-Moore Theorem and (possibly) avoid the use of the axiom of choice. Indeed, while the original proof of  Birkhoff and Witt requires the use of the axiom of choice, more recent  generalisations of their theorem do not. It could also be of interest to see if the three conjectural statements of this paper (and in particular the third one) are enough to prove a locality version of the Poincaré-Birkhoff-Witt Theorem.


\sy{  An interesting open question  concerns the interplay of locality and duality. Indeed, there is not yet a natural notion of locality on the algebraic dual of a locality vector space, which turns the dual into a (non-trivial) locality vector space. Similarly, the algebraic dual of a locality algebra (resp. coalgebra) is not in general a locality coalgebra (resp. algebra). 
%

	The locality Milnor-Moore Theorem gives a classification of  cocommutative graded connected locality Hopf algebras and we hope for   other classification results for locality structures, and in particular for locality Lie algebras. 
Our motivation to search for a locality dual is that the  dual version of the locality Milnor-Moore Theorem could potentially give rise to a classification  of graded, connected commutative Hopf algebras  derived from Theorem \ref{thm:CQMMTheorem}. Such a dual Milnor-Moore Theorem can also be proved in a more pedestrian way, dualising the proof of this paper, which for the sake of the length of the paper, we have omitted here.  }

\section*{Acknowledgments}

The  paper was initiated and much of it was written while the first author held a postdoctoral position at the University of Potsdam and the Technical University of Berlin. The authors are grateful to the Perimeter Institute for hosting them during the preparation of this paper. The second author thanks the DAAD for funding his PhD of which this article is a part. We are very grateful to   Joachim Kock for inspiring discussions and providing useful references, as well as to Li Guo and Bin Zhang for their  precious  comments on a first version of the paper.

\section*{Notations} 

	Throughout the   document we use the notation $[n]:=\{1,\cdots,n\}$, for $n\in\N$ and with the convention $[0]=\emptyset$. $\K$ denotes a commutative field of characteristic zero, and unless otherwise specified we take  $\K=\C$ or $\R$ to be the underlying field of every vector space. Given a set $X$, we denote the free span of $X$ by $\K(X)$, $\K X$ or  $\lan X\ran$ indistinctly.
	
	\newpage

\part{Universal properties in the pre-locality setup} \label{part:one}

This first part of the document is devoted to the construction of tensor products as quotient spaces in the locality and 
pre-locality setup and to derivation of their first properties.

\section{ {  (Pre-) locality spaces and the locality tensor product}} \label{section:one}

Let us start by introducing the most important concepts of this paper, namely locality and pre-locality structures.

 \subsection{ {(Pre-) locality vector spaces}}

 We recall and partially extend some basic notions  introduced in \cite{CGPZ1}  relative to linear structures.
\begin{defn} \label{defn:basic_loc_defn}
\begin{itemize}
	\item A {\bf locality set} is a pair $(S,\top)$ where $S$ is a set and $\top\subset S\times S$ is a symmetric relation on $S$. We sometimes denote $(x,y)\in\top$ as $x\top y$.
	\item  A {\bf locality map} is a map $f: (X, \top_X)\longrightarrow (Y, \top_Y)$ between locality sets which preserves locality i.e., such that $(f\times f)(\top_X)\subset \top_Y$.
	\item  Two maps $f,g:(X, \top_X)\longrightarrow (Y, \top_Y)$ between locality sets are called locally independent (or simply independent if there is no risk of ambiguity) if \[(f\times g)(\top_X)\subset \top_Y.\] When $X=Y$ we sometimes denote $f$ and $g$ being locality independent as $f\top g$.
	\item  A {\bf pre-locality $\K$-vector space} is a  locality set $(V, \top)$ such that $V$ has the structure of a $\K$-vector space{, and $(0_V,0_V)\in\top$}.
	\item  A {\bf locality $\K$-vector space} or simply {\bf locality vector space} is a  pre-locality $\K$-vector space $(V,\top)$ such that the polar set $U^\top=\{x\in V: (\forall u\in U),\, u\top x\}$ of   any  subset $U\subset V$ is a linear subspace of $V$. Equivalently, the following condition should be fulfilled
 	\begin{equation} \label{eq:linearloc} 
 	 u\top v\, {\rm  and}\, u'\top v \Longrightarrow (\lambda u+  \lambda' u')\top v \quad \forall (\lambda,  \lambda' )\in \K^2, \, \forall (u, u', v )\in V^3. 
    \end{equation}
{We sometimes refer to the previous property as linear locality.}
 \item Let $(V,\top_V)$ and $(W,\top_W)$ be two  pre-locality $\K$-vector spaces. We call a linear map $f:V\to W$  a \sy{ {\bf locality  morphism} or {\bf locality linear map} } if it is also a locality map i.e., if it is locality independent of itself: 
 \begin{equation} \label{eq:localf}
  f\top f~\Longleftrightarrow~(f\times f)(\top_V)\subset \top_W.
 \end{equation}
 	\item Let $(W,\top_W)$ be a (pre-)locality vector space and $V\subseteq W$. We call $(V, \top_V)$ a {\bf (pre-)locality subspace} of $(W, \top_W)$ if the inclusion map $\iota: V\hookrightarrow W$ is a locality linear map.
 
 \item  Let $(V,\top_V)$ and $(W,\top_W)$ be two pre-locality $\K$-vector spaces (resp. locality vector spaces). We say they are isomorphic as pre-locality (resp. locality) vector spaces if there is an invertible locality linear map $f:V\to W$ such that $f^{-1}$ is also a locality linear map.
 \end{itemize}
\end{defn}
\begin{rk}
 (Pre-)locality is a hereditary property. Indeed, any linear subspace $W$ of a (resp. pre-) locality vector space  $(V,\top)$  can be endowed with a locality structure $\top_W:=\top\cap(W\times W)$  induced by $\top$, in which case $(W,\top_W)$ is a (resp. pre-) locality vector space. Our definition of (pre-)locality subspace (which is more general than the one in \cite{CGPZ1}) amounts {to require} that $V\subseteq W$ and $\top_V\subseteq\top_W$ such that $(V,\top_V)$ is a (pre-)locality vector space.
\end{rk}
We recall that locality sets  build a category ${\mathbf S}_{\rm loc}$  whose morphisms are locality maps.
\begin{prop}
     Let $(V,\top_V)$ and $(W,\top_W)$ be two (resp. pre-)locality vector spaces and $f:V\to W$ be a locality linear map. Then the image of any (resp. pre-)locality subspace of $(V,\top_V)$ by $f$ is a (resp. pre-)locality subspace of $(W,\top_W)$ and the  {collection} ${\bf V}_{\rm loc}$   (resp. ${\bf V}_{\rm pre-loc}$)  of (resp. pre-)locality vector spaces, forms a category whose morphisms are  locality morphisms.
    \end{prop}
\begin{proof}
  Let $f: V\to W$  {be a  morphism as in the statement.}
	\begin{itemize}
		\item
	 Since $f$ is a  locality map, it sends locality  sets to locality sets, so it maps the (pre-)locality vector space $(V, \top_V)$ to  the   vector space  $f(V)$ equipped with the locality structure $ \top_{f(V)}:=\top_W\cap (f(V)\times f(V))$ inherited from $  \top_W$, which makes $(f(V),\top_{f(V)})$ a  { pre-locality linear subspace of $(W, \top_W)$}.
 \item  If $(V, \top_V)$ is a locality vector space, so is $(f(V), \top_{f(V)}) $. Indeed, the polar set \[f(X)^{\top_{f(V)}}=\left\{w\in f(V), w\top f(x),\, \forall x\in X\right\}=f(X)^{\top_W}\cap f(V)\] of the range $f(X)$ of a subset $X\subset V$ is a vector space for by assumption $f(X)^{\top_W}$ is a vector space.
 \item The fact that the composition of two locality maps is a locality maps was already proven in \cite{CGPZ1} (see Remark 2.6). Associativity of the composition of locality maps follows from the usual associativity of composition. Thus the fact that ${\mathbf V}_{\rm loc}$ and ${\mathbf V}_{\rm pre-loc}$ are categories as claimed follows from the fact that the identity maps trivially are locality morphisms. \qedhere
\end{itemize}
	\end{proof}
	We have the following forgetful functors between the various  categories: 
\begin{equation*} 
\begin{tikzcd}
& & &  {{\mathbf V}}\\
{\mathbf V}_{\rm loc}\arrow[rr, hook, "\iota_ {{\mathbf V}_{pl}}"]& &{\mathbf V}_{\rm pre-loc}\arrow[ur, hook, "\iota_ {\mathbf V}"] \arrow[dr, hook, "\iota_ {\mathbf S}"] &\\
& & &{\mathbf S}_{\rm loc}
\end{tikzcd}
\end{equation*}  { where $\mathbf V$ denotes the category of vector spaces. The map $\iota_{V_{pl}}$  } takes a locality vector space to its underlying pre-locality vector space (thus ignoring the linearity condition on polar sets),  {$\iota_V$ ignores the pre-locality and takes it into (usual) vector spaces}	and $\iota_S$   ignores the linearity.

  \begin{defn}
 Let $V$ and $W$ be two linear subspaces of a (resp. pre-)locality vector space $(E,\top)$. 
 The \textbf{locality Cartesian product} of $V$ and $W$, denoted by $V\times_\top W$ is  {the restriction }

 \begin{equation*}
      V\times_\top W:=\top\cap(V\times W)
     \end{equation*}
 \end{defn}
 {Note that $(0_V\times 0_W)\in V\times_\top W$  {since $0_V=0_W=0_E$}.}
	\begin{rk}
	  Assume that $(E, \top)$ is a locality vector space. For any subsets $X\subset V$ or $Y\subset W$, the  relative polar sets $X^{\top}\cap W=\{w\in W: (x,w)\in\top\  \forall x\in X\}  $   and $ Y^\top \cap V:=\{v\in V: (v,y)\in\top\ \forall y\in Y\}$ are linear subspaces of $W$ and $V$ since $(E, \top)$ is a locality vector space. In the terminology  of \cite[Sect. 4.1]{CGPZ1}, the triple $(V,W,\top)$  is a  relative locality vector space. 
\end{rk}

Before generalising the concept of bilinear map to the locality set up, recall that the freely generated vector space $\K X$ of a set $X$ obeys the following universal property. 

Given a set $X$  and a vector space $G$, any map $f:X\to G$  uniquely extends  to a linear map $\bar{f}:\K X \to G$ as follows \[\bar f \left(\sum_{x\in X}\alpha_x\,x\right)=\sum_{x\in X}\alpha_x\,f(x).\]

The following diagram, where $\iota$ stands for the canonical injection, therefore commutes:
	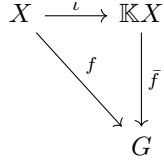
\begin{figure}[h!] \label{diag:GammaPgk} 
	\begin{center}
		\begin{tikzcd}
		X\arrow[r,"\iota"]\arrow[rdd,"f"]&\K X \arrow[dd,"\bar f"]\\
		\\
		&G\\
		\end{tikzcd}
		\caption{Universal property of freely generated vector spaces}\label{fig:}
	\end{center}
\end{figure}

In the following, $V$ and $W$ are two vector spaces over the same field $\K$. In the free linear span $\K(V\times W)$, we consider the linear subspace ${\rm I_{bil}}$ generated by all elements of the form
\begin{gather}\label{eq:bilforms1} 
(a+b,x)-(a,x)-(b,x)\\
(a,x+y)-(a,x)-(a,y)\label{eq:bilforms2}\\
(ka,x)-k(a,x)\label{eq:bilforms3}\\
(a,kx)-k(a,x)\label{eq:bilforms4}
\end{gather}

with $a,b\in V$, $x,y\in W$ and $k\in\K$ the underlying field. Clearly, given any vector space $G$, a  map $f: X\times  Y\to G$ is bilinear if and only if $ \bar f(I_{\rm bil})=\{0_G\}$.

The notion of local bilinearity is inspired from \cite[Paragraph 3.3]{CGPZ1}, adapted to our construction of the (pre-)locality tensor product.
\begin{defn}\label{defn:biliresplocality}
Let  {$V$ and  $W$ be subspaces of a pre-locality vector space $(E,\top)$} and $G$ any vector space. We call {\bf  $\top_\times$-bilinear} a map $f:V\times_\top W\to G$  
which satisfies the  $\top_\times$-bilinearity condition:  \begin{equation}\label{eq:topbilinear}  \bar f(I_{\rm bil}^{\top_\times})=\{0_G\},\end{equation}
	where we have set $I_{\rm bil}^{\top_\times}:=\K(V\times_\top W) \cap I_{\rm bil}$. 
 \end{defn}

	\begin{defn}\label{defn:locbiliresplocality}
		Let  {$V$ and  $W$ be subspaces of a pre-locality vector space $(E,\top)$} and   { $(G, \top_G)$ a pre-locality vector space}. We call a map $f:V\times_\top W\to G$  
		{\bf locality  $\top_\times$-bilinear}, if it is $\top_\times$-bilinear   (see (\ref{eq:topbilinear}))  and   satisfies the 
		locality condition 
		\begin{equation}\label{eq:locbilinear} 
		  {(f\times f)}( {\top_{V\times_\top W} })\subset \top_G,\quad{\rm  i.e.,} \quad  (v, w) {\top_{V\times_\top W}}(v',w')\Rightarrow f(v, w)\top_Gf(v',w'),
		 \end{equation} 
		  where $ {\top_{V\times_\top W}}$ is a locality relation on $V\times_\top W$ given by 
		 \begin{equation}\label{eq:locreltoptimes}
		   {\top_{V\times_\top W}}:=\{((v,w),(v',w'))\in (V\times_\top W)^2: v\top_V v'\, \wedge w\top_W w' \}. 
		  \end{equation}	
		  \end{defn} 
		  

\cy{Our new concept of pre-locality vector space is related to the previous notion of locality vector space through the following result.}
\begin{thm} \label{thm:equiv_categories}
 Let $(E,\top_E)$ be a pre-locality $\K$-vector space. There is a unique locality relation $\top^\ell_E$ on $E$ containing $\top_E$
such that for any locality space $(F,\top_F)$ and any \sy{locality} linear map $f:(E,\top_E)\longrightarrow (F,\top_F)$,
$f$ is a locality linear map  from $(E,\top_E^\ell)$ to $(F,\top_F)$.

\cy{The map  \begin{eqnarray*}\calL: {\bf V}_{\rm pre-loc}&\longrightarrow &{\bf V}_{\rm loc}\\
	 (E,\top_E)&\longmapsto &(E,\top_E^\ell)\end{eqnarray*}
on the  objects of the category combined with the identity map $\calL(f)=f$   on the morphisms   of the category given by locality linear maps $f$,} \lf{defines a non essentially surjective functor between the two categories.  } 
 
\end{thm}

\begin{proof}
\begin{itemize}
 \item Existence: \sy{With the notations of the theorem}, let $\mathcal{E}$ be the set of locality relations $\top$ on $E$ such that $\top_E\subseteq \top$. 
This is not empty, as it contains $E\times E$. We consider
\[\top^\ell_E=\bigcap_{\top\in \mathcal{E}} \top.\]
Then $\top_E\subseteq \top^\ell_E$. If $x\top_E^\ell y$,  $x\top_E^\ell z$ and $\lambda \in \mathbb{K}$,
then for any $\top$ in $ \mathcal{E}$, $x\top y$ and $x\top z$, so $x\top (y+\lambda z)$ \cy{since by definition of $\mathcal{E}$, $\top$ is a locality relation}. Hence,
$x\top_E^\ell \cy{(x+\lambda z)}$, and $\top_E^\ell$ is a locality relation on $E$.

Let  $f:(E,\top_E)\longrightarrow (F,\top_F)$ be a \sy{locality} linear map. We define a relation $\top$ \cy{on $E$} by
\sy{	\[   \top_f:= (f\times f)^{-1}(\top_F). \]
	 Since $f$ is a locality map, $(f\times f)(\top_E)\subset \top_F$ so $\top_E\subset \top_f$.}
Combining the locality  of the map $f$, which  implies   \[ \left(x\sy{\top_f} y\wedge x\sy{\top_f} z\right)\Longrightarrow \left(f(x)\top_F f(y)\wedge  f(x)\top_F f(z)\right),\] with the fact that $(F,\top_F)$ is a locality space yields   
\[f(x)\top_F \cy{(f(y)+\lambda f(z))}~\Longrightarrow f(x)\top_F f(y+\lambda z),\]
so $x \sy{\top_f} \cy{(y+\lambda z)}$, and hence $\sy{\top_f}$ lies in $ \mathcal{E}$. \sy{Thus, $ \top_E^\ell \subset \top_f$, which as before implies that $f: (E, \top^\ell)_E\to (F, \top_F)$ is a locality linear map.}

\item  Uniqueness: let $\top'$ be another relation \sy{obeying the assumptions of the theorem}.
 Since  $\top_E\subseteq \top_E^\ell$, the identity map $\Id_E:(E,\top_E)\longrightarrow (E,\top_E^\ell)$ is a pre-locality linear map. Thus, \sy{by the very definition of $\top'$,} $\Id_E:(E,\top')\longrightarrow (E,\top_E^\ell)$ is a locality linear map. \cy{Consequently, $\top'\subseteq\top_E^\ell$}. 
Similarly, exchanging the roles of $\top'$ and $\top_E^\ell$, $\Id_E:(E,\top_E^\ell)\longrightarrow (E,\top')$
is a locality linear map, so $\top_E^\ell=\top'$. 

\cy{
\item ${\mathcal L}$ defines a functor: Given two pre-locality spaces $(E,\top_E)$,  $(F,\top_F)$, we want to show that a \sy{locality linear map}  $f:(E,\top_E)\longrightarrow(F,\top_F)$ extends to  a locality linear map $f:(E,\top_E^\ell)\longrightarrow(F,\top^\ell_F)$.   We first observe that since $\top_F\subset \top_F^\ell$, $f:(E,\top_E)\longrightarrow(F,\top^\ell_F)$ is also a \sy{locality linear map}. It then follows from \sy{the  definition of $T_E^\ell$},  that the map $f:(E,\top_E^\ell)\longrightarrow(F,\top^\ell_F)$ is indeed a locality linear map.}

\item  \lf{${\mathcal L}$ is not essentially surjective:  the vector space
$(\K,\top)$ with $\top=\{(0,0),(1,1)\}$  is a pre-locality vector space. A linear map $f:\K\longrightarrow \K$ is  a pre-locality map if and only if $(f(1),f(1))\in \{(0,0),(1,1)\}$. Therefore,
\[\mathrm{End}_{\mathrm{pre-loc}}(\K,\top)=\{0,Id_\K\}.\] }
 \sy{For a locality relation $\top'$ containing $\top$, we have $1\top 1\Rightarrow 1\top'\lambda$ for any $\lambda\in \K$ and from there $\lambda \top'\mu$  for any $(\lambda, \mu)\in \K^2$ so that $\top^\ell=\K\times \K$.} \lf{It follows that 
\[\mathrm{End}_{\mathrm{loc}}(\K,\top^\ell)=\{\lambda Id_\K,\lambda \in \K\}.\]
The map $f\mapsto \calL(f)$ from $\mathrm{End}_{\mathrm{pre-loc}}(\K,\top)$ to $\mathrm{End}_{\mathrm{loc}}(\K,\top^\ell)$
is therefore not surjective.}
\qedhere
\end{itemize}

\end{proof}

\subsection{ {Locality tensor product of two pre-locality spaces}} \label{subsec:tensor_prod_pre_loc}

	We adapt to the pre-locality setup, concepts  introduced in \cite{CGPZ1} relative to  locality  tensor products. The assumptions are weakened, yet the definition stays the same so we keep the terminology "locality tensor product".  
 As before,  $V$ and $W$ are two vector spaces over the same field $\K$.  Let us recall that the tensor product of two vector spaces $V$ and $W$ reads

\[V\otimes W:=\sfrac{\K (V\times W)}{{\rm I_{bil}}}.\]
 {It comes with a map 
	\[\otimes :=\pi\circ\iota:V\times W\to V\otimes W\] 
	built from the canonical inclusion $\iota:V\times  W\to\K(V\times  W)$ and  the canonical quotient map  $\pi:\K(V\times  W)\to V\otimes  W$, which makes the following diagram commute}

\begin{figure}[h!] \label{diag:GammaPgk} 
	\begin{center}
		\begin{tikzcd}
		V\times W\arrow[r,"\iota"]\arrow[rdd,"\otimes"]&\K (V\times W) \arrow[dd,"\pi"]\\
		\\
		&V\otimes W
		\end{tikzcd}
		\caption{ {The} tensor product  {from the Cartesian product} }
	\end{center}
\end{figure}

 Following \cite{CGPZ1}, we define the  (pre-)locality counterpart.

\begin{defn}\label{defn:loctensprod1}			
	Given  $V$ and  $W$ subspaces of a pre-locality vector space $(E,\top)$, the  locality  tensor product is the  vector space
	\begin{equation} \label{eq:VotimestopW}
	 V\otimes_{\top}W:=\sfrac{\K(V\times_\top W)}{I_{\rm bil}^{ \top_\times}}. 
	\end{equation}
	
\end{defn}
\begin{rk}
	Since $ V\times_\top W\subset V\times W$ and  {$I_{\rm bil}^{\top_\times}:=\K(V\times_\top W) \cap I_{\rm bil}$},  we have an inclusion of vector spaces $V\otimes_{\top}W\subset V\otimes W$. If $V\times_\top W= V\times W$, then $V\otimes_{\top}W= V\otimes W$.
\end{rk}
 {$V\otimes_{\top}W$ comes with a map 
	\begin{equation}\label{eq:otimestop}\otimes_\top :=\pi_\top \circ\iota_\top:V\times_\top W\to V\otimes_\top W\end{equation}
	built from the canonical inclusion $\iota_\top:V\times_\top  W\to\K(V\times_\top  W)$ and  the canonical quotient map  $\pi_\top:\K(V\times_\top  W)\to V\otimes_\top  W$, which makes the following diagram commute:}

\begin{figure*}[h!]
	\begin{center}
		\begin{tikzcd}
		V\times_\top W\arrow[r,"\iota_\top"]\arrow[rdd,"\otimes_\top"]&\K (V\times_\top W) \arrow[dd,"\pi_\top"]\\
		\\
		&V\otimes_\top W
		\end{tikzcd}
		\caption{ {The} locality tensor product  {from subspaces of a pre-locality vector space.} }\label{fig:}
	\end{center}
\end{figure*}  
\begin{prop}\label{prop:otimesastxbilinear}
 Given $V$ and  $W$ subspaces of a pre-locality vector space $(E,\top)$, the map \[\otimes_\top: V\times_\top W\to V\otimes_\top W\] is a $\top_\times$-bilinear map.
\end{prop}
\begin{proof} Let  $\overline{{\otimes}_\top}$ be  {the linear extension of  (\ref{eq:otimestop}) to a map  $\K( V\times_\top W)\to V\otimes_\top W$.  By construction, } $ \overline{{\otimes}_\top}( I_{\rm bil}^{\top_\times})= \pi_\top\left(I_{\rm bil}^{\top_\times}\right)$   coincides with $\{0_{V\otimes_\top W}\}$.  The map ${\otimes}_\top$  therefore satisfies (\ref{eq:topbilinear}) and defines a $\top_\times$-bilinear map. 
\end{proof} 
Properties  of the non-locality tensor product commonly used are the distributivity with respect to direct sums \begin{equation}\label{eq:distributivity}(V_1\oplus V_2)\otimes W=(V_1\otimes W)\oplus (V_2\otimes W), \end{equation} and with respect to  intersections, namely $(V_1\cap V_2)\otimes W=(V_1\otimes W)\cap(V_2\otimes W)$. These two properties rely on the existence of a direct complement to every subspace $V_1\subset V$, i.e, the existence of a subspace $V_2\subset V$ such that $V_1\oplus V_2=V$.  However, for the locality tensor product, distributivity does not always hold  as the following example illustrates.
    \begin{coex}
    Consider $\R^2$ with the orthogonality locality. Then \begin{equation}
    (\langle e_1\rangle\otimes_\top \R^2)\oplus(\langle e_2\rangle\otimes_\top \R^2)=\langle e_1\otimes e_2\rangle \oplus \langle e_2\otimes e_1\rangle, 
\end{equation}
is not $\left(\langle e_1\rangle \oplus \langle e_2\rangle\right)\otimes_\top \R^2$ since it does not contains  $(e_1+e_2)\otimes (e_1-e_2)\in \R^2\otimes_\top\R^2$.    
    \end{coex}

In order to accomodate for distributivity properties of the tensor product in the locality set up, we need to ensure some compatibility of the splitting $V=V_1\oplus V_2$ with the locality relation $\top$. The following proposition gives sufficient conditions to   ensure the   distributivity of the locality tensor product w.r. to the direct sum. 

\begin{prop} \label{prop:tproductanddirectsum}
Let $V$ and $W$ be linear subspaces of a pre-locality vector space $(E,\top)$, and let $V_1$ and $V_2$ be subspaces of $V$ such that $V_1\oplus V_2=V$.
\begin{enumerate}
	\item  If for $\{i, j\}=\{1, 2\}$ the projection maps $\pi_i:V\to V_i$ onto $V_i$ along $V_j$ are locally independent of the identity map ${\rm Id}_W$ on $W$ (i.e., $(\pi_i\times {\rm Id}_W)(\top\vert_{V\times W})\subset \top\vert_{V_i\times W}$), then \[(V_1\otimes_\top W)\oplus(V_2\otimes_\top W)=V\otimes_\top W.\]

\item If   $(E,\top)$ is a locality vector space, and   one of the projections $\pi_i$ is locally independent of ${\rm Id}_W$, then the other projection is also locally independent of ${\rm Id}_W$.
\end{enumerate}
\end{prop}  
\begin{proof}
\begin{enumerate}
\item We  first prove that $(V_1\otimes_\top W)\oplus(V_2\otimes_\top W)\subset V\otimes_\top W.$ Since $V_i\subset V$ then $V_i\otimes_\top W\subset V\otimes_\top W$ and the expected inclusion follows. To prove   the other direction, w.l.o.g.  consider $a\otimes b$ in $V\otimes_\top W$ such that $(a,b)$ lies in $\top\vert_{V\times W}$. Since $\pi_i$ and ${\rm Id}_W$ are locally independent $(\pi_i(a),b)\in\top|_{V_i\times W}$ thus $\pi_i(a)\otimes b\in V_i\otimes_\top W$, which implies that \[a\otimes b=(\pi_1(a)+\pi_2(a))\otimes b=\pi_1(a)\otimes b + \pi_2(a)\otimes b\in (V_1\otimes_\top W)\oplus(V_2\otimes_\top W).\]
\item Moreover, if $(E,\top)$ is a locality vector space and if $\pi_1$ is independent of ${\rm Id}_W$, then for any pair $(a,b)\in V\times W$ we have
$a\top b\Longrightarrow \pi_1(a)\top b$. Using  $\pi_2(a)=a-\pi_1(a)$, it follows that $\pi_2(a)\top b$ since $\{b\}^\top$ is a subvector space of $E$. Thus $\pi_2$ is independent of ${\rm Id}_W$ as claimed. \qedhere
\end{enumerate}
\end{proof}

We prove a (weak) locality version of the distributivity property w.r. to the intersection $(V_1\cap V_2)\otimes W=(V_1\otimes W)\cap(V_2\otimes W)$.

\begin{cor}\label{cor:Localitytensorintersection}
Let $V$ and $W$ be linear subspaces of a pre-locality vector space $(E,\top)$, and let $V_1$, $V_2$ be subspaces of $V$. Let   $V_2'$ be a direct complement  of the intersection $V_1\cap V_2$ in $V_2 $ i.e.  $(V_1\cap V_2)\oplus V_2'=V_2$.
\begin{enumerate}
\item If the projection maps $\pi:V_2\to V_1\cap V_2$ onto $V_1\cap V_2$ along $V_2'$  and ${\rm Id}_{V_2}-\pi$  are locally independent of the identity map ${\rm Id}_W$ on $W$, then 
\[(V_1\cap V_2)\otimes_\top W=(V_1\otimes W)\cap(V_2\otimes_\top W). \]
\item In particular, if $V_1\subset V$, and if the projection maps $\pi:V \to V_1$ onto $V_1 $ along $V_2'$ and ${\rm Id}_{V_2}-\pi$ are locally independent of the identity map ${\rm Id}_W$ on $W$, then
\[V_1\otimes_\top W=(V_1\otimes W)\cap(V\otimes_\top W).\]
\end{enumerate}
\end{cor}
\begin{proof}

\begin{enumerate}
\item  Using the distributivity property of the locality tensor product 
$ ((V_1\cap V_2)\otimes_\top  W) \oplus(  V_2'\otimes_\top W) = V_2\otimes_\top W  $  which follows from Proposition \ref{prop:tproductanddirectsum}, we have
\begin{equation}\label{eq:distriboftensorandcap}
 (V_1\otimes W)\cap(V_2\otimes_\top W)= (V_1\otimes W)\,\bigcap\,\left(((V_1\cap V_2)\otimes_\top  W)\oplus (V_2'\otimes_\top  W)\right).
\end{equation}
We now make use   of a result of elementary linear algebra, namely $A\cap(B\oplus C)=B$ whenever $A,B$ and $C$ are linear subspaces of a   linear space $E$ with $B\subset A$ and $A\cap C=\{0\}$. Since $(V_1\cap V_2)\otimes_\top  W\subset V_1\otimes W$ and $(V_1\otimes W)\cap(V_2'\otimes _\top W)=\{0\}$, the right hand side of \eqref{eq:distriboftensorandcap} is equal to $(V_1\cap V_2)\otimes_\top  W$, which yields the result.

\item  Setting $V_2=:V$ in the previous item, yields the result since $V_1'=\{0\}$. \qedhere
\end{enumerate}
\end{proof}

\subsection{A first universal property of the locality tensor product}

We recall the universal property of the usual tensor product.  Given three $\K$-linear spaces $ V,W $ and  $G$,  a $\K$-bilinear map $f:V\times  W\to G$,  there is a unique $\K$- linear map $\phi_f:V\otimes  W\to G$ such that the following 
diagram commutes:\begin{equation*}
\begin{tikzcd}
V\times  W\arrow[r,"\otimes "]\arrow[rdd,"f"]&V\otimes  W \arrow[dd,"\phi_f"]\\
\\
&G
\end{tikzcd}
\end{equation*}

 {In a similar manner, the locality tensor product of a pair of subspaces of a pre-locality vector space also satisfy a universal property.}

\begin{thm}\label{thm:univpptyloctensprod-nolocalityrel} 
Given  {$V$ and $W$ two subspaces of a pre-locality vector space $(E,\top)$} over $\K$, $G$ a $\K$-linear space and $f_\top:V\times_\top W\to G$ a 
$\top_\times$-$\K$-bilinear map,  there is a unique $\K$-linear map $\phi_{f_\top}: V\otimes_\top W\to G$ such that the following diagram commutes:\begin{equation}\label{eq:otimesT}
\begin{tikzcd}
		V\times_\top W\arrow[r,"\otimes_\top"]\arrow[rdd,"f_\top"]&V\otimes_\top W \arrow[dd,"\phi_{f_\top}"]\\
		\\
		&G
		\end{tikzcd}
\end{equation}
\end{thm}
	
 \begin{proof}  {
  \textit{Existence of $\phi_{f_\top}$}. There exists a unique linear map $ {\overline{f_\top}}$ from $\K(V\times_\top W)$
to $G$ sending $(v,w)\in V\times_\top W$ to $f_\top(v,w)$. As  { $f_\top$ is $\top_\times$-bilinear, i.e.,} ${\overline{f_\top}}(I_{\rm bil}^{ \top_\times})=\{0_G\}$,
this map induces a linear map $\phi_{f_\top}$ from $V\otimes_\top W$ to $G$. For any $(v,w)\in V\times_\top W$,
\[\phi_{f_\top}(v\otimes w)=\phi_{f_\top}([(v,w)])= {\overline{f_\top}}(v,w)=f_\top(v,w).\]
Therefore, $\phi_{f_\top}\circ \otimes=f_\top$. \\
\textit{Uniqueness of $\phi_{f_\top}$}. As $V\times_\top W$ is a basis of $\K(V\times_\top W)$,
the elements $(v\otimes {w})_{(v,w)\in V\times_\top W}$ linearly generate $V\otimes_\top W$, which 
implies the uniqueness of $\phi_{f_\top}$.  }   
 \end{proof}

 {We now prove the equivalence between Theorem \ref{thm:univpptyloctensprod-nolocalityrel} and the universal property of the usual tensor product. For that purpose,} we recall that a subset $B$ of a vector space $V$  is a (Hamel or algebraic) basis if it satisfies the linear independence property i) for every finite subset $\{b_1, \cdots, b_n\}$ of $B$ and every $\alpha_1, \cdots, \alpha_n$ in $\K$ we have $\sum_{i=1}^n \alpha_i b_i=0\Longrightarrow \alpha_1=\cdots=\alpha_n=0$ and  
	 the spanning property ii)
	 every  vector $v$ in $V$ can be written as a finite linear combination $v=\sum_{k=1}^n\alpha_i b_i$  in which case there is an  	isomorphism of vector spaces $\K B\simeq V$.
	
	 We further recall a classical result of linear algebra which relies on Zorn's lemma. The latter reads: A partially ordered set  ${\mathcal P}$ whose chains all have  an upper bound in $({\mathcal P}, \preceq)$ contains at least one maximal element. Consequently,
	\begin{lem}\label{lem:compbasis} 
	Let $G$ be a generating subset of a vector space $V$.
	A linearly independent set  {$A\subset G$} can be extended to a  Hamel basis  $B\subset G$ of $V$.
	\end{lem}
	
  \begin{prop}\label{prop:locbiltobil} 
   Let $V$ and $W$ be subspaces of a pre-locality vector space $(E,\top)$ and $G$ a $\K$-linear  space. Any $\top_\times$-($\K$-)bilinear map $f:V\times_\top W\to G$  extends to a $\K$-bilinear map $g:V\times W\to G$ i.e.,  $g|_{V\times_\top W}=f$.
  \end{prop}
  
  \begin{proof}
  Consider the set $\Oo:=\{B\subset V\times_\top W|\, \otimes B\ {\rm is\ a\ linearly\ independent\ subset\ of}\, V\otimes W\}$,  {where $ \otimes B$ is the image of $B$ under the canonical map} $\otimes:V\times W\to V\otimes W$. We equip $\Oo$ with the partial inclusion order $ B_1\subset B_2$ and consider a chain $\CC$ of $\Oo$. We  {observe} that $\bigcup_{B\in\CC}B\in\Oo$ since $\otimes\left(\bigcup_{B\in\CC}B\right)=\bigcup_{B\in\CC}\otimes(B)$, and the union of nested linearly independent sets is a linearly independent set.
  
  Thus,  $\Oo$ satisfies the assumption of   Zorn's Lemma which ensures the existence of a maximal element $\B\in\Oo$ and correspondingly a linearly independent set $\otimes(\B)\subset V\otimes W$.   {Since $\otimes(\B)\subset\otimes(V\times W)$ and $\otimes(V\times W)$ generates $V\otimes W$}, by Lemma  \ref{lem:compbasis}, we can complete $\otimes(\B)$ to a basis $\overline{\otimes(\B)}\subset\otimes(V\times W)$ of $V\otimes W$.
  
   Since by construction,  $\overline{\otimes(\B)}\subset \otimes (V\times W)$, any element $y\in\overline{\otimes(\B)}\setminus \otimes(\B)$ can be written  $y=\otimes(x_y) $ for some  $x_y\in V\times W$. We claim that the set $\overline\B:=\{x_y: y\in\overline{\otimes(\B)}\setminus \otimes(\B)\}\cup \B$ fulfills the relation  \[\overline{\otimes (\B)}=\otimes(\overline \B).\] Indeed, if $x\in\overline\B$, either $x\in\B$, in which case $ {y:=\otimes(x)\in}\otimes(\B)\subset\overline{\otimes(\B)}$, or $y:=\otimes(x)\in\overline{\otimes(\B)}\setminus \otimes(\B)\subset\overline{\otimes(\B)}$ and therefore $\otimes(\overline\B)\subset \overline{\otimes(\B)}$. Conversely, for $y\in\overline{\otimes(\B)}$ either $y\in\otimes(\B)\subset\otimes(\overline\B)$ or $y\in\overline{\otimes(\B)}\setminus \otimes(\B)$ in which case $x_y\in\overline\B$ and therefore $\overline{\otimes(\B)}\subset\otimes(\overline\B).$

 Let $g:V\times W\to G$ be the unique $\K$-bilinear map defined on $\overline{\B}$  by 
  \[g(x,y):=\left\{\begin{array}{ll}
  f(x,y)&\rm{if}\ (x,y)\in\B\\
  0&\rm{if}\ (x,y)\notin\B
  \end{array}\right.\]
 It  remains to  show that $g|_{V\times_\top W}=f$.  {Given  $(p,q)\in V\times_\top W$, the maximality of $\B$ yields the existence of $(x_i,y_i)\in \B$ with $1\leq i\leq n$ for some $n\in\N$ such that $\sum_{i=1}^n\alpha_ix_i\otimes y_i=p\otimes q$, and thus $\sum_{i=1}^n\alpha_i(x_i, y_i)=(p, q)+ \omega$ which amounts to $\omega\in I_{\rm bil}^{\top_\times}$. Using the extension of $f$ and $g$ to $\bar f$ and $\bar g$, and the fact that $\bar f(\omega)=\bar g(\omega)=0$, this implies 
 \[f(p,q)=\sum_{i=1}^n\alpha_if(x_i,y_i)=\sum_{i=1}^n\alpha_i g(x_i,y_i)=g(p,q).\] } 
 
It follows that $g|_{V\times_\top W}=f$ as expected.
  \end{proof}
\begin{cor}\label{cor:univpptyloctensprod-nolocalityrel}   
	The universal property of the locality tensor product of two  {subspaces of a pre-locality vector space} is equivalent to the universal property of the usual (by usual we mean non locality) tensor product of two  subspaces of an ordinary vector space.
\end{cor}
\begin{proof}

 We know that  the statement holds for   $\top=E\times E$ as a result of the universal property of the usual tensor product.
  From there, we build the map $\phi_{f_\top}$ for any locality relation $\top$ on $E$.
		
 {We now prove that the usual universal property implies the one in the locality set up. Thanks to }		
Proposition \ref{prop:locbiltobil}, $f_\top$  {extends} to a bilinear map $g:V\times W\to G$ such that $g|_{V\times_\top W}=f_\top$. The universal property of the usual tensor product yields the existence of a unique linear map $\phi_g:V\otimes W\to G$ such that \begin{equation}\label{eq:conmunivpropltp}g=\phi_g\circ \otimes.\end{equation} Since $V\otimes_\top W\subset V\otimes W$, we can      restrict  $\phi_g$ to $V\otimes_\top W$ and set $\phi_{f_\top}:=\phi_g|_{V\otimes_\top W}$. Since $\otimes|_{V\times_\top W}=\otimes_\top$, we can further restrict  (\ref{eq:conmunivpropltp}) to $V\times_\top W$  and we have \[f_\top=\phi_{f_\top}\circ \otimes_\top\]
as expected. 
Now assuming that  the universal property holds for locality tensor products, we want to show that it holds for ordinary tensor products. So, let $V, W$ be two subspaces of a vector space $E$, $G$ another vector space. We equip $E$ with the trivial locality relation $\top= E\times E$, in which case   $V\times_\top W=(V\times W)\cap \top=  V\times W$ and the bilinear map $f:V\times W\to G$   can be interpreted as a $\top_\times$-bilinear map  $f_\top:=f:V\times_\top W\to G$ .  Applying Theorem \ref{thm:univpptyloctensprod-nolocalityrel} yields a linear map $\phi_f:=\phi_{f_\top}: V\otimes_\top W=V\otimes W\to G$ such that $f=\phi_f\circ \otimes$ since $\otimes_\top=\pi_\top\circ \iota_\top= \pi\circ \iota=\otimes$. The uniqueness of $\phi_f$ then follows from that of $\phi_{f_\top}$.
\end{proof}

\section{ {A locality relation on locality tensor products}} \label{sec:localityrelationontp}

We want  to define a locality relation $\top_\otimes$ in such a way that $\otimes_\top$   defined in (\ref{eq:otimesT}) is a locality map.  {For this purpose, we define a notion of final locality inspired by that of final topology. Let us first recall some definitions.}

\subsection{Free locality vector spaces}

 We recall a rather straightforward construction from \cite[Lemma 2.3 (ii)]{CGPZ1}.
A locality relation $\top$ on a set $X$ induces another one in the power set $\mathcal P(X)$, {which with some abuse of notation, we denote by the same symbol $\top$}:
	\[\forall U, V\in {\mathcal P}(X), \quad U {\top} V\Longleftrightarrow u\top v\,\, \forall (u, v)\in U\times V.\]
Conversely a locality relation $\top$ on a power set $\mathcal{P}(X)$ induces a locality relation, again written $\top$ on the set $X$ by restriction:
\begin{equation*}
 \forall(x,y)\in X^2,~x\top y\Longleftrightarrow \{x\}\top \{y\}.
\end{equation*}
 A locality relation $\top$  on a set $X$ further induces a locality relation (denoted with some abuse of notation by the same symbol $\top$)
on the vector space $\K\, X$ generated by $X$ given by the linear extension of the locality relation $\top$ on $X$. Explicitly, two elements $a$ and $b$ in $\K X$ are independent if the basis elements from $X$ appearing in  {$a$} are independent of  the basis elements arising in  $b$.   More precisely, the linear span $(\K \, X, \top)$  of a locality set $(X, \top)$  is a pre-locality vector space when equipped with the symmetric binary relation
\begin{equation} \label{eq:locrelonthefreespan} 
 \left(a:=\sum_{x\in X_a} \alpha_x\, x\right)\, \top  \, \left(b:=\sum_{x\in X_b} \beta_x\, x   \right)\Longleftrightarrow X_a \top_X X_b
\end{equation} 
where the coefficients $\alpha_x$ and $\beta_x$ are all non zero.

\begin{lem}\label{lem:linearspan} The linear span $(\K \, X, \top)$  of a locality set $(X, \top)$  is a locality (and hence also a pre-locality) vector space.
\end{lem}
\begin{proof} By definition we have 
	\[ \left(\sum_{i=1}^n\lambda_i\, u_i\right)\, \top \, \left(  \sum_{j=1}^{n^\prime}\lambda_j^\prime\,  u_j^\prime\right) \Longleftrightarrow  u_i\top u_j^\prime \,  \, 
	\forall (i, j)\in \{1, \cdots, n\}\times \{1, \cdots, n^\prime\}.  \]  
 In order to check Condition (\ref{eq:linearloc}), we take $v:=\sum_{k=1}^{m}\mu_k\,  v_k$, $u= \sum_{i=1}^n\lambda_i\, u_i$  {in $\K X$}, and $u^\prime= \sum_{j=1}^{n^\prime}\lambda_j^\prime\,  u_j^\prime$ {with $u \top v$ and $u'\top  {v}$}. For any $(\lambda, \lambda^\prime)\in \K^2$, 
	the element $\lambda u+\lambda^\prime u^\prime = \sum_{i=1}^n\lambda\,\lambda_i\, u_i + \sum_{j=1}^{n^\prime}\lambda^\prime\,\lambda_j^\prime\,  u_j^\prime$ is locality independent of $v$. Indeed, it follows from the definition of the linearly extended relation $\top$, and from the $ u_i$'s and $u_i'$'s being locality independent of $v$ for all $(i, j)\in \{1, \cdots, n\}\times \{1, \cdots, n^\prime\}$.
\end{proof}

\subsection{Quotient locality as a final locality relation} 
 We define a final locality in much the same way as a final topology. 
Recall that given two topologies $\tau_1$, $\tau_2$ on some set $X$,   $\tau_1$ is said to be {\bf coarser  {(weaker or smaller)}}  than $\tau_2$, or equivalently $\tau_2$ {\bf finer    {(stronger or larger)}}
 than $\tau_1$ if, and only if $\tau_1\subset \tau_2$.  
Also, given  a set $X$ and  $(X_i,\tau_i)_{i\in I}$  a family of topological spaces together with a family of maps  $f_i:X_i\to X$, the {\bf final topology  { (or strong,  colimit, coinduced, or inductive topology)}  
 $\bar\tau$} is the finest topology on $X$ such that all maps $f_i$ are continuous.
  With a small abuse of language, one says that the topology $\bar\tau$ is final with respect to the maps $f_i$.

  A typical example is the quotient topology on $X\times I$ where $I$ is a subset of a  locality set $(X, \top)$,  defined as the final topology for  the projection map $\pi: X\to X\setminus I$.

Let us now transpose this terminology to the locality setup. 
\begin{defn}
Let $\top_1$ and $\top_2$ be two locality relations over a set $A$. We say $\top_1$ is  {\bf coarser} than $\top_2$  or equivalently, that $\top_2$ is 
 {\bf finer} than $\top_1$ if, and only if $\top_1\subset \top_2$.
\end{defn}
\cy{
\begin{rk}
 In Theorem \ref{thm:equiv_categories}, a pre-locality relation $\top$ on a vector space is completed into a locality relation $\top^\ell$ on the same vector space. Then $\top^\ell$ can be defined as the coarsest locality relation that is finer than $\top$.
\end{rk}
}
The following example provides a justification of the terminology in our transposition from a topological to a locality context.
\begin{ex} {Let $X$ be a set and ${\mathcal P}(X)$ its powerset. Disjointness of sets:
		\[A\top B\Longleftrightarrow A\cap B=\emptyset\] defines a locality relation on any subset ${\mathcal O}$ of ${\mathcal P}(X)$. 
	If $(X, {\mathcal O})$ is a topological space with topology ${\mathcal O}  \subset {\mathcal P}(X)$, this disjointness relation gives rise to  
	another locality relation (which with some abuse of notation, we denote by the same notation) given by the separation of points:
	\[x\top y\Longleftrightarrow \exists U, V\in {\mathcal O}, \quad \left( U\, \top \, V\right)\, \wedge \,  \left(x\in U\wedge y\in V\right).\]
The finer (coarser) the  topology ${\mathcal O}$, the  {larger (smaller)} the graph $\{(x, y), \, x\top y\}$ of the locality relation, hence the terminology we have chosen. } 
 	\end{ex}

\begin{defn}\label{defn:final_relation}
Let $X$ be a set, $(X_i,\top_i)_{i\in I}$ a family of locality sets, and $f_i:X_i\to X$ a family of maps. The {\bf final locality relation $\ttop$} 
on $X$ is the  {coarsest  locality relation among the locality relations $\top$ on $X$ for which \[  f_i: (X_i,\top_i)\longrightarrow (X, \top), \quad i\in I \] are locality maps}. 

As before, with a slight abuse of language, we shall say that $\ttop$ is a {\bf  final locality relation} on $X$  for the maps $f_i$.
\end{defn}

\begin{prop} \label{prop:description_final_relation}
 Given a  surjective map $ {\phi:A\to B}$, the locality relation $\top$ on $A$ induces a locality relation $\ttop$ on $B$ defined by 
\[ b_1\ttop b_2\Longleftrightarrow (\exists (a_1,a_2)\in A\times A : \phi(a_i)=b_i\ \rm{and}\ a_1\top a_2),\]
which is the final locality relation for the map $\phi$.
\end{prop}
\begin{proof} 
 {It is clear from the definition  of $\ttop$,}  that $\phi:(A,\top)\longrightarrow(B,\ttop)$ is a locality map.

Let $\top_B$ be a locality relation on $B$ such that $ {\phi:(A,\top)\longrightarrow(B,\top_B)}$ is a locality map. For any $(b_1,b_2)\in B^2$ we have
\begin{align*}
 b_1\ttop b_2 ~&\Longrightarrow~\left(\exists(a_1,a_2)\in A^2|\phi(a_i)=b_i~\wedge~a_1\top a_2\right)\quad \text{for }i\in\{1,2\}\\
 &\Longrightarrow~\left(\exists(a_1,a_2)\in A^2|\phi(a_i)=b_i~\wedge~\phi(a_1)\top_B \phi(a_2)\right)\quad \text{since }\phi\text{ is a locality map}\\
 &\Longrightarrow~b_1\top_B b_2.
\end{align*}
Therefore $\ttop\subseteq\top_B$.
\end{proof}

\begin{ex}The map 
\begin{align*}
\phi&:\left\{\begin{array}{rcl}
\phi: \N&\longrightarrow & 2\, \N\\
		m&\longmapsto & 2m
\end{array}\right.
\end{align*}
is surjective. We equip $A:=\N$ with the  locality relation $m_1\top m_2\Longleftrightarrow \vert m_1-m_2\vert =3$. Then $n_1\overline\top n_2$ if and only if $\vert n_1-n_2\vert =6$.
\end{ex}

Applying Proposition \ref{prop:description_final_relation}   to  the canonical   projection map $\pi:V\to V/ W$  of a  {pre-locality}  vector space  $(V, \top)$ to its quotient $ V/ W$ by a linear subspace $W$, we equip  the quotient with the quotient locality relation.
\begin{defnprop}\label{defn:quotientlocality}
 For a subspace $W$ of a  pre-locality vector space $(V, \top)$, we call {\bf quotient locality} on 
 the quotient $V/ W$, the final locality relation
  \begin{equation}\label{eq:quotientloc}
 \left([u]\ttop[v]~\Longleftrightarrow \exists(u',v')\in[u]\times[v]:~u'\top v'\right)\quad\quad \forall ([u],[v])\in (V/ W)^2 
 \end{equation}
  for the canonical projection map    $\pi:V\to V/ W$. This way, the  pre-locality space $(V, \top)$ gives rise to a  {pre-locality} vector quotient space $(V/ W, \ttop)$ and the projection map $\pi:(V, \top)\to (V/ W, \ttop)$ is a morphism of pre-locality vector spaces.
\end{defnprop}
\begin{proof}
     The facts that $(V/W,\ttop)$ is a pre-locality space and that $\pi:(V, \top)\to (V/ W, \ttop)$ is a morphism of pre-locality spaces hold by definition of $\ttop$, since it is the coarsest locality relation such that $\pi$ is a locality map.
    \end{proof}

 The following simple examples  illustrate this last concept.
\begin{ex} Consider the  {pre-locality} vector space $(\R^3,\top)$ where $\top$ is the orthogonality relation, namely $v\top w\Leftrightarrow v\perp w$. Let $W=\K e_1\subset \R^3$ be the span of $e_1$ where $\{e_i\}_{i=1}^3$ is the canonical basis of $\R^3$. The quotient locality on $\R^3/ W$ is $\overline\top=(\R^3/ W)\times(\R^3/ W)$ since for any pair $([q_2e_e+q_3e_3], [k_2e_2+k_3e_3])\in (\R^3/ W)^2$ there are  {scalars} $q_1$ and $k_1$ in $\K$ such that $(q_1e_1+q_2e_2+q_3e_3)\perp(k_1e_1 {+}k_2e_2 {+}k_3e_3)$, so that $[q_2e_e+q_3e_3]\overline\top[k_2e_2+k_3e_3]$.
\end{ex}

\begin{ex}
Consider the  pre-locality vector space $(V,\top)$ where $V=\R^4$ and $\top=\R^4\times \{0\}\cup \{0\}\times\R^4\cup(\lan\{e_1,e_3\}\ran\times\lan e_2+e_4\ran)\cup(\lan e_2+e_4\ran\times\lan\{e_1,e_3\}\ran)$. For $W=\K(e_4)$,   the quotient locality on $V/ W$ is given by $\overline\top=\left(V/ W\times\{[0]\}\right)\cup\left(\{[0]\}\times V/ W\right)\cup\left( \lan[e_1+e_3]\ran\times \lan[e_2]\ran\right)\cup \left(\lan[e_2]\ran\times\lan[e_1+e_3]\ran\right)$.
\end{ex}

\subsection{An enhanced universal property on tensor products of pre-locality spaces}\label{subsec:LTPofPL}

	 Let  $(E,\top)$ be a pre-locality vector space, and let us consider the locality cartesian product $V\times_\top W= (V\times W)\cap \top$ of two linear subspaces  $V$ and $W$ of $E$. Both subspaces inherit from   $\top$  a locality relation $\top_V=\top\cap (V\times V)$ and similarly for $W$, turning them into  pre-locality subspaces. We want to equip their locality tensor product $V\otimes_\top W$ defined in (\ref{eq:VotimestopW})  with a locality relation. 
\begin{defn} \label{defn:top_otimes}
 The locality relation $\top_\otimes$ on $V \otimes_\top W$ is defined as the quotient relation  (see  Definition \ref{defn:quotientlocality}) for the quotient map $\pi:\K\left(V \times_{\top}W\right)\to V \otimes_\top W$, where the locality relation on $\K\left(V \times_{\top}W\right)$ is \cy{the relation $\top_{V\times_\top W}$ defined in Equation} \eqref{eq:locreltoptimes}.
\end{defn}
 {With this locality relation on the locality tensor product, the map $\otimes_\top$ is a locality map:}
\begin{prop}

 The map $\otimes_\top:(V\times_\top W,\top_\times)\to (V\otimes_\top W,\top_\otimes)$ is a locality $\top_\times$-bilinear map.
\end{prop}

\begin{proof}
It was shown in Proposition \ref{prop:otimesastxbilinear} that $\otimes_\top$ is $\top_\times$-bilinear. We therefore only need to show that it is a locality map. Recall that $\otimes_\top=\pi_\top\circ\iota_\top$, where $\iota_\top:V\times_\top W\to \K(V\times_\top W)$ is the canonical inclusion map. The latter  is a locality map since the locality $\top_\times$ on $\K(V\times_\top W)$ is a linear extension of the locality relation in $V\times_\top W$. The map $\pi_\top: \K(V\times_\top W)\to V\otimes_\top W$ is also a locality map by construction of the locality relation $\top_\otimes$. The statement then follows from the fact that the composition of locality maps is again a locality map.
\end{proof}

Locality tensor products equipped with the  locality relation $\top_\otimes$ are pre-locality vector spaces and the 
universal property of the   tensor product proved in Theorem \ref{thm:univpptyloctensprod-nolocalityrel}  can be enhanced as follows.

\begin{thm}[Universal property of  the locality tensor product on pre-locality vector spaces]  \label{thm:univpropltp-preloc}
	Let
$(G,\top_G)$ be a pre-locality vector space,  and $f:(V\times_\top W,\top_{V\times_\top W})\to (G,\top_G)$ a locality 
	$\top_\times$-bilinear map. 
	There is a unique locality linear map $\phi:V\otimes_\top W\to G$ such that the following diagram commutes.
	\begin{equation*}
	\begin{tikzcd}
	(V\times_\top W,\top_{V\times_\top W})\arrow[r,"\otimes_\top"]\arrow[rdd,"f"]&(V\otimes_\top W,\top_\otimes) \arrow[dd,"\phi"]\\
	\\
	&(G,\top_G)
	\end{tikzcd}
	\end{equation*}
\end{thm}
\begin{proof}
	Theorem \ref{thm:univpptyloctensprod-nolocalityrel} yields the existence and uniqueness of the linear map $\phi$.
	We are only left to show that $\phi$ is a locality map.  Recall that two equivalence classes $[a]$ and $[b]$ in 
	$V\otimes_{\top}W$ verify $[a]\top_{\otimes}[b]$ if there are $\sum_{i=1}^n\alpha_i(x_i,y_i)\in[a]$ and 
	$\sum_{j=1}^m\beta_j(u_j,v_j)\in[b]$ such that every possible pair taken from the set $\{x_i,y_i,u_j,v_j\}$ lies in 
	$V\times_\top W$ for every $1\leq i\leq n$ and every $1\leq j\leq m$. Since $f$ is locality bilinear, then 
	$f(\sum_{i=1}^n\alpha_i(x_i,y_i))\top_Gf(\sum_{j=1}^m\beta_j(u_j,v_j))$ which amounts  to $\phi([a])\top_G\phi([b])$. Therefore 
	$\phi$ is locality as expected.
\end{proof}

\begin{rk}
	The universal property of  the locality tensor product on pre-local vector spaces clearly implies the  universal property of  the  tensor product on ordinary vector spaces; take $\top_V= V\times V, \top_W= W\times W, \top_G=G\times G$ in the above theorem.  
\end{rk}

\section{Higher locality tensor products} \label{section:three}

 We generalise Definition  \ref{defn:loctensprod1} and build  the locality tensor product of  $n$ pre-locality vector spaces, which we equip with a locality relation.  In this paragraph, 
  	$(E,\top)$ is a pre-locality vector space over $\K$, and $V_1, \cdots, V_n$ are linear subspaces of $E$.
We build  the quotient  of $\K( V_1\times \cdots \times V_n) $   as the subspace $I_{{\rm mult}} {(V_1, \cdots, V_n)}$  generated by all elements of the form
\begin{equation}\label{eq:multlinform1}
(x_1,...,x_{i-1},a_i+b_i,x_{i+1},...,x_n)-(x_1,...,x_{i-1},a_i,x_{i+1},...,x_n)-(x_1,...,x_{i-1},b_i,x_{i+1},...,x_n)
\end{equation}
\begin{equation}\label{eq:multlinform2}
(x_1,...,kx_i,...,x_n)-k(x_1,...,x_i,...,x_n)
\end{equation}
for every $i\in[n]$, $k\in\K$ and $a_i,b_i,x_i\in V_i$. If $V_1=\cdots =V_n=V$, we write  $I_{{\rm mult},n}(V) $.
\begin{defn}\label{defn:loctensalgebra} 
	We   define  
	\begin{itemize}
		\item\cite[\S 3.1]{CGPZ1} the {\bf  locality cartesian product} 
		\begin{equation}\label{eq:Vtopn} 
		 V_1\times_\top \cdots \times_\top V_n:=\{(x_1,...,x_n)\in  V_1\times  \cdots \times  V_n|\forall(i,j\in[n]): i\neq j\Rightarrow (x_i,x_j)\in V_i\times_\top V_j:= (V_i\times V_j)\cap \top\};
		 \end{equation}
If  $V_i=V$ for any $i\in [n]$, we set $V^{\times_{\top}^n}:=  V_1\times_\top \cdots \times_\top V_n=V\times_\top\dots\times_\top V$.
		In particular $V^{\times_{\top}2}=\top$, $V^{\times_{\top}1}=V$ and we set by convention  $V^{\times_{\top}0}=\K$ and $V^{\times_\top\infty}:=\bigcup_{n\geq0}V^{\times_\top n}.$  {Note that $V_1\times_\top \cdots \times_\top V_n\ni (0, 0, \cdots, 0)$ where $0$ is the zero element in $E$, since $\top\ni (0, 0)$ by definition of a pre-locality vector space.}
		\item \cite[\S 4.1]{CGPZ1} the {\bf     locality tensor product} \begin{equation} \label{eq:Votimesn}
		 V_1\otimes_{\top} \cdots \otimes_{\top} V_n~:=~\sfrac{\K(V_1\times_\top \cdots \times_\top V_n)}{(I_{{\rm mult}} {(V_1, \cdots, V_n)}\cap\K(V_1\times_\top \cdots \times_\top V_n))}
		 \end{equation}
 If $V_i=V$ for any $i\in [n]$, we set $V^{\otimes_{\top}^n }:=  V_1\otimes_\top \cdots \otimes_\top V_n$.
		
\item
We endow the locality tensor product $V_1\otimes_\top\cdots\otimes_\top V_n$ with the locality relation $\top_{\otimes n}$ defined as the quotient locality (see Definition \ref{defn:quotientlocality}) for the quotient map $\pi_n:\K(V_1\times_\top \cdots \times_\top V_n)\longrightarrow V_1\otimes_{\top} \cdots \otimes_{\top} V_n$.
		\end{itemize}
	\end{defn} 
\begin{rk}
For $n=2$ we recover Definition \ref{defn:loctensprod1}.
\end{rk}
	The size of $V^{\otimes_{\top}^n}\subset V^{\otimes n}$ depends on   the locality relation, namely on how many elements   mutually independent  it allows.
	\begin{ex} Consider the pre-locality vector space $(\R^2,\perp)$, then $V^{\times_{\top}n}=\{0\}$ for all $n\geq3$ since there are no three pairwise orthogonal non zero elements in $\R^2$.
	\end{ex}
	This contrasts with  the following example.
	
%
%
\begin{ex}
  {We equip the vector space $V:=\R^{\infty}$ with the canonical orthogonality relation $\perp$ as locality relation: $u\perp v:\Longleftrightarrow\langle u,v\rangle=0$ (see \cite[Subsection 2.2.1]{CGPZ1} for details). One easily checks that $(\R^\infty,\perp)$ is a locality vector space. In this case, there is no integer $n$ in $\N$ such that $V^{\otimes_{\perp}n}=\{0\}$. }
\end{ex}

\begin{defn} \label{defn:top_otimes_m_comma_n}
\begin{itemize}
 \item 		We define the relation $\top_{\times m,n}\subset\left( V_1\times_\top\cdots\times_\top  V_m\right) \times \left(V_{m+1}\times_\top\cdots\times_\top  V_{m+n}\right)$ as follows 
		\begin{equation} \label{eq:locreltimesmn}
		(x,y)\in\top_{\times m,n}\Longleftrightarrow \forall (i,j)\in[m]\times[n], (x_i,y_{ {m+}j})\in\top,
		\end{equation}  extend it linearly to $\K\left(V_1\times_\top\cdots\times_\top  V_m\right)\times \K\left(V_{m+1}\times_\top\cdots\times_\top  V_{m+n}\right)$ as follows
		\begin{align*}
&\left( \sum_{k\in K}\alpha_k\, x_k,\sum_{l\in L}\beta_l\,y_l\right)\in   \K\left(V_1\times_\top\cdots\times_\top  V_m\right)\times_{ {\top_{\times m,n}}} \K\left(V_{m+1}\times_\top\cdots \times_\top V_{m+n}\right)\\
&\Longleftrightarrow \forall (k,l)\in K\times L, (x_k,y_l)\in {\top },
				\end{align*}
	for any  {$K$ and $L$ finite sets,} for every $k\in K$ and every $l\in L$, $(\alpha_k,\beta_l)\in(\K\setminus\{0_\K\})^2$,  $x_k\in V_k$ and $y_l\in V_l$.
	
\item  Finally,
    	we define the relation $\top_{\otimes m,n}\subset\left( V_1\otimes_\top\cdots\otimes_\top  V_m\right) \times \left(V_{m+1}\otimes_\top\cdots\otimes_\top  V_{m+n}\right)$   induced from $\top_{\times m,n}$ as follows:
	\begin{equation} \label{eq:locrelotimesmn}
 (x,y)\in\top_{\otimes m,n}\Leftrightarrow \left(\exists x'\in x\right)\wedge\left(\exists y'\in y\right) : \,(x',y')\in\top_{\times m,n},
 \end{equation}
 where $x\in V_1\otimes_\top\cdots \otimes_\top V_m$ and $y\in V_{m+1}\otimes_\top\cdots \otimes_\top V_{m+n}$ (notice that according to Definition \ref{defn:loctensalgebra}, $x$ and $y$ are equivalence classes so that the notation $x'\in x$ makes sense). 
\end{itemize}
		\end{defn}
		\begin{rk}
		     $\top_{\times m,n}$ and $\top_{\otimes m,n}$ are not locality relations since they are not in general subset of a set of the form $S\times S$. Instead, $\top_{\times m,n}$ can be seen as a relation between $ {\K(}V_1\times_\top\cdots\times_\top V_m {)}$ and $ {\K(}V_{m+1}\times_\top\cdots\times_\top V_{m+n} {)}$; and $\top_{\otimes m,n}$ as a relation between $V_1\otimes_\top\cdots\otimes_\top V_m$ and $V_{m+1}\otimes_\top\cdots\otimes_\top V_{m+n}$. In contrast, $\top_{\otimes n}$ is a locality relation on $V_1\otimes_\top\cdots\otimes_\top V_n$.
		    \end{rk}

\begin{lem} \label{lem:iso_loc_cart_prod}
 The map	\begin{align} \label{eq:Psiiso}
	\Psi_{m,n}&:\left\{\begin{array}{rcl} \K\left(V_1\times_\top\cdots\times_\top  V_m\right)\times_\top \K\left(V_{m+1}\times_\top\cdots\times_\top  V_{m+n}\right) &\longrightarrow & \K\left(V_1\times_\top\cdots\times_\top  V_{m+n}\right)\\
	\left((x_1, \cdots, x_m),(y_{1}, \cdots, y_{n})\right)&\longmapsto & (x_1, \cdots,x_m, y_1, \cdots, y_n)
	\end{array}\right.
	\end{align} linearly  extends 
	to a  {surjective} morphism of {pre-locality} vector spaces:
\begin{equation}\label{eq:locisocartesianproduct}  {\Psi_{m,n}:}\K\Big(	 \K\left(V_1\times_\top\cdots\times_\top  V_m\Big)\times_\top \K\left(V_{m+1}\times_\top\cdots\times_\top  V_{m+n}\right)\right)\to  \K\left(V_1\times_\top\cdots\times_\top  V_{m+n}\right).\end{equation} \end{lem}
\begin{rk}{Note that $\Psi_{m, n}$ is not expected to be an isomorphism. Let us take $m=n=1$ to simplify. 
A basis of $\K\Big(	 \K (V_1 )\times_\top \K\left(V_2\right)\Big)$ is given by elements
$(k_1v_1+\ldots+k_p v_p,l_1w_1+\ldots+l_qw_q)$, with  $p,q$ in $\N$, $(v_i,w_j)$ in $V_1\times_\top V_2$   and $k_i,l_j$ in $\K$ non zero for all indices  $i,j$. A basis of $\K(V_1\times_\top V_2)$ is given by pairs $(v_1,v_2)\in V_1\times_\top V_2$.  Since
\[\psi_{1,1}((k_1v_1+\ldots+k_p v_p,l_1w_1+\ldots+l_qw_q))
=\sum_{i,j}k_i l_j (v_i,w_j)=\psi_{1,1}\left(\sum_{i,j} k_il_j (v_i,w_j)\right),\]
$\psi_{1,1}$ is not injective. 
}
\end{rk}
\begin{proof}
	The fact that $\Psi_{m,n}$ extends to a    {surjective} morphism of vector spaces is a classical result of linear algebra. To show that $\Psi_{m,n}$ is a locality morphism we just need to check that  {it is a} locality map, which is an easy consequence of the following equivalence:
	for any $x:=(x_1,\dots,x_m)\in V_1\times_\top\cdots\times_\top  V_m$ and $y:=(y_1,\dots,y_n)\in V_{m+1}\times_\top\cdots\times_\top  V_{m+n}$, we have 
	\begin{align*}
	(x,y)\in\top_{\times m, n}  &\Longleftrightarrow (x_i,y_j) \in\top\,  \quad\forall (i,j)\in[m]\times[n]\\
	&\Longleftrightarrow \Psi(x, y)= \left(x_1,\dots,x_m,y_1, \cdots, y_n\right)\in \K\left(V_1\times_\top\cdots\times_\top  V_{m+n}\right)
	\end{align*}
	One simply needs to take two independent pairs $(x,y)$ and $(x',y')$ and work with these equivalences. We omit here the  detailed proof which is straightforward but rather combersome to write.
\end{proof} 

The locality  morphism (\ref{eq:locisocartesianproduct}) induces a locality   morphism between locality tensor products.
 
  \begin{thm}\label{thm:asso_pre_loc_tensor_prod}
	 	For any subspaces $V_1, \cdots, V_{m+n}$ of the pre-locality space $(E, \top)$, we set    
	  \begin{align*}  
        &\left(V_1\otimes_\top\cdots \otimes_\top V_m\right)\otimes_\top \left(V_{m+1}\otimes_\top\cdots \otimes_\top V_{m+n}\right)\\
	 		:&=\K\left(\left(V_1\otimes_\top\cdots \otimes_\top  V_m\right)\times_{\top_{\otimes m,n}} \left(V_{m+1}\otimes_\top\cdots \otimes_\top V_{m+n}\right)\right)\\
	 		&/ \left( {I_{\rm bil}} \cap \K\left(\left(V_{1}\otimes_\top\cdots \otimes_\top V_m\right)\times_{\top_{\otimes m, n} } \left(V_{m+1}\otimes_\top\cdots \otimes_\top V_{m+n}\right)\right)\right)
        \end{align*} 
        and  Definition \ref{defn:top_otimes} yields a pre-locality relation
        on this quotient.
        
	 	\sy{  There is an isomorphism of pre-locality vector spaces
	 	\begin{align}
 \label{eq:conjtensor}
	& \Phi_{m, n}:	\left( \left(V_1\otimes_\top\cdots \otimes_\top V_m\right)\otimes_\top \left(V_{m+1}\otimes_\top\cdots \otimes_\top V_{m+n}\right), \top_\otimes^{m,n}\right) \\
\nonumber &	 \overset{\sim}{\longrightarrow} 
	 	\left(V_1\otimes_\top\cdots \otimes_\top V_m\otimes_\top V_{m+1}\otimes_\top\cdots \otimes_\top V_{m+n}, \top_{\otimes (m+n)}\right).	 	
\end{align}	 	}
	 \end{thm}   
	 \begin{rk}
	      Notice that the locality relation $\top_{\otimes(m+n)}$ is the locality relation $\top_{\otimes N}$ of the third item  of Definition \ref{defn:loctensalgebra} with $N=m+n$, not to be confused with the   relation $\top_{\otimes m,n}$ of Definition \ref{defn:top_otimes_m_comma_n}, nor with the newly introduced locality $\top_\otimes^{m,n}$.
	     \end{rk}
 \begin{proof}
 
We build the isomorphism $\Phi_{{m,n}}$
from  the morphism $\Psi_{{m,n}}$ defined in (\ref{eq:Psiiso}).

\begin{itemize}

	  		\item  
An element $[y]$ of $ \left(V_1\otimes_\top\cdots \otimes_\top V_m\right)\otimes_\top \left(V_{m+1}\otimes_\top\cdots \otimes_\top V_{m+n}\right) $ reads 
\begin{equation*}
 [y]=\left[\sum_{i\in I}\alpha_i\left( \sum_{j\in J}\beta_j^iv_{j,1}^i\otimes\dots\otimes v_{j,m}^i,\sum_{k\in K}\gamma_k^iv_{k,1}^i\otimes\dots\otimes v_{k,n}^i\right)\right],
\end{equation*}
for some vectors $v_{j,r}^i\in V_r$ and $ v_{k,s}^i\in V_s$ and scalars $\alpha_i,\beta_j^i,\gamma_k^i$ in $ \K$ with $I,\,J,\, K$ three finite sets, such that 
 $\sum_{j\in J}\sum_{k\in K}\beta_j^i\gamma_k^i\left( v_{j,1}^i, \dots, v_{j,m}^i,v_{k,1}^i,\dots, v_{k,n}^i\right)\in \left(V_1\times_\top \cdots \times_\top V_m\right)\times_{\top n,m}\left(V_{m+1}\times_\top \cdots \times_\top V_{m+n}\right)\simeq \left(V_1\times_\top \cdots \times_\top V_{m+n}\right)$. 
\item	
	The linear map $\Phi_{{m,n}}:\left(V_1\otimes_\top \cdots \otimes_\top V_m\right)\otimes_{\top}\left(V_{m+1}\otimes_\top \cdots \otimes_\top V_{m+n}\right)\longrightarrow  \left(V_1\otimes_\top \cdots \otimes_\top V_{m+n}\right) $ is defined by the following action on 
	an element  $[y]$ of $ \left(V_1\otimes_\top \cdots \otimes_\top V_m\right)\otimes_{\top n,m}\left(V_1\otimes_\top \cdots \otimes_\top V_n\right)$:
	\begin{equation*}
	 \Phi_{{m,n}}([y])=\sum_{i\in I}\sum_{j\in J}\sum_{k\in K}\alpha_i\beta_j^i\gamma_k^i\left( v_{j,1}^i\otimes \dots\otimes v_{j,n}^i\otimes v_{k,1}^i\otimes\dots\otimes v_{k,m}^i\right).
	\end{equation*}
	Notice that $\left( v_{j,1}^i, \cdots, v_{j,n}^i,v_{k,1}^i,\cdots, v_{k,m}^i\right)$ is an element of $V_1\times_\top \cdots \times_\top V_{m+n}$ by definition of 
	$\left(V_1\otimes_\top \cdots \otimes_\top V_m\right)\otimes_{\top}\left(V_{m+1}\otimes_\top \cdots \otimes_\top V_{m+n}\right)$, and that the difference of two 
	representatives of $[y]$ lies in $I_{\rm bil}$ whose image by $\Phi$ lies  in ${ I_{\rm mult} (V_1, \cdots, V_{m+n})}$. Thus 
	$\Phi_{m, n}$ is well-defined. 
	
\item The injectivity of $\Phi_{m,n}$  {  follows from the commutativity of Diagram \ref{diag:comm_tens_prod} in which the vertical arrows are quotient maps.}

\begin{figure}[h!]
	\begin{center}

\begin{tikzcd}
\K\Big(\K\left(V_1\times_\top \cdots \times_\top V_m\right)\times_\cy{\top_{ \times m, n}}\K\left(V_{m+1}\times_\top \cdots \times_\top V_{m+n}\right)\Big)\arrow[d,"{\pi_{m}\times\pi_n}"]   \arrow[r,"\Psi_{m,n}"] &\K(V_{1}\times_\top \cdots  \times_\top V_{m+n}) \arrow[dd,"\pi_{m+n}"]\\ 
\K\Big(\left(V_1\otimes_\top \cdots \otimes_\top V_m\right)\times_\cy{\top_{\otimes{m, n}  }}\left(V_{m+1}\otimes_\top \cdots \otimes_\top V_{m+n}\right)\Big) \arrow[d] &\\
\left(V_1\otimes_\top \cdots \otimes_\top V_m\right)\otimes_{\top}\left(V_{m+1}\otimes_\top \cdots \otimes_\top V_{m+n}\right) \arrow[r,"\Phi_{m,n}"] &  V_1\otimes_\top \cdots  \otimes_\top V_{m+n} 
\end{tikzcd}
\caption{Tensor products of pre-locality spaces}
\label{diag:comm_tens_prod}
\end{center}

\end{figure}
Assume  $\Phi_{m, n}([y])=0$ for some $[y]\in \left(V_1\otimes_\top \cdots \otimes_\top V_m\right)\otimes_{\top}\left(V_{m+1}\otimes_\top \cdots \otimes_\top V_{m+n}\right) $. Then the preimage of $\Phi_{m, n}([y])$ under the quotient map $\pi_{m+n} $ in 
	$V_1\times_\top \cdots  \times_\top V_{m+n}$ lies in $ {I_{\rm mult} (V_1, \cdots, V_{m+n})}$. This implies that its preimage under $\pi_{m+n}\circ\Psi_{m, n}$ lies in 
	$ \K({I_{\rm mult} (V_1, \cdots, V_{m })\times_{\top m, n} I_{\rm mult} (V_{m+1}, \cdots, V_{m+n})})$ or in $I_{\rm bil}(\K\left(V_1\times_\top \cdots \times_\top V_m\right),\K\left(V_{m+1}\times_\top \cdots \times_\top V_{m+n}\right))$ by construction of $\Psi_{m, n}$ in Lemma \ref{lem:iso_loc_cart_prod}.\\
	 This element of $ \K({I_{\rm mult} (V_1, \cdots, V_{m })\times_{\top m, n} I_{\rm mult} (V_{m+1}, \cdots, V_{m+n})})$ combined with \\$I_{\rm bil}(\K\left(V_1\times_\top \cdots \times_\top V_m\right),\K\left(V_{m+1}\times_\top \cdots \times_\top V_{m+n}\right))$ is the preimage of $[y]$ under the two projections on the left column of Figure \ref{diag:comm_tens_prod}, and therefore $[y]=0$.

	\item The surjectivity of $\Phi_{m,n}$  \sy{  follows from that of $\Psi_{m, n}$ combined with} the  commutativity of Diagram \ref{diag:comm_tens_prod}.
Indeed, any element $y\in V_1\otimes_\top \cdots  \otimes_\top V_{m+n} $ has a preimage $\tilde y\in \K(\K\left(V_1\times_\top \cdots \times_\top V_m\right)\times_{\top n,m}\K\left(V_{m+1}\times_\top \cdots \times_\top V_{m+n}\right)) $      under $\pi_{m+n}\circ\Psi_{m, n}$ since $\Psi_{m, n}$ and $\pi_{m, n}$ are  {surjections}. Taking the image of $\tilde y$ under the two projections of the left column of Figure \ref{diag:comm_tens_prod} we obtain a preimage of $y$ in $\left(V_1\otimes_\top \cdots \otimes_\top V_m\right)\otimes_{\top}\left(V_{m+1}\otimes_\top \cdots \otimes_\top V_{m+n}\right) $.
\qedhere
\end{itemize} 
\end{proof}

\section{Tensor and universal pre-locality algebras}\label{sec:Tensoranduniversalpla}

 We now turn to the pre-locality counterpart of the notions of tensor algebras and universal envelopping algebras, and discuss  their universal properties. Both are pre-locality algebras, a notion we first introduce.

\subsection{Pre-locality algebras}

We slightly adapt various notions  introduced in \cite{CGPZ1} to the pre-locality framework.
\begin{defn}\label{defn:prelocalg}\begin{itemize}
\item  A {\bf non-unital pre-locality algebra} is a triple $(A,\top,m)$, where $(A,\top)$ is a pre-locality vector space, equipped with a partially defined product, namely a $\top_\times-$bilinear map 
$m:A\times_{\top}A\to A$, 
 {which is}  associative in the following sense  
\[m(m(x, y), z)=m(x, m(y, z))\quad \forall (x,y,z)\in A^{\times_{\top}3},\]
whenever $m(m(x, y), z)$ and $m(x, m(y, z))$ are defined.
	
\item \cite[Definition 3.16 (ii)]{CGPZ1} We call {\bf non-unital locality algebra} a non-unital pre-locality algebra $(A,\top,m)$,  whose underlying vector space is a locality vector space (so that in particular $U^\top$ is a vector space for any $U\subset A$), and whose partial  product $m:A {\times_{\top}}A\to A$ is  compatible with the locality relation in the following sense
\begin{equation}\label{eq:mlocalgm}m\left(U^{\top} {  \times_\top}U^{\top}\right)\subset U^{\top}\quad \forall   U\subset A, \end{equation}
  where we have set $V\times_\top V:= (V\times V)\cap \top$ for any subset $V\subset A$.
\item A non-unital locality (resp. pre-locality) algebra is called {\bf unital} (or simply (pre-)locality algebra) if there is a map $u:\K\to A$  such that 
		 $u(\K)\subseteq A^\top$, which makes the following diagram  \ref{diag:unit} commute
\begin{center}
 \begin{tikzpicture}[->,>=stealth',shorten >=1pt,auto,node distance=2cm,thick] \label{diag:unit}

    \node (1) {$\mathbb{K}\otimes A$};
    \node (2) [right of=1] {$A\otimes_\top A$};
    \node (3) [right of=2] {$A\otimes\mathbb{K}$};
    \node (4) [below of=2] {$A$};

    \path
      (1) edge node [above] {$u\otimes Id$} (2)
      (3) edge node [above] {$Id\otimes u$} (2)
      (2) edge node [right] {$\mu$} (4)
      (4) edge [<->] node [right] {$\simeq$} (1)
      (4) edge [<->] node [right] {$\simeq$} (3);
  \end{tikzpicture}
  \captionof{figure}{The unit map}
\end{center}
\item A  (resp. pre-)locality subspace $I$ of a (resp. pre-)locality  algebra $\left(A,\top ,m  \right) $ is called a  left, resp. right {\bf (pre-)locality ideal} of $A$, if 
\begin{equation} \label{eq:locideal} 
 m\left(  I\times I^\top\right)\subset I; \, {\rm resp.}\,     m\left(  I^\top \times I \right)\subset I.
\end{equation} 
If it is both a left and a right ideal, we call it a (resp. pre-) locality ideal.

 If  $(A, \top, m)$ is a locality algebra, then  $I^\top$ is a linear subspace of $A$, and we call $I$ a  locality ideal. 
\item \cite[Definition 3.16 (ii)]{CGPZ1} Given two   {(resp. pre-)locality algebras} $(A_i,\top_i,m_i,u_i), i=1,2 $, a  locality  linear map $f: A_1\to A_2$  is called a {\bf (resp. pre-) locality algebra morphism} if \begin{equation}\label{eq:localgmorp}f\circ m_1\vert_{\top_1}= m_2\circ(f\times f)\vert_{{\top}_1}.\end{equation}
 \item We call $(A_1, \top_{A_1}, m_1, u_1)$ a (pre-)locality subalgebra of $(A_2, \top_{A_2}, m_2, u_2)$ if {there is an} inclusion map $\iota: A_1\hookrightarrow A_2$ {which is also} a (pre-)locality algebra morphism.
\end{itemize}
\end{defn}
\begin{rk}
     Note that this definition is more general than the definition of sub-locality algebra given in \cite{CGPZ1}. We will need this degree of generality for the locality version of the Milnor-Moore theorem. A case of particular importance is when $A_1=A_2$ and $\top_2\subseteq\top_1$.
    \end{rk}
{\begin{ex}
Let $(A,\top, m)$ be a locality algebra, the polar set $U^\top$ of  any  {non-empty} subset $U$ of $A$ gives rise to a non-unital locality subalgebra $(U^\top,\top_{U^\top}, m)$ of $(A, \top, m)$. Here $\top_{U^\top}=\top\cap(U^\top\times U^\top)$. 
\end{ex}}

 {\begin{rk}\label{rk:partial-product-is-locality}
Notice that for a non-unital locality algebra $(A,\top,m)$, Condition \eqref{eq:mlocalgm} is equivalent to  the product $m:(A\times_\top A,\top_{A\times_\top A})\to(A,\top)$ being a locality map, and thus a locality $\top_\times$-bilinear map. 
\end{rk}}

\begin{lem}\label{lem:Kerlocideal}
  Let $f: A_1\longrightarrow A_2$ be a locality linear map between two  (resp. pre-)locality algebras $(A_i,m_i, \top_i), i\in \{1, 2\}$. Its kernel is a 
   (resp. pre-)locality ideal of $A_1$ and its range is a   (resp. pre-)locality subalgebra of $A_2$. 
\end{lem}
\begin{proof} We prove that the kernel $ {\rm Ker}(f) $ is a   {(resp. pre-)}locality ideal.  Take
$ a\in {\rm Ker}(f) $ and $ b\in  {\rm Ker}(f)^{\top_1}$, then $f(m_1(a\otimes b))=m_2(f(a)\otimes f(b))=m_2(0\otimes f(b))=0,$ hence
	$m_1\left({\rm Ker}(f)\times {\rm Ker}(f)^{\top_1}\right)\subset {\rm Ker}(f)$. Similarly we check that $m_1\left({\rm Ker}(f)^{\top_1}\times {\rm Ker}f) \right)\subset {\rm Ker}(f)$.
	
	 If $A_1$ is a locality algebra, then ${\rm Ker}(f)^\top$ is a linear subspace of $A_1$ and ${\rm Ker}(f)$ a locality ideal in $A_1$.
	
	We prove that the range  ${\rm Im}(f) $ is a   (resp. pre-)locality algebra. Given 
	$ (f(a), f(b))\in ({\rm Im}(f)\times  {\rm Im}(f))\cap  {{\top}_2} $, by (\ref{eq:localgmorp}) we have
	 $m_2(f(a)\otimes f(b))=f\circ m_1(a\otimes b)\in  {\rm Im}(f)$.
	 
	  If $A_2$ is a locality algebra, then ${\rm Im}(f)^\top$ is a linear subspace of $A_2$.  {Moreover, setting $\top_{{\rm Im}(f)}:=\top_2\cap({\rm Im}(f)\times {\rm Im}(f))$, and given $U\subset{\rm Im}(f)$, 
	 \begin{align*}m(U^{\top_{{\rm Im}(f)}}\times_{\top_{{\rm Im}(f)}} U^{\top_{{\rm Im}(f)}})&=m\left((U^{\top_2}\cap{\rm Im}(f))\times_{\top_2} (U^{\top_2}\cap{\rm Im}(f))\right)\\
	 &=m\left((U^{\top_2}\times_{\top_2} U^{\top_2})\cap{\rm Im}(f)^{\times_{\top_2} 2}\right)\\
	 &\subset U^{\top_2}\cap {\rm Im}(f)=U^{\top_{{\rm Im}(f)}}.
	 \end{align*} 
	 The last inclusion is a consequence of Condition (\ref{eq:mlocalgm}) for $A_2$ and ${\rm Im}(f)$ being closed under the product $m$. Therefore} Condition (\ref{eq:mlocalgm}) is satisfied for ${\rm Im}(f)$, so that ${\rm Im}(f)$ a locality subalgebra of $A_2$.
	 
	\end{proof}

\subsection{The locality tensor algebra  {of a pre-locality vector space}}  \label{sect:lta}		

The tensor algebra over a vector space $V$ is defined as ${\mathcal T}(V):=\bigoplus_{n\geq 0}V^{\otimes n}$ with the convention that 
$V^{\otimes 0}=\K$ and $V^{\otimes 1}=V$.  Following \cite{CGPZ1}, we define    a  locality tensor algebra in a similar manner, modulo the fact that we first build it on a pre-locality vector space.
\begin{defn}\label{defn:filteredtensoralg}
\begin{itemize}
\item The $n$-th filtration of the locality tensor algebra over a pre-locality vector space $(V,\top)$ is defined as 
\[{\mathcal T}_{\top}^n (V):= \bigoplus_{k=0}^n V^{\otimes_\top^k}.\]

\item  For any integers $0\leq m\leq n$, the canonical injections
\begin{equation}\iota_{m, n}: {\mathcal T}_{\top}^m (V) \longrightarrow{\mathcal T}_{\top}^n (V)
\end{equation}
are linear maps.
\item The {\bf locality tensor algebra} over a pre-locality vector space $(V, \top)$ is defined as 
\[{\mathcal T}_\top (V):= \bigcup_{n\geq0}\mathcal T_\top^n(V)=\bigoplus_{n\geq0} V^{\otimes_{\top}n}.\]
\end{itemize}
\end{defn}
The following results set the basis to define a locality relation on $\mathcal T_{\top}(V)$ as a quotient relation.
\begin{prop}\label{prop:talgaseqclasses}
Given a pre-locality vector space  $(V,\top)$, then
	 \[{\mathcal T}_{\top}^n(V)={}^{\K\left(\bigcup_{k=0}^nV^{\times_{\top}k}\right)}{\mskip -4mu\big/\mskip -3mu}_{\left(I_{\rm mult}^n {(V)}\cap \K\left(\bigcup_{k=0}^nV^{\times_{\top}k}\right)\right)}\]
	where we have set 
	$I_{\rm mult}^n {(V)}:=\bigoplus_{k=1}^nI_{{\rm mult},k} {(V)}$.
\end{prop}
\begin{proof} 
	 Given a direct sum of vector spaces $V=\bigoplus_{k= 0}^nV_k$ and $I_k$ a linear subspace of $V_k$ for each $0\leq k\leq n$, it is a standard result of linear algebra that the quotient space 
	 $\mathbb{V}=\sfrac{(\bigoplus_{k=0}^nV_k)}{\bigoplus_{k=0}^nI_k}$ inherits the structure of a direct sum of vector space 
	 $\mathbb{V}=\oplus_{k= 0}^n \mathbb{V}_k$ with   $\mathbb{V}_k=\sfrac{V_k}{I_k}$.  Applying this to 
	 $V_k:=  \K(V^{\times_{\top}k})$ and $I_k:= I_{\rm mult, k} {(V)}\cap\K(V^{\times_\top k})$ ({we simplify the notation leaving out the dependence on $V$}) and noticing that 
\begin{equation*}
\bigoplus_{k=0}^nI_k = \bigoplus_{k=0}^n\left(I_{\rm mult,k}\cap\K(V^{\times_\top k})\right) = \bigoplus_{k=0}^1 I_{\rm mult,k}\cap\bigoplus_{k=0}^n\K(V^{\times_\top k})
=I_{\rm mult}^n\cap \K\left(\bigcup_{k=0}^nV^{\times_{\top}k}\right).
\end{equation*}
then yields the result, where we have set $I_{\rm mult,0}:=\{0\}.$
\end{proof} 
\begin{defn}\label{defnprop:locrelonfilteredtenalg} Let $N\in \N$ and $(V,\top)$ be a pre-locality vector space.
\begin{itemize} 
\item 
 We define the locality relation $\top_\times^N$ on $\bigcup_{k=0}^NV^{\times_{\top}k}$ as \[\top_\times^N:=\bigcup_{k=1}^N\bigcup_{l=1}^N \top_{\times k,l}\ \ (\mathit{see }\ (\ref{eq:locreltimesmn})),\]
which we then linearly extend  to $\K(\bigcup_{k=0}^NV^{\times_{\top}k})$ as in (\ref{eq:locrelonthefreespan}).
\item The locality relation $\top_{\otimes}^N$ on $\mathcal{T}_\top^N(V)$ is defined as the quotient relation for the canonical map $\pi:\K(\bigcup_{k=0}^NV^{\times_\top k})\to\mathcal{T}_\top^N(V)$.  
\end{itemize}
\end{defn}
\begin{rk}
For pre-locality vector space $(V,\top)$ and any $n\in \N$, the pair $(\mathcal T_\top^n(V),\top_\otimes^n)$ is a pre-locality vector space.
\end{rk}

It follows from the definition of $\top^n_\times$ and $I_{\rm mult}^n$, that $\top_\otimes^n\subset\top_{\otimes}^{n+1}$ for any $n\in\N$. Consequently, for every pair $0\leq m, n$ natural numbers, there is a canonical embedding $\iota_{m,m+n}:\top_{\otimes}^m\to\top_\otimes^{m+n}$. This leads to the following definition.

\begin{defn} \label{defn:loc_rel_tensor_alg}
Given a pre-locality vector space $(V,\top)$, the locality relation on the locality tensor algebra $\mathcal T_\top(V)$ is defined as a direct limit $\top_\otimes:=\underset{\longrightarrow}{\rm Lim}\, \top_\otimes^n.$

\end{defn} 
\begin{rk}
 The pair $(\mathcal T_\top(V),\top_\otimes)$ is a trivially a pre-locality vector space, since $\top_\otimes$ is symmetric by construction.
\end{rk}

\begin{notation}
From now on we use $\top_\otimes$ instead of $\top_{\otimes m, n}$, $\top_{\otimes n}$, or $\top_{\otimes}^n$ since they are all restrictions of the first one. 
\end{notation}

\begin{rk} \label{rk:unit_loc_tensor_prod}
		Since $V^{\times_\top^0}:=\K$, and $k\top_\times^N (v_1,\cdots,v_n)$ for any $k\in\K$ and 
		$(v_1,\cdots,v_n)\in V^{\times_\top^n}$ with $N\geq n\geq1$, then $k\top_\otimes\, \left(v_1\otimes\dots\otimes v_n\right)$ for every 
		$v_1\otimes\dots\otimes v_n\in V^{\otimes_\top^n}$ for every $n\geq 1$ and hence
		$(\K=V^{\otimes_\top^0})\top_\otimes {\mathcal T}_\top (V)$.
\end{rk}

\begin{prop}\label{prop:tensprodaslocmap}
	Let $(V,\top)$ be a pre-locality vector space. The usual product $\bigotimes:{\mathcal T}(V)\times {\mathcal T}(V)\to {\mathcal T}(V) $ on 
	the tensor algebra  restricts to $\top_\otimes\di{=}{\mathcal T}_\top(V)\times_{{\top_\otimes}}  {\mathcal T}_\top(V)$ where it defines a   $\top_\times$-bilinear map (see (\ref{eq:topbilinear})) and 
	\[\left({\mathcal T}_\top (V), \top_\otimes, \bigotimes,u\right) \] defines a pre-locality algebra, where $u$ is the canonical injection $u:\K\to V^{\otimes_\top^0}$.
\end{prop}
\begin{proof}
	\begin{itemize}
	 \item Let us first check that the restriction is  ${\mathcal T}_\top(V)$-valued, namely that 
	 $\bigotimes\left(\top_{\otimes}\right)\subset {\mathcal T}_\top(V)$.  
	 
	For $([a],[b])\in\top_{\otimes}$, we may assume without loss of generality that $a=(a_1,\cdots,a_m)\in V^{\times_{\top}m}$, 
	 $b=(b_1,\cdots,b_n)\in V^{\times_{\top}n}$ and $(a,b)\in\top_{\times m,n}$. Therefore $(a_i,b_j)\in\top$ for every $i$ and $j$ implying that 
	 $ab:=(a_1,\cdots,a_m,b_1,\cdots,b_n)\in V^{\times_{\top}(m+n)}$ so that $\bigotimes([a],[b])=[ab]\in {\mathcal T}_{\top}(V)$.
	 \item The fact that it is  $\top_\times$-bilinear follows from the fact that $\bigotimes$ is a restriction of the usual product on 
	 the tensor algebra.
	 \item The associativity of the usual product $\bigotimes$ is preserved when we restrict to $\top_\otimes$ whenever it is well defined. Therefore $(\mathcal T_{\top}(V),\top_\otimes, \bigotimes,u)$ is indeed a pre-locality algebra. \qedhere
	\end{itemize}
	\end{proof}

	The following Proposition will be of use in the sequel.
	\begin{prop}\label{prop:CompareTensorAlgebras}
Let $(W,\top')$ be a (pre-)locality subspace of $(V,\top)$ a (pre-)locality vector space. Then $\mathcal T_{\top'}(W)$ is a (pre-)locality subalgebra of $\mathcal T_\top(V)$.
\end{prop}
\begin{proof}
We prove first that for every $n\in\Z_{\geq0}$, $W^{\otimes_{\top'}^n}$ is a subspace of $V^{\otimes_{\top}^n}$. This is trivially true for $n\in\{0,1\}$. For $n\geq2$. Notice that 
\[I_{\rm mult}(\underbrace{W,\dots,W}_{n-\text{times}})=I_{\rm mult}(\underbrace{V,\dots,V}_{n-\text{times}})\cap\K(\underbrace{W\times\dots\times W}_{n-\text{times}}).\] 
Intersecting both sides with $\K(\underbrace{W \times_{\top'}\dots\times_{\top'} W}_{n-\text{times}})$ it follows that 
\[\left(I_{\rm mult}(\underbrace{V,\dots,V}_{n-\text{times}})\cap\K(\underbrace{W\times\dots\times W}_{n-\text{times}})\right)\cap\K(\underbrace{W\times_{\top'}\dots\times_{\top'} W}_{n-\text{times}})= I_{\rm mult}(\underbrace{W,\dots,W}_{n-\text{times}})\cap\K(\underbrace{W\times_{\top'}\dots\times_{\top'} W}_{n-\text{times}}).\]
Moreover, since 
 $\K({W\times_{\top'}\dots\times_{\top'} W})\subset\K(W\times\dots\times W)$ and 
$\K({W\times_{\top'}\dots\times_{\top'} W})\subset\K(V\times_\top\dots\times_\top V)$ then
\[\left(I_{\rm mult}(\underbrace{V,\dots,V}_{n-\text{times}})\cap\K(\underbrace{V\times_{\top}\dots\times_{\top}V}_{n-\text{times}})\right)\cap\K(\underbrace{W\times_{\top'}\dots\times_{\top'} W}_{n-\text{times}})= I_{\rm mult}(\underbrace{W,\dots,W}_{n-\text{times}})\cap\K(\underbrace{W\times_{\top'}\dots\times_{\top'} W}_{n-\text{times}})\] 
and hence, using the identity $\K({W\times_{\top'}\dots\times_{\top'} W}) \cap \K(V\times_\top\dots\times_\top V)= \K({W\times_{\top'}\dots\times_{\top'} W})$ we have
 {\begin{eqnarray*}
 W^{\otimes_{\top'}^n} & =& \K({W\times_{\top'}\dots\times_{\top'} W})\,  / \,  I_{\rm mult}({W,\dots,W})\cap\K({W\times_{\top'}\dots\times_{\top'} W})\\
 & =&  \K(V\times_\top\dots\times_\top V) \cap \K({W\times_{\top'}\dots\times_{\top'} W}) \, /\,  I_{\rm mult}({W,\dots,W})\cap\K({W\times_{\top'}\dots\times_{\top'} W})  \\
 & \subset &  \K({V\times_{\top}\dots\times_{\top}V})\, / \, I_{\rm mult}({V,\dots,V})\cap\K({V\times_{\top}\dots\times_{\top}V})\\
 &=&V^{\otimes_{\top}^n}, 
\end{eqnarray*}} 
(since $(A\cap C)/(B\cap C)\subseteq A/B$ for $B\subseteq A$).
In  particular $\mathcal T_{\top'}(W)$ is a subspace of $\mathcal T_\top(V)$. We are only left to prove that the injection map $\iota:\mathcal T_{\top'}(W)\to \mathcal T_\top(V)$ is a locality map. The inclusion $\top'\subset\top$ implies that $\top_{\times}^{'n}\subset \top_{\times}^n$ for any $n$ (See \ref{eq:locreltimesmn} and Definition \ref{defnprop:locrelonfilteredtenalg}), and therefore \begin{align*}
([w_1],[w_2])\in\top_{\otimes}^{'n}&\Rightarrow(\exists w_1'\in[w_1])(\exists w_2'\in[w_2]):\ (w_1',w_2')\in\top_\times^{'n}\\
&\Rightarrow(\exists w_1'\in[w_1])(\exists w_2'\in[w_2]):\ (w_1',w_2')\in\top_\times^n\\
&\Rightarrow ([w_1],[w_2])\in\top_{\otimes}^{n}.
\end{align*} Thus $\iota$ is a morphism of (pre-)locality vector spaces. One easily checks  that it is moreover a morphism of (pre-)locality algebras which proves the statement of the proposition.
\end{proof}

As in the usual (i.e. non-locality) case, the pre-locality tensor algebra enjoys a universal property.

\begin{thm}[Universal property of locality tensor algebra over a pre-locality vector space]\label{thm:univpptyloctenalg-preloc}		
		Let $(V,\top)$ be a pre-locality vector space, $(A,\top_A)$ a pre-locality algebra  {whose product $m_A:(A\times_{\top}A,\top_{A\times_{\top}A})\to(A,\top_A)$ is a locality map}, and $f:V\to A$ a locality linear map. There is a unique pre-locality algebra morphism $\sy{\psi}:{\mathcal T}_{\top}(V)\to A$ such that the following diagram commutes		\begin{equation*}
		\begin{tikzcd}
		(V,\top)\arrow[r,"\otimes_\top"]\arrow[rdd,"f"]&({\mathcal T}_\top(V),\top_\otimes) \arrow[dd,"\sy{\psi}"]\\
		\\
		&(A,\top_A)
		\end{tikzcd}
		\end{equation*}
		
		where $\otimes_\top:V\to {\mathcal T}_{\top}(V)$ is the canonical (locality) injection map.
	\end{thm}
	
	\begin{proof}
		Let $f:V\to A$ be a locality linear map. We define for every $n\in\N$, a locality $n$-linear map $f_n:V^{\times_{\top^n}}\to A$ as $f_n(x_1,\cdots,x_n):= {m_A^{n-1}(f(x_1),\cdots, f(x_n))}$.
		Thanks to the universal property of the locality tensor product (Theorem \ref{thm:univpropltp-preloc}), there are locality linear maps $\sy{\psi}_n:V^{\otimes_{\top^n}}\to A$ such that $f_n=\sy{\psi}_n\circ \otimes_{\top n}$ where $\otimes_{\top n}$ is the canonical map from $V^{\times_{\top^n}}\to V^{\otimes_{\top^n}}$. The map $\sy{\psi}:{\mathcal T}_{\top}(V)\to A$ defined as the sum of the $\sy{\psi}_n$'s is  a locality algebra morphism such that $f=\sy{\psi}\circ \otimes_\top$. \cy{It is unique due to the uniqueness of each of the maps $\sy{\psi}_n$ (Theorem \ref{thm:univpropltp-preloc}).}
	\end{proof}

Similarly to the universal property of the locality tensor product, this last statement is equivalent to the universal property of the usual (non-locality) tensor algebra. 
	
	\begin{thm}
	The universal property of the locality tensor algebra is equivalent to the universal property of the usual (non-locality) tensor algebra.
	\end{thm}
	\begin{proof} 
		The fact that the universal property of the locality tensor algebra implies the usual one follows from choosing the trivial locality relation $\top=V\times V$. 
		
		For the converse: Assuming the universal property of the usual tensor algebra and given a locality linear map $f:V\to A$, it is in particular a linear map. Applying the usual universal property, we get an algebra morphism $\phi:{\mathcal T}(V)\to A$ such that $f=\phi\circ \otimes$ where $\otimes$ is the canonical injection of $V$ into ${\mathcal T}(V)$. Recall that ${\mathcal T}_{\top}(V)$ is a linear subspace of ${\mathcal T}(V)$ which contains $V$, so $\otimes$ is also the canonical injection of $V$ into ${\mathcal T}_{\top}(V)$. Restricting $\phi|_{{\mathcal T}_{\top}(V)}$ we therefore obtain a locality algebra morphism such that $f=\phi|_{{\mathcal T}_{\top}(V)}\circ \otimes$. We conclude that the usual universal property implies the locality one. 
	\end{proof}

\subsection{The  locality universal enveloping algebra}\label{subsec:luea}
	  Let us first introduce some terminology inspired by  \cite{CGPZ1}. Here, instead of considering  locality Lie algebras straightaway, we first introduce a notion of pre-locality  Lie algebra.
	\begin{defn} \label{defn:localityLieAlgebra}
			\begin{itemize}
			\item A {\bf pre-locality Lie algebra} is a triple $(\g, \top_{\g}, [,])$ where $(\g, \top_{\g})$ is a pre-locality vector space, and $[,]:\top_{\g}\to \g$ is a locality $\top_\times$-bilinear map antisymmetric which satisfies the following properties:
			\begin{itemize}
				\item For every $U\subset \g$,   the Lie bracket stabilises polar sets, i.e. it maps $(U^{\top}\times U^{\top})\cap \top_{\g}$ into $U^{\top}$.
				\item For $(a,b,c)\in V^{\times_{\top^3}}$ then $[[a,b],c]+[[c,a],b]+[[b,c],a]=0$.
			\end{itemize}
			\item A {\bf locality Lie algebra} is a pre-locality Lie algebra $(\g,\top_\g,[,])$ such that $(\g,\top_\g)$ is also a locality vector space.
			\item Let $(\g_1, \top_{\g_1}, [,]_1)$ and $(\g_2, \top_{\g_2}, [,]_2)$ be two (resp. pre-)locality Lie algebras. A locality linear map $f:\g_1\to\g_2$ is called a {\bf (resp. pre-)locality Lie algebra morphism} if $f([x,y]_1)=[f(x),f(y)]_2$, for every independent pair $x\top_1y$.
			
		\item	Let $(\mathfrak g_2, \top_{\mathfrak g_2}, [,]_2)$ be a (pre-)locality Lie algebra and and $\g_1\subseteq\g_2$. We call $(\mathfrak g_1, \top_{\mathfrak g_1}, [,]_1)$ a {\bf (pre-) locality  Lie subalgebra} of $(\mathfrak g_2, \top_{\mathfrak g_2}, [,]_2)$ if the inclusion map $\iota: \mathfrak g_1\hookrightarrow  \mathfrak g_2$ is a (pre-) locality Lie algebra morphism.
		\end{itemize}
	
 {The following example provides a  justification of the universal property of the enveloping algebra.}

 {\begin{ex}
If $(A,\top_A,m_A)$ is a (pre-)locality associative algebra, then it is a (pre-)locality Lie algebra, with the bracket defined by
\[\forall (x,y)\in A\times_\top A,\: [x,y]=xy-yx.\]
\end{ex}}
	
	\end{defn}
	We now define the locality universal enveloping algebra similarly to the usual universal enveloping algebra
	\begin{defn}\label{defn:luenvelopingalgebra}
		Let $(\g, \top_{\g}, [,])$ be a pre-locality Lie algebra. Consider the pre-locality ideal $J_\top(\g)$ (see Definition \ref{defn:prelocalg}) of ${\mathcal T}_{\top}(\g)$ generated by all terms of the form $a\otimes b-b\otimes a-[a,b]$ for $(a,b)\in\top_{\g}$. The {\bf  locality universal enveloping algebra} of $\g$ is defined as 
		\begin{equation} \label{eq:univlocalg} 
		U_{\top}(\g):=\sfrac{{\mathcal T}_{\top}(\g)}{J_\top(\g)}. 
		\end{equation}
		The locality relation $\top_U$ on $U_{\top}(\g)$ is the final locality relation for the quotient map on $U_\top(\g)$ induced by $\top_{\otimes}$ on ${\mathcal T}_{\top}(\g)$. This means that two equivalence classes $x$ and $y$ in $U_{\top}(\g)$ are locally independent i.e. $x\top_Uy$ if, and only if there are elements $a$ in $x$ and $b\in y$ in ${\mathcal T}_\top(\g)$ such that $a\top_{\otimes}b$. 
	\end{defn}
	
	Notice that the locality relation $\top_U$ is defined in a similar manner to $\top_{\otimes}$, namely it is induced by $\top_\times$ on $\K(\g\times_\top\g)$ (see Definition-Proposition \ref{defnprop:locrelonfilteredtenalg}).
	\begin{rk}
		The ideal $J_\top(\g)$ is not graded,  
		 so one does not expect   $U_{\top}(\g)$  to inherit a grading from  ${\mathcal T}_{\top}(\g)$. Yet it does inherit a filtration \[(U_{\top}(\g))^n=\sfrac{({\mathcal T}_{\top}(\g))^n}{J_\top(\g)\cap({\mathcal T}_{\top}(\g))^n},\]where $({\mathcal T}_{\top}(\g))^n=\bigoplus_{i=0}^{n}\g^{\otimes_{\top}i}$.
			It is easy to check that this is indeed a filtered pre-locality algebra which induces the grading $(U_{\top}(\g))_n=\sfrac{(U_{\top}(\g))^n}{(U_{\top}(\g))^{n-1}}$. 
		\end{rk}

		Similarly to Proposition \ref{prop:CompareTensorAlgebras}, the following Proposition compares the universal enveloping algebras of two locality Lie algebras.
\begin{prop}\label{prop:CompareUnivAlgebras}
Let $(\g',\top')$ be a (pre-)locality Lie subalgebra of $(\g,\top)$ a (pre-)locality Lie algebra. Then $(U_{\top'}(\g'),\top'_U)$ is a (pre-)locality subalgebra of $(U_\top(\g),\top_U)$.
\end{prop}
\begin{proof}

By Proposition \ref{prop:CompareTensorAlgebras} $\mathcal T_{\top'}(\g')\subset\mathcal T_\top(\g)$, and by construction $J_{\top'}(\g')=J_\top(\g)\cap \mathcal T_{\top'}(\g')$. It follows that $U_{\top'}(\g')$ is a subspace of $U_{\top}(\g)$. Proposition \ref{prop:CompareTensorAlgebras} also states   that $\top'_\otimes\subset\top_{\otimes}$ and we can   show the inclusion $\top'_U\subset\top_U$ in a similar manner. Therefore $U_{\top'}(\g')$ is (pre-)locality subspace of $U_\top(\g)$. It is straightforward to see that it is moreover a (pre-)locality subalgebra as expected.
\end{proof}
	We are now ready to prove the   universal property of the locality universal enveloping algebra  $U_{\top}(\g)$.
	\begin{thm}\label{thm:univproplocunivenvalg-preloc}  
		Let $(\g, \top_{\g}, [,])$ be a pre-locality Lie algebra, $(A,\top_A)$ a pre-locality algebra \cy{with} product $m_A:(A\times_\top A,\top_{A\times_\top A})\to(A,\top_A)$, and $f:\g\to A$ a pre-locality Lie algebra morphism \cy{for the Lie bracket on $A$ defined as the commutator of} the product. 
		 There is a unique pre-locality algebra morphism $\phi:U_{\top}(\g)\to A$ such that $f=\phi\circ \iota_\g$ where \sy{\begin{equation}\label{eq:iotag}\iota_\g:\g \longrightarrow U_{\top}(\g) \end{equation}}  
		 is the canonical  (locality) map from $\g$ to $U_{\top}(\g)$. \cy{In other words, the following diagram commutes}
		\begin{equation*}
		\begin{tikzcd}
		(\g,\top_\g)\arrow[r,"\iota_\g"]\arrow[rdd,"f"]&(U_\top(\g),{\top_U}) \arrow[dd,"\phi"]\\
		\\
		&(A,\top_A).
		\end{tikzcd}
		\end{equation*}
	\end{thm}
	\begin{proof}
	\begin{itemize}
		\item Existence:	Since $f$ is   linear map, we may apply the universal property of the locality tensor algebra (Theorem \ref{thm:univpptyloctenalg-preloc}) to obtain a pre-locality algebra morphism $\psi:{\mathcal T}_{\top}(\g)\to A$ such that $f=\psi\circ \otimes_\top$ where $\otimes_\top$ is the canonical map from $\g$ to ${\mathcal T}_{\top}(\g)$. We then define $\phi:U_{\top}(\g)\to A$ by $\phi([x]):=\psi(x)$. This is clearly well defined since $f$ is a pre-locality Lie algebra morphism for indeed we have:
		\[\psi(x+a\otimes b-b\otimes a-[a,b])=\psi(x)+\psi(a\otimes b-b\otimes a-[a,b])=\psi(x)+f(a) f(b)-f(b)f(a)-f([a,b])=\psi(x)\] 
		It also satisfies the equation $f=\phi\circ \iota_\g$ since $\iota_\g$ is the composition of $\otimes_\top$ with the map which takes every element of ${\mathcal T}_{\top}(\g)$ to its equivalence class on $U_{\top}(\g)$. \cy{For the same reason, and because $\psi$ is a locality map (Theorem \ref{thm:univpptyloctenalg-preloc}), $\phi$ is a locality map. The fact that $\phi$ is a pre-algebra morphism is direct: with a small abuse of notation we have
		\begin{equation*}
		 \phi([a]\otimes[b])=\phi([a\otimes b])=\psi(a\otimes b)=f(ab)=f(a)f(b)=\phi([a])\phi([b]).
		\end{equation*}}
	\item Uniqueness:	\sy{ Let $\phi:U_{\top}(\g)\to A$ be  pre-locality algebra morphism such that $f=\phi\circ \iota_\g$.    By the universal property of the locality tensor algebra, the pre-locality Lie algebra morphism $f:\g\to A$  uniquely extends to  a pre-locality algebra morphism $\psi: {\mathcal T}_\top(\g)\to A$. By the uniqueness of this extension, $\psi $ has to coincide with $x\mapsto \phi([x])$, which proves that $\phi$ is determined by $\psi$ and hence the uniqueness of $\phi$. }
\end{itemize}
	\end{proof}

\begin{prop} 
	The universal property of the locality universal enveloping algebra  $U_{\top}(\g)$ implies the universal property of the usual universal enveloping algebra  $U(\g)$.
\end{prop}
\begin{proof} 
Given a Lie algebra $\g$, an algebra $A$ and a Lie algebra morphism $f:\g\to A$, it is enough to consider the trivial locality relation $\top=\g\times\g$ and this yields the existence and uniqueness of the algebra morphism $\phi:U(\g)\to A$ required for the universal property of the universal enveloping algebra.
\end{proof}

	\begin{rk}
	For the equivalence of the universal properties of the universal enveloping algebras in the usual and locality set up, one would need to extend the Lie bracket of a pre-locality Lie algebra to the whole Lie algebra. However this is not always possible  as can be seen from Counterexample \ref{coex:loc_Lie_no_Lie}.
	
	   Proposition  \ref{prop:locbiltobil}  yields a bilinear map which is antisymmetric on $\g\times_\top\g$ by construction since $\top$ is symmetric. Outside of the span of the image of the original (non-extended map), the extended map vanishes identically so that the extended map is antisymmetric. However, as the next counter-example shows, it does not in general satisfy Jacobi identity.
	\end{rk}

\begin{coex} \label{coex:loc_Lie_no_Lie} 
 {Three-dimensional Lie algebras} are known, and there are finitely many (see  {Mubarakzyanov's Classification} \cite{Mub}). We build an infinite family of distinct locality Lie algebras, which as a consequence, does not correspond to an ordinary Lie algebra equipped with a locality algebra.

 Take $\g=\R^3$, $(e_1,e_2,e_3)$ its canonical basis and the locality relation $ \top_\g$ defined by the   subset of $\R^3\times \R^3$ obtained by symmetrising the following set
 \begin{equation*}
\langle e_1\rangle\times\langle e_1\rangle\bigcup\langle e_2\rangle\times\langle e_2\rangle\bigcup\langle e_3\rangle\times\langle e_3\rangle\bigcup \langle e_1,e_2\rangle\times\langle e_3\rangle.
 \end{equation*}
 Let   $[\cdot,\cdot]:\top_\g\longrightarrow\g$ be   the linear antisymmetric map defined by
 \begin{equation} \label{eq:locality_Lie}
  [e_1,e_3] = \lambda e_1+\mu e_3,\quad [e_2,e_3]=\mu' e_3
 \end{equation}
 with $ {(\lambda,\mu,\mu')\in(\R^*)^3}$. Notice that since $(e_1,e_2)\notin\top_\g$, $[e_1,e_2]$ does not need to be defined. We argue is that it \emph{cannot} be defined such that $(\g,[\cdot,\cdot])$ is a usual (non-locality) Lie algebra.  {Indeed, let us write $[e_1,e_2]=xe_1+ye_2+ze_3$. Then computing $[[e_1,e_2],e_3]$ and its permutations, we find   that the Jacobi identity is satisfied if, and only if:
 \begin{equation*}
  -\mu'\lambda e_1-\lambda ye_2+(x\mu+y\mu'-\lambda z)e_3=0.
 \end{equation*}
 This equation has no solutions since $\mu'\lambda\neq0$.}
 
  However, it is easy to see that for any  {$(\lambda,\mu,\mu')\in(\R^*)^3$}, $(\g,\top_\g,[\cdot ,\cdot ])$ is a locality 
 Lie algebra. This follows from the fact that since $(e_1, e_2)\notin \top_{ \mathfrak g}$,  $(e_i,e_j,e_k)$ cannot lie in $\g^{\times_\top^3}$ if $i$, $j$ and $k$ are all different,  $[\cdot ,\cdot ]$ trivially satisfies the locality Jacobi identity. 
\end{coex}

\begin{rk}
     The previous counter-example indicates that the category of locality Lie algebras is much richer than the one of usual (non-locality) Lie algebras, even when they are {finite-dimensional}. This hints to the difficulty of     classifying such structures. As pointed out in the introduction, this issue of classification is left for future studies. 
    \end{rk}

\vfill \eject \noindent 
\part{Enhancement to the locality setup}

We want  to enhance to a locality setup the constructions of Part I carried out in a pre-locality framework. Since they involve tensor products, the question arises of the stability of locality vector spaces under locality tensor products. Like ordinary tensor products,  locality tensor products are defined as quotient spaces,  which brings us to the study of quotients of locality vector spaces. 

\section{Quotient of locality vector spaces} \label{section:five}

This section is dedicated to the following natural question:

\begin{eqnarray} \label{question}
&&\text{\it When is the quotient of a locality vector space  by a linear  subspace,   a locality vector space,} \nonumber\\
&& \text{\it if equipped with the quotient locality relation of  Definition \ref{defn:quotientlocality}~?}
\end{eqnarray}

Given a locality vector space $(V, \top)$ and a subspace $W$, the question can be reformulated as follows: does the following implication 
\[\forall (v_1, v_2. v_3)\in V^3, \quad  \left( [v_1] \ttop [v_2]  \, \wedge\, [v_1] \ttop [v_3]\right)\Longrightarrow \left([v_1]\ttop[ v_2+v_3]\right),\]
hold, 
where $\ttop$  stands for the quotient locality of  Definition \ref{defn:quotientlocality}~? This amounts to
whether the subsequent implication holds:	
\begin{equation}\label{eq:questionbis}
\forall (v_1, v_2. v_3)\in V^3, \quad      \left(\exists v_i'\in  [v_i],\exists v_1'' \in [v_1],  \quad v_1'\top v_2'   \, \wedge\,  v_1'' \top  v_3'\right)\Longrightarrow \left(\exists  v_{23}'\in [ v_2+ v_3], \exists \tilde v_1 \in [v_1], \quad \tilde v_{1}\top v_{23}'\right).
\end{equation}
Investigating this question in the context of locality tensor products  leads to two main assumptions formulated in the conjectural statements \ref{conj:main1} and   \ref{conj:main2}.  

We first reformulate Question (\ref{question}) in order get a better grasp on it  and  to motivate  the above statements.

\subsection{ {A group  theoretic interpretation } }

 { We can rephrase  Question (\ref{question}),} namely   whether or not a quotient of locality vector spaces 
is a locality vector space as whether or not a union  of groups is a group.   This question in full generality, is to our knowledge, still an open question.   A  good description of the state of the art on the question of when a group is the union of $n$ proper subgroups is given in \cite{Bha}.   An   answer was provided in \cite{Coh}  for $n\leq6$ and in \cite{Tom} for $n=7$. To our knowledge, the question for any $n$ or for infinite unions of groups remains open. 

\begin{notation}
	In the subsequent sections, given a locality set $(S,\top)$ we will use the polar map 
	\begin{eqnarray*}
		P^\top: {\mathcal P}(S)&\longrightarrow &{\mathcal P}(S)\\
		U&\longmapsto & U^\top.
	\end{eqnarray*}
	By convention, we take $P^\top(\emptyset):=S$. Furthermore, in a small abuse of notation we write $P^\top(u)$ instead of $P^\top(\{u\})$ if 
	$U=\{u\}$ has only one element.
\end{notation}
Let us recall an elementary result.  
\begin{lem} \label{lem:charac_LVS}
	Given a vector space $V$ equipped with  a locality relation $\top$, the pair $(V,\top)$ is a locality vector space if, and only if, 
	$P^\top(u)$ is a vector subspace of $V$ for any $u\in V$.
\end{lem}
\begin{proof}
	By definition, $(V,\top)$ is a locality vector space if, and only if, 
	$P^\top(U)$ is a vector subspace of $V$ for any subset $U\subseteq V$. If this holds, then it is obvious that $P^\top(u)$ is a vector subspace 
	of $V$ for any $u\in V$.
	
	Conversely, assume $P^\top(u)$ is a linear subspace of $V$ for any $u\in V$. Then for any subset $U\subseteq V$, we have 
	\begin{equation*}
	P^\top(U) := \bigcap_{u\in U}P^\top(u)
	\end{equation*}
	which is an intersection of vector spaces and thus a vector space.
\end{proof}

In order to make a precise statement, let us first  recall some basic facts.  Given a locality vector space $(V,\top)$ , $W\subseteq V$ a linear subspace of $V$ and 
$ {\pi}:V\longrightarrow V/W$ the {linear}  quotient map, we write $[u]:= {\pi}^{-1}( {\pi}(u))\subseteq V$. We have $[u]=\{\bar u+ w, w\in W\}=\bar u+W$ where $\bar u$ is any representative in $V$ of the class $ {\pi}(u)$. Note that $0_V\in [u]$ implies $[u]=W$. In $V/W$, we write indifferently $[u]$ or $ {\pi}(u)$ as it is more convenient. 

We need one other intermediate result:
\begin{lem} \label{lem:charac_Hu}
	$(V,\top)$ a locality vector space, $W\subseteq V$ a linear subspace of $V$ and $ {\pi}:V\longrightarrow V/W$ the quotient map. Then 
	\begin{equation*}
	P^\ttop([u]) = \bigcup_{u'\in[u]} {\pi}(P^\top(u'))
	\end{equation*}
	for any $u$ in $V$.
\end{lem}
\begin{proof}
	Let $(V,\top)$, $W$ and $ {\pi}$ be as in the statement and let $u$ be any element of $V$.
	\begin{itemize}
		\item $\subseteq$: Let $[\alpha]$ be in $V/W$ such that $[\alpha]\in P^\ttop([u])$  i.e., such that $[\alpha]\ttop[u]$. Then by definition of $\ttop$ there is some $u'\in[u]$ and some $\alpha'\in[\alpha]$ such that $\alpha'\top u'\Leftrightarrow \alpha'\in P^\top(u')$. This implies
		\begin{equation*}
		[\alpha]= {\pi}(\alpha)= {\pi}(\alpha')\in  {\pi}(P^\top(u')) \subset \bigcup_{u'\in[u]} {\pi}(P^\top(u')).
		\end{equation*}
		Thus $P^\ttop([u]) \subseteq \bigcup_{u'\in[u]} {\pi}(P^\top(u'))$.
		\item $\supseteq$: Let $[\alpha]$ be in $V/W$ such that $[\alpha]\in\bigcup_{u'\in[u]} {\pi}(P^\top(u'))$. Then
		\begin{equation*}
		\exists u'\in[u]:[\alpha]\in  {\pi}(P^\top(u')) ~\Longrightarrow~\exists (u',w)\in[u]\times W:(\alpha+w)\in P^\top(u').
		\end{equation*}
		In $V$, we have $\alpha':=(\alpha+w)\in[\alpha]$ and this in turn implies
		\begin{equation*}
		\exists(u',\alpha')\in[u]\times[\alpha]:\alpha'\top u' ~\Longrightarrow~[\alpha]\ttop[u].
		\end{equation*}
		Thus $[\alpha]\in P^\ttop([u])$ and $ \bigcup_{u'\in[u]} {\pi}(P^\top(u'))\subseteq P^\ttop([u])$ which proves the statement.
		\qedhere
	\end{itemize}
\end{proof}
We now state the  main result of this section.
\begin{thm} \label{thm:cond_locality_quotient}
	Let $(V,\top)$ be a locality vector space, $W\subseteq V$ a linear subspace of $V$ and $ {\pi}:V\longrightarrow V/W$ the quotient map. Then 
	the following statements are equivalent:
	\begin{enumerate}
		\item $(V/W,\ttop)$ is a locality vector space,
		\item The set 
		\begin{equation*}
		H_u := \bigcup_{u'\in[u]} {\pi}(P^\top(u'))
		\end{equation*}
		is a commutative group (for the internal operation $+$ induced on the quotient space by the internal operation $+$ on $V$) for any $u$ in $V$.
		\item The set $H_u$
		is a commutative semigroup (for the same product) for any $u$ in $V$.
	\end{enumerate}
\end{thm}
\begin{rk}
	Notice that $P^\top(u')$ is a subset of $V$, thus $ {\pi}(P^\top(u'))$ is a subset of $V/W$, thus the $\bigcup$ notation in the Theorem.
\end{rk}
\begin{proof}
	\begin{itemize}
		\item 2. $\Leftrightarrow$ 3.
		If $H_u$ is a group, it is in particular a semigroup. Let us show that the converse is also true. Assume that $H_u$ is a semigroup (i.e. that 
		it is closed under summation) for any $u$ in $V$ and observe that by Lemma \ref{lem:charac_Hu} $H_u=P^\ttop([u])$. Since $(V,\top)$ is a locality 
		vector space we have $[0]\ni0\top u\in[u]$ thus $[0]=0_{V/W}\in P^\ttop([u])=H_u$ be definition of $\ttop$. Hence $H_u$ is a monoid 
		for any $u$ in $V$. We are left to show that if $H_u$ is stable under addition, it is also stable under taking the inverse, that is 
		multiplication by the scalar $-1$.
		
		For any $[\alpha]\in V/W$ we have:
		\begin{equation*}
		[\alpha]\in P^\ttop([u]) ~\Longrightarrow~\exists(\alpha',u')\in[\alpha]\times[u]:\alpha'\top u' ~\Longrightarrow~\exists(\alpha',u')\in[\alpha]\times[u]:\lambda\alpha'\top u'~\forall\lambda\in\K
		\end{equation*}
		since $P^\top(u)$ is a vector subspace of $V$. Then since $\alpha'\in[\alpha]\Rightarrow\lambda\alpha'\in\lambda[\alpha]$ we deduce that 
		\begin{equation*}
		[\alpha]\in P^\ttop([u]) ~\Longrightarrow~\lambda[\alpha]\in P^\ttop([u])~\forall\lambda\in\K
		\end{equation*}
		and in particular $H_u=P^\ttop([u])$ is a group.
		\item 1. $\Leftrightarrow$ 2. Let $V$, $W$ and $ {\pi}$ be as in the statement of the Theorem. By   Lemma \ref{lem:charac_LVS}, we know  that $(V/W,\ttop)$ is a locality vector 
		space if, and only if, $P^\ttop([u])=H_u$ is a vector subspace of $V/W$ for any $[u]\in V/W$.
		
		We have already shown that $H_u$ is stable by scalar multiplication, thus it is a vector space if, and only if, it is a group for the addition. \qedhere
	\end{itemize}
\end{proof}

\subsection{ {Locality on quotients and  locality exact sequences}}
 {In this paragraph, we rephrase Question (\ref{question})   in terms of locality exact sequences.}
Let us first recall a known result of linear algebra.
\begin{lem}Let $V_1, V_2, V$ be three linear spaces. The following  properties are equivalent:
\begin{enumerate}
	\item There is a short exact sequence 
	$0\to V_1\overset{\iota_1}{\to} V\overset{\pi_2}{ \to} V_2\to 0$   i.e., $\iota_1$ is an injective morphism, $\pi_2$ is a surjective morphism, and ${\rm Im}(\iota_1)={\rm Ker}(\pi_2)$.
	\item There is an injective morphism $\iota_1: V_1\to V$ such that $V/\iota_1( V_1)\simeq V_2$.
\end{enumerate}
\end{lem}
In the following, we denote by $\phi$ the resulting isomorphism:
\begin{align*}
 \phi:V/\iota_1(V_1) & \longrightarrow V_2\\
 [v] & \longmapsto \phi([v]):=\pi_2(v),
\end{align*} 
whose inverse map is given by $\phi^{-1}(v_2)=[w_2]$ for any $w_2\in\pi_2^{-1}(v_2)$. It is well defined since $\phi^{-1}(v_2)$ is independent of the choice of  $w_2$.
\begin{ex}
	If $V_1\subset V$ is a linear subspace of $V$, we can take $\iota_1$ to be the identity map $\iota_1: v_1\mapsto v_1$ so that item 2. reads $V/V_1\simeq V_2$.
\end{ex}
To give a locality counterpart of this result, we define locality short exact sequences.
\begin{defn} \label{defn:loc_short_ex_sequence}
 A {\bf locality short exact sequence} (resp. a {\bf pre-locality short exact sequence}) is a sequence $0\to (V_1,\top_1)\overset{f_1}{\to} (V_2,\top_2)\overset{f_2}{ \to} (V_3,\top_3)\to 0$ such that
 \begin{enumerate}
  \item $(V_i,\top_i)$ are locality vector spaces (resp. pre-locality vector spaces) for $i\in\{1,2,3\}$,
  \item  $f_1$ and $f_2$ are locality maps,
  \item $0\to V_1\overset{f_1}{\to} V_2\overset{f_2}{ \to} V_3\to 0$ is a short exact sequence.
 \end{enumerate}
Note that items 2. and 3. imply that $f_1$ and $f_2$ are locality linear maps.
\end{defn}
\begin{lem} \label{prop:loc_short_ex_sequence}
Let $(V_1, \top_1), (V_2, \top_2), (V,\top)$ be three locality (resp. pre-locality) linear spaces. The following  properties are equivalent:
	\begin{enumerate}
		\item There is a short locality (resp. pre-locality) exact sequence 
		$0\to (V_1,\top_1)\overset{\iota_1}{\to} (V,\top)\overset{\pi_2}{ \to} (V_2,\top_2)\to 0$, where $\top_2$ is the final locality relation for the map $\pi_2$ (See Definition \ref{defn:final_relation}).
		\item There is an injective locality morphism $\iota_1: V_1\to V$ such that the canonical isomorphism $(V/\iota_1( V_1))\overset{\phi}{\to} V_2$ is a locality (resp. pre-locality) isomorphism for the independence relations $\ttop$ and $\top_2$.
	\end{enumerate}
\end{lem}
\begin{proof}
	\begin{itemize}
		\item 1$\Rightarrow 2$: since the short exact sequence in item 1. implies the isomorphism $V/\iota_1( V_1)\overset{\phi}{\to} V_2$ as vector spaces, all we need to check is the locality of the maps   $\phi$ and $\phi^{-1}$. If $[v]\ttop[v']$, there is some $w\in [v]$ and some $w'\in [v']$ such that $w\top w'$. The locality of $ \pi_2$ then implies that $\pi_2(w)\top_2\pi_2(w')$ so that $\phi([v])\top_2\phi([v'])$ which show the locality of $\phi$.
		
	The inverse map $\phi^{-1}$ is defined by	$\phi^{-1}(v) =  [w]$ for any $w\in\pi_2^{-1}(v)$. Since $\top_2$ is the final locality relation for $\pi_2$, $v_1\top_2 v_2$ implies that there are $w_1\in\pi_2^{-1}(v_1)$ and $w_2\in\pi_2^{-1}(v_2)$ such that $w_1\top w_2$. By definition of the locality $\ttop$ on the quotient space, $[w_1]\ttop[w_2]$ or equivalently $\phi^{-1}(v_1)\ttop\phi^{-1}(v_2)$.
		\item 2$\Rightarrow$ 1: since the isomorphism $V/\iota_1( V_1)\simeq V_2$ holds as vector spaces, we have  the short exact sequence of vector spaces in item 1.   We want to show that   the maps $\iota_1$ and $\pi_2$ are locality maps, and that $\top_2$ is the final locality relation for $\pi_2$. By assumption, $\iota_1$ is a locality map. We show the locality of $\pi_2$. Recall that the surjective  map $\pi_2:V\to V_2$ is given by
		$ \pi_2(v):=  \phi([v])$.  Since $v\top v'$ it follows that $[v]\ttop [v']$ so that from the locality of $\phi$ we infer that
			$ \pi_2(v) =  \phi([v])$  is independent of 	$ \pi_2(v') =  \phi([v'])$. We prove that $\top_2$ is the final locality for $\pi_2$. Let $\top_2'$ be another locality relation on $V_2$ such that $\pi_2:(V,\top)\to(V_2,\top_2')$ is a locality map, and let $(v,v')\in V_2^2$ such that $v\top_2v'$. Since $\phi^{-1}$ is a locality map $\phi^{-1}(v)\ttop\phi^{-1}(v')$. By definition of $\ttop$, there are elements $w\in\phi^{-1}(v)$ and $w'\in\phi^{-1}(v')$ such that $w\top w'$. Notice that in particular $\pi_2(w)=v$ and $\pi_2(w')=v'$. Since $\top_2'$ makes $\pi_2$ a locality map, then $v\top_2'v'$ which implies $\top_2\subset\top_2'$ as expected. \qedhere			
	\end{itemize}
\end{proof}
This has a simple consequence:
\begin{prop} \label{prop:simple_consequence_question}
 Let $(V,\top)$ and $(V_1,\top_1)$ be two locality vector spaces and let $\iota_1:V_1\longrightarrow V$ an injective locality linear map. Then $(V/\iota_1(V_1),\ttop)$ is a locality vector space if, and only if, there is  a locality vector space  $(V_2,\top_2)$ and a map $\pi_2:V\longrightarrow V_2$ such that $\top_2$ is the final locality relation for $\pi_2$, and $0\to (V_1,\top_1)\overset{\iota_1}{\to} (V,\top)\overset{\pi_2}{ \to} (V_2,\top_2)\to 0$ is a locality short exact sequence.
\end{prop}
\begin{proof}
\begin{itemize}
 \item $\Leftarrow$: if   such a $(V_2,\top_2)$ exists then by Lemma \ref{prop:loc_short_ex_sequence},  $(V/\iota_1(V_1),\ttop)\simeq (V_2,\top_2)$ as locality vector space. It follows that $(V/\iota_1(V_1),\ttop)$ is a locality vector space.
 \item $\Rightarrow$: if $(V/\iota_1(V_1),\ttop)$ is a locality vector space, take $(V_2,\top_2):=(V/\iota_1(V_1),\ttop)$ and $\pi_2(v):=[v]$. Then by construction $0\to (V_1,\top_1)\overset{\iota_1}{\to} (V,\top)\overset{\pi_2}{ \to} (V_2,\top_2)\to 0$ is a locality short exact sequence, and by definition of the quotient locality $\ttop$, $\top_2$ is the final locality relation for $\pi_2$.\qedhere
\end{itemize}
\end{proof}

We have now rephrased Question \ref{question}  in terms of locality short exact sequences as follows:

\textit{Given a locality vector space $(V,\top)$ and $V_1$ a subspace of $V$, when is the pre-locality short exact sequence $0\to (V_1,\top\vert_{V_1\times V_1})\overset{\iota_1}{\to} (V,\top)\overset{\pi_2}{ \to} (\sfrac{V}{V_1},\ttop)\to 0$ a locality short exact sequence?}
 
 \subsection{Locality quotient vector spaces: examples and counterexamples}

 The following  statement provides a class of examples of locality quotient vector spaces   for the locality relation $\ttop$ (see Definition \ref{defn:quotientlocality}) induced by the orthogonality relation.
 \begin{prop} \label{ex:quotientlocHilbert}
 Take $V$ be any Hilbert space with scalar product $\langle,\rangle$.  We equip $V$ with the  locality relation $\top$ given by the orthogonality relation: $v\top v':\Longleftrightarrow \langle v, v'\rangle=0$. Then for any closed linear subspace $W$ of $V$ not reduced to $\{0\}$, the quotient locality $\ttop$  on $V/W$ , induced by $\top$ is the  complete locality relation: $\ttop=(V/W)\times(V/W)$.
 
 In particular, for any closed linear subspace $W$ of $V$, $(V/W,\ttop)$ is a locality vector space with $\ttop$ induced from $\top$, which is given by the orthogonality on $V$.
 \end{prop}
 \begin{rk} \label{rk:quotientlocHilbert}
  When $V$ is a finite dimensional euclidean space, then  $(V/W,\ttop)$ is a locality vector space for any linear subspace $W\neq\{0\}$ since any linear subspace of a finite dimensional euclidean space is closed.
 \end{rk}
 \begin{proof}
  We want to show that if $W\neq\{0\}$ then for any $(v_1,v_2)\in V^2$, there exist $(w_1,w_2)\in W^2$ such that $\langle v_1+w_1,v_2+w_2\rangle=0$. Write $v_i=u_i+\tilde w_i$ for $i\in\{1,2\}$ with $\tilde w_i\in W$ and $u_i\in W^\perp$. Then:
  \begin{equation} \label{eq:equation_perp_to_solve}
   \langle v_1+w_1,v_2+w_2\rangle=0~\Longleftrightarrow~\langle v_1,v_2\rangle + \langle \tilde w_1,w_2\rangle + \langle w_1,\tilde w_2\rangle + \langle w_1,w_2\rangle=0.
  \end{equation}
  and we want to find $w_1,w_2$ that solve \eqref{eq:equation_perp_to_solve}.
  Let us consider three different cases:
  \begin{itemize}
   \item If $\tilde w_1\neq 0$, 
   then 
   \begin{equation*}
    (w_1,w_2)=\left(0,-\frac{\langle v_1,v_2\rangle}{||\tilde w_1||^2}\tilde w_1\right)\in W^2
   \end{equation*}
   (with $||w||:=\sqrt{\langle w,w\rangle}$) solves \eqref{eq:equation_perp_to_solve}.
   \item If $\tilde w_1=0$ and $\tilde  w_2\neq 0$, then  we find a solution to \eqref{eq:equation_perp_to_solve} as in the first item by exchanging the roles of $w_1$ and $w_2$.
   \item If $\tilde w_1=\tilde w_2=0$, we pick $w\in W \neq \{0\} $ and  set
   \begin{equation*}
    (w_1,w_2)=\left(\frac{w}{||w||},- \frac{\langle v_1,v_2\rangle}{||w||}w\right)\in W^2
   \end{equation*}
   which solve \eqref{eq:equation_perp_to_solve}. \qedhere
  \end{itemize}
 \end{proof}

The following counterexample gives a  locality vector space with a linear subspace   whose quotient is 
not a locality vector space for the locality relation $\ttop$ (see Definition \ref{defn:quotientlocality}).  It shows that the answer to question \ref{question} cannot be always positive.
\begin{coex}
 We equip  the vector space $V$ of  {real valued maps on $\R$} with the locality relation $\top$ given by disjoint supports: $f\top g\Leftrightarrow {\rm supp}(f)\cap {\rm supp}(g)=\emptyset$. Let $W$ denote the linear subspace of constant functions. 
 Take three functions  $(u,v,w)\in V^3$ whose supports are respectively ${\rm supp}(u)=]1,+\infty[$, ${\rm supp}(v)=]0,+\infty[$ and ${\rm supp}(w)=]2,+\infty[$ and 
 $v(x)=u(x+1)=w(x+2)=1$ for any $x>0$. 
 Then in $V/W$ we have $[u]\ttop[v]$ since $u\top(v-1)$ and $[u]\ttop[w]$ since $(u-1)\top w$. However, 
 $[u]\cancel{\ttop}[v+w]$. Thus $V/W$ is not a locality vector space for $\ttop$. 
\end{coex} 

The   question (\ref{question})  is  particularly relevant if we want   the tensor product of a pair of locality vector spaces to be again a locality (not only pre-locality) vector space. The following is an example of a quotient of locality spaces which is again a locality vector space.

The subsequent folklore result provides a justification for the identification $V^{\otimes 1}=V$ made in Section \ref{sect:lta}. We give a constructive proof for it since it will be of use in Proposition \ref{prop:Visomorphism2}.
\begin{lem}\label{lem:Visomorphism}
Consider the subspace $I_{\rm lin}(V)$ of $\K(V)$ generated by all elements of the form
\begin{equation}\label{eq:1tuples1}
(a+b)-(a)-(b),
\end{equation}
and 
\begin{equation}\label{eq:1tuples2}
k(a)-(ka)
\end{equation}
for $a$ and $b$ in $V$ and $k$ in $\K$.

Then $V\simeq\sfrac{\K(V)}{I_{\rm lin}(V)}$ are isomorphic as vector spaces.
\end{lem}
\begin{rk} 
We use the notation $(a)$ with rounded brackets to denote elements of $\K V$ and to distinguish from elements of $V$. These brackets should not be confused with the squared one $[a]$ used to denote equivalence classes.

Notice that by construction of $I_{\rm lin}(V)$, the only element $a$ in $ V$ such that $(a)\in I_{\rm lin}(V)$ is $a=0$.
\end{rk}
 \begin{proof} We give two proofs, the first of which does not use Zorn's lemma. 
 	
 	Let us consider the map
\begin{align*}
\theta&:\left\{\begin{array}{rcl}
V&\longrightarrow&\K(V)/{I_{\rm lin}(V)}\\
a&\longmapsto&[a].
\end{array}\right.
\end{align*}
By definition of ${I_{\rm lin}(V)}$, for any $a,b\in V$, for any $k\in \K$,
\[  
[a+b] =[a]+[b],\quad [ka]=k[a],
\]
which implies that $\theta$ is linear. Let us consider the linear map
\[\overline{\theta'}:\left\{\begin{array}{rcl}
\K(V)&\longrightarrow&V\\
(a)&\longmapsto&a.
\end{array}\right.\]
For any $a,b\in V$, for any $k\in \K$,
\begin{align*}
\overline{\theta'}((a+b)-(a)-(b))&=a+b-a-b=0,\\
\overline{\theta'}((ka)-k(a))&=ka-ka=0,
\end{align*}
so ${I_{\rm lin}(V)}\subseteq \ker(\overline{\theta'})$. Consequently, we obtain a linear map
\[\theta':\left\{\begin{array}{rcl}
\K(V)/{I_{\rm lin}(V)}&\longrightarrow&V\\
\:[a]&\longmapsto&a.
\end{array}\right.\]
For any $a\in V$, $\theta'\circ \theta(a)=a$, so $\theta'\circ \theta=\mathrm{id}_V$. For any $a\in V$,
$\theta\circ \theta'([a])=[a]$. As $\theta\circ \theta'$ is linear and the elements $[a]$ generate $\K(V)/{I_{\rm lin}(V)}$,
$\theta\circ \theta'=\mathrm{id}_{\K(V)/{I_{\rm lin}(V)}}$. 
 \vskip 0,5cm 
\sy{ Here is an alternative proof which uses Zorn's lemma.}

 {We define the equivalence relation $\sim$ on $\K(V)$ as $x\sim y$ if, and only if, there is an element $w$ in $ I_{\rm lin}(V)$ such that $x+w=y$.} We show that the linear   map $V\mapsto \sfrac{\K(V)}{I_{\rm lin}(V)}$, which sends    the basis     $\{v_i\}_{i\in I}$ of $V$ to the family $\{[v_i]\}_{i\in I}$  of vectors in $ \sfrac{\K(V)}{I_{\rm lin}(V)}$  is bijective. For this purpose we need to show that $\{[v_i]\}_{i\in I}$   is a basis of $ \sfrac{\K(V)}{I_{\rm lin}(V)}$. Since $\{v_i\}_{i\in I}$ generates $V$,  for any $a\in V$, there are $\alpha_i$'s in $\K$ such that $a=\sum_{i\in I}\alpha_iv_i$, then $(a)\sim\sum_{i\in I}\alpha_i(v_i)$ (using equations (\ref{eq:1tuples1}) and (\ref{eq:1tuples2})) which implies $[a]=\sum_{i\in I}\alpha_i[v_i]$. Since the elements $(a)$, with $a\in V$ generate $\K(V)$, it therefore follows that $\{[v_i]\}_{i\in I}$ generates $\sfrac{\K(V)}{I_{\rm lin}(V)}$. We show now that $\{[v_i]\}_{i\in I}$ is free. Let $\sum_{i\in I}\alpha_i[v_i]=0$, this is equivalent to $\sum_{i\in I}\alpha_i(v_i)\sim(\sum_{i\in I}\alpha_iv_i)\in I_{\rm lin}(V)$, thus $\sum_{i\in I}\alpha_iv_i=0$. Since $\{v_i\}_{i\in I}$ is a basis of $V$, this implies that for all $i\in I$, $\alpha_i=0$ so that $\{[v_i]\}_{i\in I}$ is a basis of $V\simeq\sfrac{\K(V)}{I_{\rm lin}(V)}$ as announced.
\end{proof}

If $(V,\top)$ is a locality vector space, $\K(V)$ can be endowed with the locality relation $\top_{\times 1,1}$ (see \ref{eq:locreltimesmn}),  {which leads to an isomorphism of locality vector spaces, which generalises the isomorphism of Lemma \ref{lem:Visomorphism}.}

\begin{prop}\label{prop:Visomorphism2}
Let $(V,\top)$ be a locality vector space, then $(V,\top)$ and $ \left(\sfrac{\K(V)}{I_{\rm lin}(V)},\ttop\right)$ are isomorphic as locality vector spaces:
\begin{equation}\label{eq:lociso} {   (V,\top)\simeq\left(\sfrac{\K(V)}{I_{\rm lin}(V)},\ttop\right), }\end{equation} where $\ttop$ is the quotient locality (\textit{See Definition \ref{defn:quotientlocality}}).
\end{prop}
\begin{proof}  We need to show that the isomorphism   of Lemma \ref{lem:Visomorphism}  is a locality map as well as its inverse. Let $(x,y)\in \top$,  then $((x),(y))\in\top_{\times 1,1}$ and thus $([x],[y])\in\ttop$. Conversely if $([x],[y])\in\ttop$, this implies that there exist $\sum_{i\in I}\alpha_i(x_i)\in[x]$ and $\sum_{j\in J}\beta_j(y_j)\in[y]$ such that $\left(\sum_{i\in I}\alpha_i(x_i),\sum_{j\in J}\beta_j(y_j)\right)\in\top_{\times1,1}$ or equivalently $(x_i,y_j)\in \top$ for every $(i,j)\in I\times J$. Since $(V,\top)$ is a locality vector space the latter implies that $\left(\sum_{i\in I}\alpha_ix_i,\sum_{j\in J}\beta_jy_j\right)\in\top$. And finally, by means of Lemma \ref{lem:Visomorphism}, $\left[\sum_{i\in I}\alpha_ix_i\right]=[x]$ and $\left[\sum_{j\in J}\beta_jy_j\right]=[y]$ implies $\sum_{i\in I}\alpha_ix_i=x$ and $\sum_{j\in J}\beta_jy_j=y$, thus $(x,y)\in \top$ as  announced.
\end{proof}

 We note that the proof of the above proposition uses the locality property of the vector space $V$. The following counterexample shows that the isomorphism is not necessarily a locality morphism   when $V$ is only a  pre-locality vector space. 
\begin{coex}
Let $V=\R^2$ and $\top=\{(e_1,e_2),(e_1,3e_2),(e_2,e_1),(3e_2,e_1)\}$. Then $(V,\top)$ and $\left(\sfrac{\K(V)}{I_{\rm lin}(V)},\ttop\right)$ are not isomorphic as locality vector spaces, where $\ttop$ is the quotient locality. Indeed $((3e_2)-(e_2),(e_1))\in\top_{\times1,1}$ and thus $([3e_2-e_2]=[2e_2],[e_1])\in\ttop$, but $(2e_2,e_1)\notin \top$.
\end{coex}

\section{ Sufficient conditions for the locality of quotient vector spaces} \label{sec:main_conjectures}

In this section we discuss and compare sufficient conditions  for the   quotient of two vector spaces to be a locality vector space.
\subsection{Split locality exact sequences and  locality quotients }
 
We start by giving a first sufficient (but not necessary) condition to answer Question \ref{question}. 
\begin{defnprop}\label{defnprop:stronglocality} 

	We call a {\bf strong locality complement} of a subspace $W$ of a locality vector space $(V,\top)$, a complement subspace $\widetilde W$ of $W$  in $(V, \top)$ i.e. such that $W\oplus \widetilde W=V$,  which satisfies one the following equivalent properties:\begin{enumerate}
		\item   $\pi\top \pi$  and  $\pi\top \widetilde\pi$; 
		\item $ \pi\top {\rm Id}_V$; 
		\item $\widetilde\pi\top \widetilde \pi$ and $\pi\top \widetilde\pi$; 
		\item $\widetilde \pi\top{\rm Id}_V$; 
	\end{enumerate}
	where $\pi$ (resp. $\widetilde \pi$) is the projection onto $W$ (resp.  onto $\widetilde W$)  parallel to $\widetilde W$  (resp. parallel to $W$).
	
We call $W$ and $\widetilde W$ {\bf    strong locality complement subspaces} of the locality space $(V, \top)$.
\end{defnprop}
 \begin{rk} Note that any subspace $W$ of a   strongly non-degenerate locality vector space $(V, \top)$ discussed in \cite[Equation (30)]{CGPZ3} has an algebraic complement  given by its polar set $W^\top$ (or orthocomplement) so that $V=W\oplus W^\top$. However, unlike in the above definition,  in \cite{CGPZ3}   the projections $\pi: V\to W$ and $\tilde \pi: V\to W^\top$ were not required to be locality maps.\end{rk}
\begin{proof}
We use the fact that 
		\begin{equation}\label{eq:phi12}\left(\varphi\top \varphi_1\, {\rm and}\, \varphi \top \varphi_2\right)\Longrightarrow \left(\varphi\top   (\varphi_1\pm \varphi_2)\right) \end{equation}
		and \begin{equation}\label{eq:Idphi}\varphi\top {\rm Id}_V\Rightarrow \varphi\top\varphi.\end{equation} Indeed, by assumption $x\top y\Rightarrow \varphi (x)\top \varphi_i(y)$ for $i=1, 2$ and the locality of the vector space $(V, \top)$ implies that $\varphi (x)\top \varphi_1(y)\pm\varphi_2(y)$, which shows   (\ref{eq:phi12}).  Furthermore, if $\varphi\top {\rm Id}_V$ then $x\top y\Rightarrow \varphi(x)\top y$ which,  using the symmetry of $\top$ implies $\varphi(y)\top {\rm Id}_V(\varphi(y))$ and hence $\varphi(x)\top \varphi(y)$, which proves  (\ref{eq:Idphi}).  
 	\begin{itemize}
	\item $1)\Rightarrow 2)$  follows from (\ref{eq:phi12}) applied to $\varphi= \pi$ and $\varphi_1= \widetilde\pi, \varphi_2=  \pi$, which implies $\pi\top {\rm Id}_V\ $ since ${\rm Id}_V=\pi+\widetilde\pi$.  
	
		\item $2)\Rightarrow 1):$ Follows from  (\ref{eq:Idphi}) applied to $\varphi= \pi$ followed by (\ref{eq:phi12}) applied to $\varphi= \pi$ and $\varphi_1={\rm Id}_V$, $\varphi_2= \pi$ using the fact that $\widetilde \pi={\rm Id}_V-\pi$.
	
	\item $3)\Leftrightarrow 4):$ Analogous to the last two points exchanging the roles of $\widetilde\pi$ and $\pi$.
	
	\item $2)\Rightarrow4):$ Since ${\rm Id}_V\top{\rm Id}_V$, it follows from \eqref{eq:phi12} applied to $\varphi={\rm Id}_V$, $\varphi_1={\rm Id}_V$, and $\varphi_2= {\pi}$.
	\item $4)\Rightarrow 2):$ Analogous to the last point exchanging the roles of $\pi$ and $\widetilde\pi$.

	\end{itemize}	
	
	\end{proof} 
	\begin{rk}
     Notice that any two of the four points of Definition-Proposition \ref{defnprop:stronglocality} imply 
     \begin{equation*}
      \pi\top\pi ~\wedge~\tilde\pi\top\tilde\pi
     \end{equation*}
     which we will freely use later on.
     However, $\pi\top\pi ~\wedge~\tilde\pi\top\tilde\pi$ is not the fifth point of Definition-Proposition \ref{defnprop:stronglocality} since it does not imply any of the four points of the aformentionned Definition-Proposition.
    \end{rk}
	
	The previous definition can be expressed in terms of short split exact sequences. Indeed,  complement subspaces $W$ and $\widetilde W$ in $V$, i.e.  such that  $W\oplus\widetilde W=V$, give rise to a short exact sequence $0\to W\overset{\iota}{\to} V\overset{\widetilde\pi}{ \to} \widetilde W\to 0$, where $\iota$ is the inclusion map and $\widetilde\pi$ is the projection onto $\widetilde W$ along $W$.  The exact sequence is split since the projection $\pi$ onto $W$ along $\widetilde W$ yields a right  inverse for $\iota$. The equivalent conditions in Definition-Proposition \ref{defnprop:stronglocality}  lead to the following generalisation.
	  
\begin{cor}\label{cor:split_exact_seq}
Two complement linear subspaces $W$ and $\widetilde W$  of a locality space $(V, \top)$, i.e. such that $W\oplus \widetilde W=V$, are strong locality complements if and only if 
  the projection map $\widetilde \pi$ and  the right inverse $\pi$ of $\iota$  in the split short exact sequence 
 \begin{equation}\label{eq:splitexseq}  0\to W\underset{\pi}{\underset{\leftarrow}{ {\overset{\iota}{\to}}}} V\overset{\widetilde\pi}{ \to} \widetilde W\to 0\end{equation}  satisfy one of the equivalent conditions of  Definition \ref{defnprop:stronglocality}, where  ${\rm Id}_V$ corresponds to   the  sum $\pi+\widetilde \pi$. We call such sequence a {\bf locality split exact sequence}.  {Furthermore, in this case $(V/W,\ttop)$ is a locality vector space isomorphic to $(\widetilde W,\top\cap(\widetilde W\times\widetilde W))$}.
\end{cor}
\begin{rk}
    For $W$ a subspace of $(V,\top)$, Proposition \ref{prop:simple_consequence_question} shows the equivalence between
     \begin{enumerate}
     \item  whether a quotient $(V/W,\ttop)$ of locality vector spaces is a locality vector space, and
     \item the existence of a locality vector space $(\widetilde W,\widetilde\top)$ which makes the short exact sequence \eqref{eq:splitexseq} a locality short exact sequence, where $\widetilde\top$ is the final locality relation for $\widetilde\pi$. 
     \end{enumerate}
     
     Corollary \ref{cor:split_exact_seq} is a specialisation of Proposition \ref{prop:simple_consequence_question} in so far as it requests the existence of a strong locality complement $\widetilde W$ of $W$ in $(V,\top)$, and yields a locality split exact sequence.
     
    \end{rk}

\begin{proof}
By definition strong locality complements are equivalent to the existence of a split short exact sequence \[  0\to W\underset{\pi}{\underset{\leftarrow}{ {\overset{\iota}{\to}}}} V\overset{\widetilde\pi}{ \to} \widetilde W\to 0\] such that $\pi$ and $\tilde\pi$ satisfy one of the four equivalent conditions of Proposition-Definition \ref{defnprop:stronglocality}.

In this case we have $(V/W,\ttop)\simeq(\widetilde W,\top\cap(\widetilde W\times\widetilde W))$ by Proposition \ref{prop:simple_consequence_question}
\end{proof}

\begin{rk}Recall from   \cite[Definition 6.2]{CGPZ4}, that a   locality relation $\top$ on $V$ is non degenerate if $v\top v\Longrightarrow v=0$ and    called strongly non-degenerate if moreover for any subspace  $U\subsetneq  V$, the
  polar space $U^\top$ is nonzero.  A strongly non degenerate locality  relation $\top$ on $V$ yields the splitting $V=W\oplus W^\top$ for any subspace $W$ of $V$, so  $W$ has a complement space $\widetilde W=W^\top$.  The corresponding projections $\pi: V\to W$, resp. $\widetilde \pi : V\to W^\top$  are locally independent of each other by construction.  By  \cite[Proposition 6.5]{CGPZ4}, a  strongly separating locality vector space of countable dimension  admits  a locality basis  for $\top$, in which case the projections can be shown to be locality maps.
	
\end{rk}
 \subsection{A weaker condition ensuring the locality of quotient spaces}
 
 This paragraph provides a weaker sufficient condition to answer (\ref{question})  positively. This new condition is also more computational friendly than the a priori more natural condition  (for a subspace to have a strong locality complement) given in the previous paragraph, which is why we   use  it in the sequel.
 \begin{defn}\label{defn:loccompatible}
 	Let $(V,\top)$ be a locality vector space and $W\subset V$ a linear subspace. We say that {\bf $W$ is locality compatible with $\top$} if
 	$\forall (x,y,z)\in V^3,~\forall w\in W,~$ \begin{equation}\label{eq:loccomp}
 	(x,y)\in\top\, \wedge (x+w,z)\in \top\Longrightarrow(\exists w'\in W): (x+w',y)\in\top\,\wedge (x+w',z)\in \top. \end{equation} 
 	
 	We shall see (Theorem \ref{thm:quotientlocality-vs}) that if $W$ is locality compatible with $\top$ then the condition (\ref{eq:questionbis}) is fulfilled so that the quotient $V/W$ is a locality vector space. 
 \end{defn}
 Here are first rather trivial examples.
 \begin{ex}
 	Let $(V,\top)$ be a locality vector space. Then $V$ and $\{0\}$ are locality compatible with $\top$. Indeed,  to see that $V$ is locality compatible with $\top$ it is enough to consider $w'=-x$ using the notations of Definition \ref{defn:loccompatible}. To  see that $\{0\}$ is locality compatible with $\top$ is enough to notice that $w=w'=0$ and thus $x\top y$ and $x\top z$.
 \end{ex}
	\begin{prop}\label{prop:strongloccomp}
If a  subspace  $W$ of $(V, \top)$   admits a strong locality complement, then  it is locality compatible with $\top$.  
\end{prop}
\begin{proof}With the notations of  (\ref{eq:loccomp}), since $\widetilde\pi$ is locality independent of ${\rm Id}_V$,
 \[(x,y)\in\top\, \wedge (x+w,z)\in \top\Longrightarrow (\widetilde\pi(x),y)\in\top\,\wedge (\widetilde\pi(x),z)\in \top.\]
 Setting $ -w':=x-\widetilde\pi(x)$ which lies in $W$, we have $(x+w'=\widetilde\pi(x),y)\in\top$ and $(x+w'=\widetilde\pi(x),z)\in \top$ as expected.

	\end{proof}

The following counterexample shows that for a locality vector space $(V,\top)$ and a subspace $W$, the condition of $W$ having a strong locality complement is stronger than the condition of $W$ being locality compatible with $\top$.

\begin{coex}
Consider the vector space $\R^7$, its subspace $W=\lan \{e_1,e_2,e_3\}\ran$ where $\{e_i\}_{i=1}^7$ are the elements of the canonical basis, and \begin{align*}\top=&\R^7\times\{0\}\bigcup\lan e_1+e_7\ran\times\lan \{e_4,e_5\}\ran\bigcup\lan e_2+e_7\ran\times\lan \{e_5,e_6\}\ran\bigcup\lan e_3+e_7\ran\times\lan\{e_4,e_6\}\ran\\
&\bigcup\lan\{e_1+e_7,e_3+e_7\}\ran\times\lan e_4\ran\bigcup\lan\{e_1+e_7,e_2+e_7\}\ran\times\lan e_5\ran\bigcup\lan\{e_2+e_7,e_3+e_7\}\ran\times\lan e_6\ran\bigcup {\rm Sym.}\, {\rm  terms}.\end{align*}

Notice that $\top$ is invariant under the natural action of the subgroup $\Omega:=\langle\sigma_1,\sigma_2\rangle$ of the group of permutations $S_7$ generated by $\sigma_1:=(1,2)(4,6)$ and $\sigma_2:=(1,3)(5,6)$, where $(i,j)$ stands for the transposition of $i$ and $j$. $\Omega$ has six elements:
\[\Omega=\{\sigma_1:=(1,2)(4,6),\,\sigma_2:=(1,3)(5,6),\,\sigma_3=(2,3)(4,5),\,\sigma_4=(1,2,3)(4,5,6),\,\sigma_5=(3,2,1)(6,5,4),\,{\rm Id}_7\}.\]  
This follows from the relations $\sigma_3=\sigma_1\circ\sigma_2\circ\sigma_1$, $\sigma_4=\sigma_1\circ\sigma_2$, $\sigma_5=\sigma_2\circ\sigma_1$, $\sigma_5^2=\sigma_4$ and $\sigma_4^2=\sigma_5$, combined with the involutivity of the transpositions. 

One can see that $W$ is locality compatible with $\top$ by checking all cases. For instance, if $k\in\R$, 
$k(e_1+e_7)\top e_4$ and $k(e_1+e_7)+k(e_2-e_1)=k(e_2+e_7)\top e_6$, there is $k(e_3+e_7)=k(e_1+e_7)+k(e_3-e_1)$ such that $k(e_3+e_7)\top e_4$ and $k(e_3+e_7)\top e_6$. In terms of equation \eqref{eq:loccomp}, $x=k(e_1+e_7)$, $w=k(e_2-e_1)$, $y=e_4$, $z=e_6$, and $w'=k(e_3-e_1)$. 
Another possible case is when $x=k(e_1+e_7)+q(e_3+e_7)$ for $k,q\in\K$, then $x\top e_4$. If $w=q(-e_3+e_2)$, then $x+w=k(e_1+e_7)+q(e_2+e_t)\top e_5$. In this case $w'=q(-e_3+e_1)$ makes $x+w'=(k+q)(e_1+e_7)$ locality independent to both $e_4$ and $e_5$. All other possible cases are analogous, in the sense that  they are obtained from the previous two via the action of a permutation in $\Omega$ on the subindices of the $e_i$'s.

We show that $W$ has no strong locality complement by proving there is no projection $\pi:\R^7\to W$ such that $\pi\top {\rm Id}_{\R^7}$. Indeed, if there were such projection, then $\pi(e_1+e_7)\top e_4$, but $\pi(e_1+e_7)=e_1+\pi(e_7)$ where $\pi(e_7)\in W$. From the construction of $\top$, the only option is $\pi(e_7)=-e_1$. On the other hand, $(e_3+e_7)\top e_4$ but $\pi(e_3+e_7)=e_3-e_1$ is not locality independent to $e_4$ which yields the contradiction. 
\end{coex}
	
 The following Proposition yields an example of a locality compatible subspace for a vector space freely generated by another locality vector space.
\begin{prop}
Let $(V,\top)$ be a locality vector space, then $I_{\rm lin}(V)\subseteq \K(V)$ (\emph{see} Lemma \ref{lem:Visomorphism}) is locality compatible with $\top_{\times 1,1}$.
\end{prop}
\begin{proof}
 Let $(x,y,z)\in\K(V)^3$, with  {$x=\sum_{i\in I}\alpha_i(x_i)$,}  $y=\sum_{j\in J}\beta_j(y_j)$ and $z=\sum_{k\in K}\gamma_k(z_k)$  {(where, as before, we use   brackets to distinguish   elements of $V$ from elements of $\K V$)}, and $w\in I_{\rm lin}(V)$ such that $(x,y)\in\top_{\times1,1}$ and $(x+w,z)\in\top_{\times1,1}.$ We search for an element $w'$ in $I_{\rm  lin}(V)$ such that  $(x'+w',y)\in\top_{\times1,1}$ and $(x+w,z)\in\top_{\times1,1}.$ 
 

Since  $(x,y)\in\top_{\times1,1}$ and $(x+w,z)\in\top_{\times1,1},$ we have $([x],[y])\in\ttop$ and $([x],[z])\in\ttop$.     {Take $x_V$ to be the preimage in $V$ of $[x]$ under the isomorphism of} Proposition \ref{prop:Visomorphism2}.  {Since $[x]=[x+w]$ and since the isomorphism of Proposition \ref{prop:Visomorphism2} is a locality isomorphism, we have that $x_V$ is independent of the preimages of $y$ and $z$. Thus, writing}
  {$x_V=\sum_{i\in I}\alpha_ix_i$, it follows that $x_V\top y_j$ and $x_V\top z_k$ for every $j\in J$ and every $k\in K$. In particular} $\left(x_V,\sum_{j\in J}\beta_j(y_j)\right)\in\top$ and $\left(x_V,\sum_{k\in K}\gamma_k(z_k)\right)\in\top$. Hence, $((x_V),y)\in\top_{\times1,1}$ and $((x_V),z)\in\top_{\times1,1}$. {We claim that the difference $  {(x_V)-x}$ is a candidate for $w'$.}  {Indeed, since $x_V$ was the preimage of $x$ by the isomorphism of Proposition \ref{prop:Visomorphism2}, we have $w'= (x_V)-x\in I_{\rm lin}(V)$ as required.} 
\end{proof}

As a result of the above constructions we get that the quotient by a locality compatible subspace yields  a locality quotient, a fact on which most of the forthcoming constructions will rely.
\begin{thm}\label{thm:quotientlocality-vs}
Let $(V,\top)$ be a locality vector space and $W\subset V$ a subspace locality compatible with $\top$, then $(V/W,\ttop)$ is a locality vector space, where $\ttop$ is the quotient locality (see Definition \ref{defn:quotientlocality}).
\end{thm}
\begin{proof} Given any subset $U$ of $ V/W$, we want to prove that $U^{\ttop}$ is a linear subspace. For any $[x]\in U^{\ttop}$, for every $[u]\in U$,    there is an element $u'$ of $[u]$ and an element $x'$of $[x]$, such that $(u',x')\in\top$. The fact that $V$ is a linear vector space implies that for $\lambda\in\K$, $(u',\lambda x')\in \top$, therefore $\lambda[x]\in U^{\ttop}$.  

On the other hand, for $[x]$ and $[y]$ equivalence classes in $U^{\ttop}$ and for every $[u]\in U$, there are $u'\in [u]$ and $w\in W$ such that $(x',u')\in\top$ and $(y',u'+w)\in\top$ for some $x'\in[x]$ and some $y'\in[y]$. Since $W$ is compatible with $\top$, there is a $w'\in W$ such that $(x',u'+w')\in\top$ and $(y',u'+w')\in\top$.  Since $(V,\top)$ is a locality vector space, it follows that $(x'+y',u'+w')\in\top$ and therefore $[x]+[y]\in U^{\ttop}$. Hence $V/W$ is a locality vector space.  
\end{proof}

 The following counterexample shows that  locality compatibility is not necessary to have a local quotient space (compare with  Example \ref{ex:quotientlocHilbert}  and Remark \ref{rk:quotientlocHilbert})).
 
\begin{coex}
  Take $V$ be any Hilbert space with scalar product $\langle,\rangle$.  We equip $V$ with the  locality relation $\top$ given by the orthogonality relation: $v\top v':\Longleftrightarrow \langle v, v'\rangle=0$. Then a closed linear subspace $  \{0\}\subsetneq W\subsetneq V$  is not locality compatible with $\top$.
\end{coex}
\begin{proof}
 Since $W\neq V$ and $W\neq\{0\}$,  we can choose vectors  $w^\perp\in W^\perp\setminus\{0\}$,  $w\in W\setminus\{0\}$ and we set 
 \begin{equation*}
  x:=\frac{w^\perp}{||w^\perp||},\quad y:=\frac{w}{||w||}=:w_0\in W,\quad z:=\frac{w}{||w||}-\frac{w^\perp}{||w^\perp||}.
 \end{equation*}
 We have $\langle x,y\rangle=\langle x+w_0,z\rangle=0$. However there is no $w_1\in W$ such that $\langle x+w_1,y\rangle=\langle x+w_1,z\rangle=0$. Indeed $\langle x+w_1,y\rangle=0$ implies $\langle w_1,y\rangle=0$. But since $z=y-x$, we have $\langle x+w_1,z\rangle=0$
  which implies $\langle w^\perp,w^\perp\rangle=0$ leading to   a contradiction.
\end{proof}

 In the case of a vector space freely generated by a locality set, the locality compatibility property actually implies a stronger version with $N\geq2$ elements instead of two:
\begin{prop}
Let $(S,\top)$ be a locality set, $(\K(S),\top)$ the locality vector space generated by extending linearly the relation $\top$, and $W\subset \K(S)$ a subspace locality compatible with $\top$. Let $x\in \K(S)$, $y_i\in \K(S)$ and $w_i\in W$ for $1\leq i\leq N$, such that for every $i$ $(x+w_i,y_i)\in\top$. Then there is an element $w'$ of $ W$ such that $(x+w',y_i)\in\top$ for every $1\leq i\leq N$.
\end{prop}
\begin{proof}
We prove the statement by induction on  $N$. The case $N=2$ is immediate since $W$ is locality compatible with $\top$. Assume it is true for $N-1$, and let $x\in \K(S)$, $y_i\in \K(S)$ and $w_i\in W$ for $1\leq i\leq N$, such that for every $i$ $(x+w_i,y_i)\in\top$. By induction there is a $w'_0\in W$ such that $(x+w'_0,y_i)\in \top$ for every $1\leq i\leq N-1$. We can write every $y_i$ in terms of the basis elements of $S$ as $y_i=\sum_{j\in  S}\alpha^i_js_j$ where only finitely many $\alpha^i_j\neq0$, and define $\bar y=\sum_{j\in S}M_j s_j$ where \[M_j=\begin{cases} 0 & {\rm if }\quad (\forall i\in[N-1]):\, \alpha^i_j=0\\1&{\rm if}\quad (\exists i\in [N-1]): \, \alpha^i_j\neq 0 \end{cases}\] 
Notice that $(x+w_0',\bar y)\in \top$ and moreover, if there is a $z\in\K(S)$ such that $(z,\bar y)\in \top$ then $(z,y_i)\in \top$ for every $i<N$. 

Since $W$ is locality compatible with $\top$, there is an element $w'$ in $ W$ such that $(x+w',\bar y)\in\top$  and $(x+w',y_N)\in \top$ which implies the expected result.
\end{proof}

\subsection{Two  conjectural statements}

 We now formulate two conjectural statements which will play an important role in the sequel. The first one is the conjectural statement \ref{conj:main1} which gives a sufficient condition for the tensor product of $n$ subspaces of a locality vector space to be again a locality vector space. The second one is the conjectural statement \ref{conj:main2} which enhances  the tensor algebra to  a locality algebra by giving sufficient conditions for the filtered components from Definition \ref{defn:filteredtensoralg} to become locality vector spaces.
In the following  $V$ and  $W$ are subspaces of a pre-locality vector space $(E,\top)$. We equip $ V\times_\top W$   with a locality relation $\top_{V\times_\top W} $ as in (\ref{eq:locreltoptimes}).
 
\begin{rk} 
 If $V=W$, we have    $ \top_{V\times_\top W} = V^{\times_{\top}4}$ with the notations of (\ref{eq:Vtopn}).
\end{rk}
 
As described in \eqref{eq:locrelonthefreespan},  $\top_{V\times_\top W}$ linearly extends to a locality relation on $\K(V\times_\top W)$, which we denote by the same symbol $ \top_{V\times_\top W}$. This induces a locality relation $\top_\otimes$ on $V\otimes_\top W$, (see Definition \ref{defn:quotientlocality}). Recall that the locality tensor product was defined in Definition   \ref{defn:loctensprod1}
as
\begin{equation*}
 V\otimes_\top W=\sfrac{\K( V\times_\top W)}{I_{\rm bil}\cap\K( V\times_\top W)}.
\end{equation*}
 We defined  the locality tensor product (see Definition \ref{defn:loctensprod1}) by means of  the quotient of a locality vector 
	space by a subspace of multilinear forms, and provided  a sufficient condition  called "locality compatibility" on the subspace  for the quotient to be a locality space. Whether or not the various subspaces of multilinear forms involved are locality compatible, are challenging questions which we formulate as conjectural statements. 
\begin{conj} \label{conj:main1} [Pair of locality vector spaces]
{\it  Given a locality vector space $( {E},\top)$ and \lf{$V_1$, $V_2$ two of its subspaces},  
the subspace  $I_{\rm bil} {(V_1, V_2)} {\cap\K(V_1\times_\top V_2)}\subset \K(V_1\times_\top V_2)$ is locality compatible with $\top_{V_1\times_\top V_2}$.}
\end{conj}

 {\begin{prop}\label{prop:Conjecturefornbigger}
Statement \ref{conj:main1} is equivalent to the following statement:
 
 Let $n\geq 2$ and $V_1,\dots,V_n$ be linear subspaces of a locality vector space $( {E},\top)$. The space $I_{\rm mult, n}(V_1,\cdots,V_n)\cap\K(V_1\times_\top\cdots\times_\top V_n)\subset\K(V_1\times_\top\cdots\times_\top V_n)$ is locality compatible with the locality relation $\top_{\times n,n}$ (\textit{see \ref{eq:locreltimesmn}}) on the space $\K(V_1\times_\top\cdots\times_\top V_n)$.
\end{prop} 
\begin{proof}

We can recover Statement \ref{conj:main1} from the statement  of Proposition \ref{prop:Conjecturefornbigger} in  the case $n=2$. Conversely, we  now prove that  Statenent \ref{conj:main1} implies that of Proposition \ref{prop:Conjecturefornbigger}.  

For $n\geq3$: let $x,y,z$ and $w$ be elements of $\K(V_1\times\cdots\times V_{n})$ where $w\in I_{\rm mult,n}(V_1\times\cdots\times V_{n})$ such that $(x,y)\in\top_{\times n,n}$ and $(x+w,z)\in\top_{\times n,n}$. Extending linearly the usual bijection between $(V_1\times\cdots\times V_{n-1})\times V_n$ and $V_1\times\cdots\times V_n$, one obtains an isomorphism of vector spaces $\K((V_1\times\cdots\times V_{n-1})\times V_n)\simeq\K(V_1\times\cdots\times V_{n})$. It is straightforward to show that the restriction of such an isomorphism to $\K((V_1\times_\top\cdots\times_\top V_{n-1})\times_\top V_n)$ (resp. $I_{\rm bil}(V_1\times\cdots\times V_{n-1},V_n)$) yields an isomorphism with $\K(V_1\times_\top\cdots\times_\top V_{n})$ (resp. $I_{\rm mult,n}(V_1\times\cdots\times V_{n})$). We may therefore see $x,y,z$ and $w$ as elements on $\K((V_1\times_\top\cdots\times_\top V_{n-1})\times_\top V_n)$ as well as view $w$ as an element  of $I_{\rm bil}(V_1\times\cdots\times V_{n-1},V_n)$. 
Assuming Statement  (\ref{conj:main1}) holds, then yields the existence of $w'\in I_{\rm bil}(V_1\times\cdots\times V_{n-1},V_n)\cap\K((V_1\times_\top\cdots\times_\top V_{n-1})\times_\top V_n)$ such that $(x+w',y)\in\top_{\times n,n}$ and $(x+w',z)\in\top_{\times n,n}$, which proves the statement.
\end{proof}

The following Corollary is important to ensure the stability of locality vector spaces stable under tensor products and is the main reason for introducing the conjectural statement \ref{conj:main1}.

 {\begin{cor}\label{cor:ntensorproduct}
 Let $n\geq 1$ and $V_1,\dots,V_n$ be linear subspaces of a locality vector space $(V,\top)$. Assuming the statement \ref{conj:main1} holds true, the locality tensor product $(V_1\otimes_\top\cdots\otimes_\top V_n,\top_\otimes)$ is a locality vector space.
\end{cor}
\begin{proof}
This statement follows from Proposition \ref{prop:Conjecturefornbigger} and Theorem \ref{thm:quotientlocality-vs}.
\end{proof}}

Even though the last statement is essential for the rest of the paper, it is not enough to make the locality tensor algebra a locality vector space, since it fails to relate the different graded components. Therefore we formulate the following conjectural statement whose consequences will be used in Section  \ref{sec:Firstconsequences}.
}

\begin{conj} \label{conj:main2}[Locality tensor algebra] {\it   Given a locality vector space $(V,\top_V)$ and any $n\in\N$, the subspace $\left(I_{\rm mult}^n {(V)} {\cap \K(\bigcup_{k=0}^nV^{ {\times_\top^k}})}\right)\subset \K(\bigcup_{k=0}^nV^{ {\times_\top^k}})$ is locality compatible with $\top_\times^n$. (See Definition-Proposition \ref{defnprop:locrelonfilteredtenalg}).} 
	\end{conj}
	
	 {Each of the statements \ref{conj:main1} and \ref{conj:main2} implies the following useful property:
		
	\begin{prop}\label{conj:main3}[Locality vector space]
(the case $V=W$) {\it  Given a locality vector space  $(V,\top_V)$ and assuming the statement \ref{conj:main1} holds true, the subspace $I_{\rm bil} {(V)}\subset \K(V\times_\top V)$ is locality compatible with $\top_{V\times_\top V}$.}
 
\end{prop}}
Even though they might seem rather natural, these conjectural statements turn out to be rather challenging. We  devote the following paragraphs to getting  a better grasp of these assumptions and their consequences.

\section{First consequences: enhanced universal properties}\label{sec:Firstconsequences}

We first show how assuming that statements \ref{conj:main1} and \ref{conj:main2}  hold true, enables us to enhance  the results of Part \ref{part:one} to a full-fledged locality setup.
We transpose the results of  {Sections \ref{subsec:LTPofPL}, \ref{sect:lta} and \ref{subsec:luea} (in particular Theorem \ref{thm:univpropltp-preloc}, Proposition \ref{prop:tensprodaslocmap}, Theorem \ref{thm:univpptyloctenalg-preloc} and Theorem \ref{thm:univproplocunivenvalg-preloc})} to the  locality framework. We do not dwell on  the proofs that are similar to the usual setup, and instead  we focus on results whose proofs require a treatment that differs from the pre-locality context.

\subsection{The locality tensor algebra}

	 The following proposition is in  the same spirit as Theorem \ref{thm:quotientlocality-vs}.

	\begin{prop}\label{prop:quotientlocalityalgebra}
	Let $(A,\top_A,m)$ be a non-unital locality algebra and $I\subset A$ a locality ideal of $A$ which is locality compatible with $\top_A$. Then $(A/I,\ttop,\bar m)$ is a non-unital locality algebra where $\ttop$ is the quotient locality (See Definition \ref{defn:quotientlocality}).
	\end{prop}
	\begin{proof}
	By means of Theorem \ref{thm:quotientlocality-vs} $(A/I,\ttop)$ has the structure of a locality vector space. Analogous to the usual (non locality) set up, the induced product $\bar m:A/I\times_{\ttop}A/I\to A/I$ is an associative $\top_\times-$bilinear map. We are left to prove that for any $U\subset A/I$, $\bar m(U^\ttop\times_{\ttop}U^\ttop)\subset U^\ttop$. Given $([x],[y])\in\ttop$ such that both $[x]$ and $[y]$ are in $U^\ttop$, consider also a $[u]\in U$. Then there are $x'\in[x]$, $y'\in[y]$, $u'\in[u]$, and $(w,w_1,w_2)\in I^3$ such that 
	\begin{equation}\label{eq:localg1}(x',y')\in\top_A,\end{equation}
	\begin{equation}\label{eq:localg2}(x'+w_1,u')\in\top_A,\end{equation}
	\begin{equation}\label{eq:localg3}(y'+w_2,u'+w)\in\top_A.\end{equation}
	
Since $I$ is locality compatible with $\top_A$, from \eqref{eq:localg1} and \eqref{eq:localg2} we conclude that there is $w_1'\in I$ such that 
	\begin{equation}\label{eq:localg4}(x'+w_1',y')\in\top_A,\ \ {\rm and}\end{equation}
	\begin{equation}\label{eq:localg5}(x'+w_1',u')\in\top_A.\end{equation}
From \eqref{eq:localg3} and \eqref{eq:localg4}, there is a $w_2'\in I$ such that
	\begin{equation}\label{eq:localg6}(y'+w_2',x'+w_1')\in\top_A,\ \ {\rm and}\end{equation}
	\begin{equation}\label{eq:localg7}(y'+w_2',u'+w)\in\top_A.\end{equation}
And finally from \eqref{eq:localg5} and \eqref{eq:localg7}, there is a $w'\in I$ such that
	\begin{equation}\label{eq:localg8}(x'+w_1',u'+w')\in\top_A,\ \ {\rm and}\end{equation}
	\begin{equation}\label{eq:localg9}(y'+w_2',u'+w')\in\top_A.\end{equation}
By \eqref{eq:localg6} $m(x'+w_1',y'+w_2')$ is well defined and the fact that $A$ is a locality algebra together with \eqref{eq:localg8} and \eqref{eq:localg9} imply that $(m(x'+w_1',y'+w_2'),u'+w')\in\top_A$. Hence $\bar m([x],[y])\in U^\ttop$.
	\end{proof}

	\begin{defn}\label{defn:localg}
		\begin{itemize}
			
			\item A {\bf graded locality algebra} is a locality algebra together with a sequence of vector spaces  $\{A_n\}_{n\in\N}$ called the grading, such that \[A=\bigoplus_{n\in\N}A_n,\quad m(A_p\otimes_{\top}A_q)\subset A_{p+q},\quad u(\K)\subset A_0.\]
			\item A {\bf filtered locality algebra} is a locality algebra together with a sequence of nested vector spaces  $A^0\subset A^1\subset\dots\subset A^n\subset\cdots$ called the filtration, such that \[A=\bigcup_{n\in\N}A^n,\quad m(A^p\otimes_{\top}A^q)\subset A^{p+q},\quad u(\K)\subset A^0.\]
			 {A {\bf graded (resp. filtered) locality Lie algebra} is a locality Lie algebra which is also a graded (resp. filtered) locality Lie algebra.}
		\end{itemize}
	\end{defn}
 {This is the locality counterpart of  Proposition \ref{prop:tensprodaslocmap}.}
%
	\begin{thm} \label{prop:tensor_alg_graded_loc}
			Assuming   {the conjectural statement \ref{conj:main2}} holds true, then the locality tensor algebra over  a locality vector space is  a graded locality algebra. 
	\end{thm}
	
	\begin{proof}
	Let $(V,\top_V)$ be a locality vector space, it is straightforward that the locality vector space $(\K(V^{\times_\top\infty}),\top_V)$, where $V^{\times_\top^\infty}:=\bigcup_{k\geq1}V^{\times_\top k}$, together with the concatenation product $m_c$ between locally independent elements is a locality algebra. The subspace $I_{\rm mult}$ is moreover a locality ideal. Therefore Proposition \ref{prop:quotientlocalityalgebra} 
 {implies $(\mathcal T_\top(V),\top_\otimes, \bigotimes)$ is a locality algebra. Since for any $p$ and $q$, the concatenation product $m_c$ preserves the grading of $\K(V^{\times_\top^\infty})$, namely $m_c(V^{\times_\top^p}\times_\top V^{\times_\top^q})\subset V^{\times_\top^{p+q}}$, it follows that $\bigotimes(V^{\otimes_\top^p}\otimes_\top V^{\otimes_\top^q})\subset V^{\otimes_\top^{p+q}}$. Then, the convention that $V^0=\K$ yields the result.}
	\end{proof}

	\subsection{The universal property of the locality tensor and universal enveloping algebras}
	
 Assuming the statements \ref{conj:main1} and \ref{conj:main2} hold true, we can enhance the previous universal properties we had on Part \ref{part:one} and introduce a new conjectural statement to enhance the universal property of the universal enveloping algebra. 

The following theorem generalises Theorem \ref{thm:univpropltp-preloc}.
	\begin{thm}[Universal property of the  locality tensor product]  \label{thm:univpropltp}  
		 {Given $V$ and $W$ linear subspaces of a locality vector space $(E,\top)$ over $\K$, $(G,\top_G)$ a locality vector space}
		and $f:(V\times_\top W,\top_\times)\to (G,\top_G)$ a locality  {$\top_\times$-}bilinear map. Assuming that conjectural statement \ref{conj:main1} holds true for the locality  {vector spaces $V$ and $W$}, there is a unique locality linear map $\phi:V\otimes_\top W\to G$ such that the following diagram commutes.
		\begin{equation*}
		\begin{tikzcd}
		(V\times_\top W,\top_\times)\arrow[r,"\otimes_\top"]\arrow[rdd,"f"]&(V\otimes_\top W,\top_\otimes) \arrow[dd,"\phi"]\\
		\\
		&(G,\top_G)
		\end{tikzcd}
		\end{equation*}
	\end{thm}
	\begin{proof}
		Theorem (\ref{thm:univpptyloctensprod-nolocalityrel}) yields the existence and uniqueness of the linear map $\phi$.
		 {Assuming the statement \ref{conj:main1} holds true, implies that $V\otimes_\top W$ is a locality vector space. }We are only left to show that $\phi$ is a locality map.  Recall that two equivalence classes $[a]$ and $[b]$ in $V\otimes_{\top}W$ verify $[a]\top_{\otimes}[b]$ if there are $\sum_{i=1}^n\alpha_i(x_i,y_i)\in[a]$ and $\sum_{j=1}^m\beta_j(u_j,v_j)\in[b]$ such that every possible pair taken from the set $\{x_i,y_i,u_j,v_j\}$ lies in $V\times_\top W$ for every $1\leq i\leq n$ and every $1\leq j\leq m$. Since $f$ is locality $\top_\times$ bilinear, then $f(\sum_{i=1}^n\alpha_i(x_i,y_i))\top_Af(\sum_{j=1}^m\beta_j(u_j,v_j))$ which amounts  to $\phi([a])\top_A\phi([b])$. Therefore $\phi$ is as expected.
	\end{proof}
	Assuming the conjectural statement  \ref{conj:main2} holds true, as a consequence of  the previous Theorem, we can   state and prove an enhanced universal property   {Theorem \ref{thm:univpptyloctenalg-preloc}}	 for  the locality tensor algebra.
	
	\begin{thm}[Universal property of locality tensor algebra]\label{thm:univpptyloctenalg}		
		Let $(V,\top)$ be a locality vector space, $(A,\top_A)$ a locality algebra and $f:V\to A$ a locality linear map. Assuming the conjectural statement \ref{conj:main2}  holds for tensors powers of $V$, there is a unique locality algebra morphism $\phi:{\mathcal T}_{\top}(V)\to A$ such that the following diagram commutes.
		\begin{equation*}
		\begin{tikzcd}
		(V,\top)\arrow[r,"\otimes_\top"]\arrow[rdd,"f"]&({\mathcal T}_\top(V),\top_\otimes) \arrow[dd,"\phi"]\\
		\\
		&(A,\top_A)
		\end{tikzcd}
		\end{equation*}
		
		where $\otimes:V\to {\mathcal T}_{\top}(V)$ is the canonical (locality) injection map.
	\end{thm}
	
	\begin{proof}
	The locality tensor algebra is a pre-locality algebra by Proposition \ref{prop:tensprodaslocmap}, and a locality algebra since we have assumed that the conjectural statement  \ref{conj:main2}  holds true. Thus it is a locality algebra.
	
	By means of Theorem \ref{thm:univpptyloctenalg-preloc}  {and Remark \ref{rk:partial-product-is-locality}} the pre-locality algebra morphism $\phi$ exists and is unique. Given that $\mathcal T_\top(V)$ and $A$ are locality algebras, then $\phi$ is also a locality algebra morphism as expected.
	\end{proof}
	We study the universal property of the locality universal enveloping algebra  $U_\top(\g)$ of a locality Lie algebra $\g$. In the remaining part of the paper, assuming  that  the conjectural statement \ref{conj:main2}  holds true, we make the following further assumption.
	
\begin{conj}\label{conj:univenvalg} {\bf   for the universal enveloping algebra:} given a locality Lie algebra $(\g,\top_\g,[,])$, the ideal $J_\top(\g)$ of $\mathcal T_\top(\g)$ introduced in Definition \ref{defn:luenvelopingalgebra}  is locality compatible with $\top_\otimes$. 
\end{conj}
	
\begin{prop}\label{prop:Utopgloc}
Assuming the conjectural statement \ref{conj:univenvalg} holds true for   a locality Lie algebra $(\g, \top_{\g}, [,])$,   then $U_{\top}(\g)$ defines a locality algebra.
\end{prop}
\begin{proof}
This is a direct consequence of   Proposition \ref{prop:quotientlocalityalgebra}.\end{proof}
	
This theorem is the locality counterpart of Theorem \ref{thm:univproplocunivenvalg-preloc}. 
	\begin{thm}\label{thm:univproplocunivenvalg}  
		Let $(\g, \top_{\g}, [,])$ be a locality Lie algebra, $(A,\top_A)$ a locality algebra and $f:\g\to A$ a locality Lie algebra morphism where the Lie bracket on $A$ is the commutator defined by the product. Assuming that the conjectural statements \ref{conj:main2} and  \ref{conj:univenvalg} hold true for $\g$, there is a unique locality algebra morphism $\phi:U_{\top}(\g)\to A$ such that the following diagram commutes, and \sy{where $\iota_\g: \g\longrightarrow \to U_{\top}(\g) $  was defined   in (\ref{eq:iotag})}.
		\begin{equation*}
		\begin{tikzcd}
		(\g,\top_\g)\arrow[r,"\iota_\g"]\arrow[rdd,"f"]&(U_\top(\g),\top_\g) \arrow[dd,"\phi"]\\
		\\
		&(A,\top_A)
		\end{tikzcd}
		\end{equation*}
		
	\end{thm}
	\begin{proof}
	The proof follows from Theorem \ref{thm:univproplocunivenvalg-preloc} and the fact that $U_\top(\g)$ and $A$ are locality algebras.
	\end{proof}
\vfill \eject \noindent 


\part{The Milnor-Moore theorem in the locality setup}

\section{ {Prerequisites on coalgebraic locality structures}} \label{section:seven}  
We review here  preliminary results we shall use for the actual proof of the locality version of the Milnor-Moore theorem.

\subsection{Graded connected locality Hopf algebras}

Let us   recall  some definitions from \cite{CGPZ1}. Just as a locality algebra was defined as a locality vector space equipped with a partial product  and a unit compatible with the locality relation (see Definition \ref{defn:prelocalg}), a locality coalgebra is a locality vector space equipped  with partial coproduct and a  counit  compatible with the locality relation.
\begin{defn}\label{defn:lochopfalg}\cite[Definition 4.3]{CGPZ1}
  A locality  $\K$-coalgebra is a quadruple $(C,\Delta,\epsilon,\top)$ where  $(C,\top)$ is a locality $\K$-vector space,  $\Delta:(C,\top)\to (C\otimes_{\top}C,\top_\otimes)$ and $\epsilon:C\to\K$
 are  linear maps such that
 \begin{itemize}	\item  the coproduct $\Delta$  is coassociative, namely $(Id_C\otimes\Delta)\circ \Delta=(\Delta\otimes Id_C)\circ \Delta$ on $C$;
  	\item and compatible with the locality structure i.e., for any $U\subset C$, $\Delta(U^{\top})\subset U^{\top}\otimes_{\top}U^{\top}$; 
  	\item  the counit $\epsilon: C\to \K$  satisfies $(Id_C\otimes\epsilon)\Delta=(\epsilon\otimes Id_C)\Delta=Id_C$. 
  	\end{itemize}
 \begin{rk} Note that the second condition tells us that $\Delta$ is a locality linear map in the sense of (\ref{eq:localf}) in Definition \ref{defn:basic_loc_defn}.
  	\end{rk}
 \end{defn}
 \begin{rk}
  The above definition only makes sense  in the locality framework, where   the polar set  $U^\top$ of a set is required  to be a vector space. In the pre-locality setup,  the fact that this condition is relaxed  prevents us from building locality tensor products such as $U^{\top}\otimes_{\top}U^{\top}$. This   suggests that the Milnor-Moore theorem we are about to prove does not hold in the more general pre-locality setup.
     \end{rk}
 We now refine this definition.
 \begin{defn}\label{defn:loccobialg}
 	\begin{itemize}
\item A {\bf graded locality $\K$-coalgebra} is a locality $\K$-coalgebra together with a sequence  of vector spaces $\{C_n\}_{n\in\N}$  called a {\bf grading}, such that \[C=\bigoplus_{n\in\N}C_n,\quad \Delta(C_n)\subset\bigoplus_{p+q=n}C_p\otimes_{\top}C_q, \quad \bigoplus_{n\geq 1}C_n\subset {\rm ker}(\epsilon).\]
We denote by $\vert x\vert$ the degree of $x$.

Moreover we call  a graded locality 
$\K$-coalgebra  {\bf connected} if $C_0$ has dimension $1$ and therefore $\bigoplus_{n\geq 1}C_n = {\rm ker}(\epsilon)$. 
\item A {\bf filtered locality $\K$-coalgebra} is a locality $\K$-coalgebra together with a nested sequence of vector spaces $C^0\subset C^1\subset \dots \subset C^n\dots$,  called a {\bf filtration},   such that \[C=\bigcup_{n\in\N}C^n,\quad \Delta(C^n)\subset\sum_{p+q=n}C^p\otimes_{\top}C^q,\quad  {\bigoplus_{n\geq1}C^n\subset\rm{ker}(\epsilon)}.\]
\end{itemize}
\end{defn}
The following definition provides  the coalgebraic counterpart of a locality ideal  see (\ref{eq:locideal}) and a locality morphism, see  (\ref{eq:localgmorp}).

 \begin{defn}
  \begin{enumerate}
 			\item A  locality linear subspace $J$ of a locality  $\K$-coalgebra $\left(C,\top ,\Delta  \right) $ is called a  left, (resp. right) {\bf locality coideal} of $C$, if 
 			\begin{equation}\label{eq:loccoideal}  
 			\Delta(J) \subset J\otimes_\top C; \, ({\rm resp.}\,    \Delta(J) \subset C\otimes_\top J)\quad \lf{{\rm and}\quad \epsilon(J)=(0)}.
 			\end{equation}
 	We call it a {\bf  locality coideal} if  
 	\begin{equation}\label{eq:locbicoideal}  
 	\Delta(J) \subset J\otimes_\top C+ C\otimes_\top J\quad \lf{{\rm and}\quad \epsilon(J)=(0).}
 	\end{equation} 
\sy{Note that the condition $\epsilon(J)=(0)$ does not involve the locality relation.}
 			\item Given two locality coalgebras $(C_i,\top_i,\Delta_i,\epsilon_i), i=1,2 $, a locality linear map $f: C_1\to C_2$ is called a {\bf locality $\K$-coalgebra morphism} if 
 			\begin{equation}\label{eq:loccoalgmorp}
 			(f\otimes f)\circ \Delta_1 = \Delta_2\circ f, \quad \cy{\text{and}\quad\epsilon_1=\epsilon_2\circ f.}
 			\end{equation}
 			In other words,  it  is   a $\K$-coalgebra morphism which is a locality map.
 	 \item Let $(C, \top, \Delta)$ be a locality $\K$-coalgebra. A locality $\K$-coalgebra $(C_1, \top_1, \Delta_1)$ with $C_1\subset C$ is a locality sub-coalgebra of $(C, \top, \Delta)$ if  the inclusion map $\iota: (C_1, \top_1, \Delta_1)\hookrightarrow (C, \top, \Delta)$ is a $\K$-coalgebra morphism.
 		\end{enumerate}
    \end{defn} 
    \begin{rk}
         This definition of sub-locality coalgebra is more general than the one given in \cite{CGPZ1}. This level of generalisation will be needed later.
        \end{rk}
    We first prove an elementary result of linear algebra.
    \begin{lem}\label{lem:kernelofSumnonloc}
For $i\in\{1,2\}$, let $f_i:V_i\to W_i$ be linear maps from a vector space $V_i$ to a vector space $W_i$. We have
\[\ker(f_1\otimes f_2)=\ker f_1\otimes V_2+V_1\otimes \ker f_2.\]
\end{lem}
\begin{proof}
Let $K_i=\ker(f_i)\subset V_i$ and let $X_i\subset V_i$ be any direct complement space in $V_i$ so 
  that $V_i=K_i\oplus X_i$. It follows from the distributivity property (\ref{eq:distributivity}) that   \[V_1\otimes V_2=(K_1\otimes K_2)\oplus (V_1\otimes K_2)\oplus (K_1\otimes X_2) \oplus (X_1\otimes X_2).\] 
As a consequence of the first isomorphism theorem for vector spaces there are isomorphisms $\phi_i:X_i\to {\rm Im}(f_i)$ from which we build a linear map $(\phi_1^{-1}\otimes\phi_2^{-1})\circ(f_1\otimes f_2)$  from $V_1\otimes V_2$ onto $X_1\otimes X_2$ which does not vanish on $X_1\otimes X_2$ outside the null tensor. 

Since $(K_1\otimes K_2)\oplus (V_1\otimes K_2)\oplus (K_1\otimes V_2)\subset \ker ((\phi_1^{-1}\otimes\phi_2^{-1})\circ(f_1\otimes f_2))$ and since by construction $(\phi_1^{-1}\otimes\phi_2^{-1})\circ(f_1\otimes f_2)$ does not vanish on $X_1\otimes X_2\setminus\{0\}$ (where $0$ is the null tensor), we have 
$\ker(((\phi_1^{-1}\otimes\phi_2^{-1})\circ(f_1\otimes f_2))=(K_1\otimes K_2)\oplus (V_1\otimes K_2)\oplus (K_1\otimes V_2)=\ker f_1\otimes V_2+V_1\otimes \ker f_2$. The result follows from the fact that $\phi_1\otimes\phi_2$ is  an isomorphism of vector spaces.
\end{proof}
    The following is the locality version of Lemma \ref{lem:kernelofSumnonloc}. As in Proposition \ref{prop:tproductanddirectsum}, we require a compatibility between the locality relation and  direct sums.
    \begin{prop}\label{prop:KernelofTensorPoduct}
   Let $(E,\top)$ and $(F,\top_F)$ be two locality vector spaces. For $i\in\{1,2\}$, let $f_i:V_i\to W_i$ be locality linear maps from  {locality subspaces} $V_i {\subset E}$ to  {vector subspaces} $W_i$ of $F$. We moreover assume that $f_1$ and $f_2$ are mutually  locally independent and the existence of  surjective projections $\pi_i:V_i\to \ker(f_i)$ 
    such that $\pi_i$ and ${\rm Id}_{V_j}$ are locally independent for $i\neq j$ 
    (See Proposition \ref{prop:tproductanddirectsum}). Then 
    \begin{equation} \label{eq:Kerf1otimesf2}
     \ker(f_1\otimes f_2)\cap (V_1\otimes_\top V_2)=\ker f_1\otimes_\top V_2+V_1\otimes_\top \ker f_2. 
    \end{equation}
    \end{prop} 
    \begin{proof} 
    We know from Lemma \ref{lem:kernelofSumnonloc} that 
   	$\ker(f_1\otimes f_2) =\ker f_1\otimes  V_2+V_1\otimes  \ker f_2$. Taking the intersection with $V_1\otimes_\top V_2$ yields
   	\[\ker(f_1\otimes f_2)\cap (V_1\otimes_\top V_2)=\left(\ker f_1\otimes  V_2+V_1\otimes  \ker f_2\right)\cap (V_1\otimes_\top V_2)\]
   	
   	
   	From elementary linear algebra, we obtain $V_1=\ker(f_1)\oplus {\ker}(\pi_1)$. Since by hypothesis $\pi_1$ and ${\rm Id}_{V_2}$ are locally independent and since $(E,\top)$ is a locality vector space, by the second point of Proposition \ref{prop:tproductanddirectsum}, the projection $\tilde\pi_1:V_1\mapsto {\ker}(\pi_1)$ onto $ {\ker}(\pi_1)$ along $\ker(f_1)$ is also independent of ${\rm Id}_{V_2}$. Thus we can use Corollary \ref{cor:Localitytensorintersection} with $\ker(f_1)$, $V_1$ and $V_2$ respectively playing the roles of $V_1$, $V$ and $W$ in Corollary \ref{cor:Localitytensorintersection}. This yields $\left(\ker f_1\otimes  V_2\right)\cap (V_1\otimes_\top V_2)=\ker(f_1)\otimes_\top V_2$. Similarly,
   	$\left(V_1\otimes  \ker f_2\right)\cap (V_1\otimes_\top V_2)=V_1\otimes_\top\ker(f_2)$ and hence
   	\[\ker(f_1\otimes f_2)\cap (V_1\otimes_\top V_2)=\left(\ker f_1\otimes  V_2+V_1\otimes  \ker f_2\right)\cap (V_1\otimes_\top V_2)= \ker f_1\otimes_\top  V_2+V_1\otimes_\top  \ker f_2. \qedhere\]
    \end{proof}
    
 	The following lemma is the coalgebraic counterpart of Lemma \ref{lem:Kerlocideal}.
 	\begin{lem}\label{lem:Kerloccoideal} 
 	Let  $(C_i, \Delta_i, \top_i), i\in \{1, 2\}$ be locality $\K$-coalgebras. The range of a locality $K$-coalgebra morphism $f: C_1\longrightarrow C_2$  is a locality $\K$-subcoalgebra of $C_2$. Moreover, if  {there is a} projection $\pi:C_1\to\ker(f_1)$, which is locally independent of  the identity map ${\rm Id}_{C_1}$  on $C_1$, then $\ker(f_1)$ is a locality coideal of $C_1$.
 	\end{lem} 
 \begin{proof}\begin{itemize}\item 
 We prove that the kernel ${\rm ker}(f)$ is a locality coideal.   Let $c\in {\rm ker}(f)\subset C_1 $. Since $f$ is a locality coalgebra morphism, by (\ref{eq:loccoalgmorp}), we have $(f\otimes f)(\Delta_1 c) =\Delta_2(f(c))=0.$ 
 	   Since $\Delta_1$ is a locality coproduct,   $\Delta_1( {\rm ker}(f))\subset \left(C_1\otimes_{\top_1}C_1\right)\cap {\rm ker}(f\otimes f)$. Thus $\Delta_1(\ker(f))\subseteq\ker(f\otimes f)$, which shows that $\ker(f_1)$ is a locality coideal of $C_1$.
 	   
 	  	We now apply Proposition \ref{prop:KernelofTensorPoduct} to $f_i=f$ and $V_i=C_1$, with $f\otimes f$ acting on $C_1\otimes_{\top_1}C_1$. Since $f$ is a locality morphism, it follows that ${\rm ker}(f \otimes f)\vert_{C_1\otimes_{\top_1}C_1 }= {\rm ker}(f)\otimes_{\top_1} C_1+   C_1\otimes_{\top_1} {\rm ker}(f).$  Consequently,   $\Delta_1\left(  {\rm ker}(f)\right)\subset{\rm ker}(f)\otimes_{\top_1} C_1+   C_1\otimes_{\top_1} {\rm ker}(f).$
 	 \item 
To prove that the range $ {\rm Im}(f)$ is a locality coalgebra, for any $c\in C_1$ such that 	$ \Delta_1 c= \sum_{(c)} c_1\otimes c_2$ and $(c_1, c_2) \in  \top_1$, using (\ref{eq:loccoalgmorp}) we write $ \Delta_2 f(c)=(f\otimes f) \circ \Delta_1 c= \sum_{(c)} f(c_1)\otimes f(c_2)$. Since $f$ is a locality map, $(f(c_1), f(c_2))\in  \top_2$, which proves that  $ \Delta_2\left({\rm Im}(f)\right)\subset  {\rm Im}(f)\otimes_{ {\top_2}} {\rm Im}(f) $, showing that ${\rm Im}(f)$ is a subcoalgebra of $C_2$. \qedhere
\end{itemize}
 \end{proof}
We are now ready to introduce further useful definitions.
\begin{defn}
	\begin{itemize}
\item \cite[\S 5.1]{CGPZ1}  A {\bf locality $\K$-bialgebra} is a sextuple $(B,\top,m,u,\Delta,\epsilon)$ consisting of a locality $\K$-algebra $(B,m,u,\top)$ and a locality $\K$-coalgebra 
$(B,\Delta,\epsilon,\top)$ that are locality compatible in the sense that $\Delta$ and $\epsilon$ are locality $\K$-algebra morphisms (\ref {eq:localgmorp}) and $m$ and $u$ are locality  $\K$-coalgebra morphisms\footnote{This condition was missing in \cite{CGPZ1}.} i.e.,
 \[\Delta\circ m\vert_{B^{\otimes_{\top} 2}} =\underbrace{(m\otimes m)}_{\text{domain}\, B^{\otimes_{\top} 4}}\circ \, \, ({\rm Id}_B\otimes \tau_{23}\otimes {\rm Id}_B)\, \, \circ \underbrace{ (\Delta\otimes \Delta)\vert_{B^{\otimes_{\top} 2}}}_{\text{range }\, B^{\otimes_{\top} 4}} ; \quad \epsilon \circ m= \epsilon \otimes \epsilon; \quad \Delta\circ u= u\otimes u;\quad \epsilon \circ u=  {{\rm Id}_\K} ,\] 
where $B^{\otimes_{\top} n}$ was defined in \eqref{eq:Votimesn}, and $\tau_{23}:B^{\otimes_\top^4}\to B^{\otimes_\top^4}$ is the map that switches the terms on the second and third position of the tensor.

\item Let $(B_i,\top_i,m_i,u_i,\Delta_i,\epsilon_i)$ ($i\in\{1,2\}$) be two locality $\K$-bialgebras. A {\bf locality $\K$-bialgebra morphism} from $B_1$ to $B_2$ is a locality map $f:B_1\longrightarrow B_2$ that is a morphism of locality algebras and of locality $\K$-coalgebras.

\item \cite[Proposition 4.9]{CGPZ1} Let $(B,\top,m,u,\Delta,\epsilon)$ be a locality bialgebra, and $\phi,\psi:B\to B$ two mutually independent locality linear maps. The locality convolution product of $\phi$ and $\psi$ is a locality linear map $B\to B$ defined by \[(\phi\star\psi) =m(\phi\otimes\psi)\Delta .\]
\item\cite[Definition 5.3 and Remark 5.4]{CGPZ1} A {\bf locality Hopf algebra} is a locality bialgebra $(H,\top,m,u,\Delta,\epsilon)$ together with a  locality linear map $S:B\to B$ such that $S$ and $Id_H$ are mutually independent and \[S\star Id_H=Id_H\star S=u {\circ}\epsilon.\]

\item Let $(H_i,\top_i,m_i,u_i,\Delta_i,\epsilon_i,S_i)$ be ($i\in\{1,2\}$) be two locality Hopf algebra. A {\bf locality Hopf algebra morphism} between $B_1$ and $B_2$ is a morphism of locality bialgebras $f:B_1\longrightarrow B_2$ such that $f\circ S_1=S_2\circ f$.

\item A {\bf sub-locality Hopf algebra} of $H$  of a locality Hopf algebra $(H,\top,m,u,\Delta,\epsilon,S)$ is a locality Hopf algebra $(H',\top',m',u',\Delta',\epsilon',S')$ contained in $H$   such that the injection map $f:H\hookrightarrow H'$ is a locality Hopf algebra morphism.

\item A {\bf graded (resp. filtered) locality Hopf algebra} is a locality Hopf algebra together with a grading (resp. filtration) which makes it a graded (resp. filtered) algebra and coalgebra  and such that \[S(H_n)\subset H_n \ ({\rm resp.}\ S(H^n)\subset H^n).\]
A {\bf connected locality Hopf algebra} is a graded locality algebra such that  $H_0$ has dimension $1$.
\end{itemize}
\end{defn}
\begin{rk}
     As in the algebra and coalgebra cases, our definition of sub-locality Hopf algebra is more general than the one used in \cite{CGPZ1} and will be used later.
    \end{rk}

\begin{ex}[The locality tensor algebra as a locality  Hopf algebra]\label{ex:tensorHopf}
	${\mathcal T}_\top (V)$  is a locality Hopf algebra in the sense of \cite[Definition 5.3]{CGPZ1} when equipped with the tensor product restricted to pairs in $\top_\otimes$  and the deconcatenation coproduct  defined  on  $x\in V$    by $\Delta_\shuffle(x )= 1\otimes x +x \otimes 1$  and inductively on the degree by 
\begin{equation}\label{eq:Deltashuffle}\Delta_\shuffle(x_{i_1}\otimes \cdots \otimes x_{i_n})=\sum_{J\subset (i_1, \cdots, i_n), \, w_J\top_\otimes w_{\bar J}}w_J\otimes w_{\bar J},
\end{equation} where we have set $w_J:=  x_{j_1}\otimes \cdots\otimes x_{j_k} $ for $J=(j_1,\cdots, j_k)$ and where $\bar J$ stands for the complement of $J$ in $\{i_1, \cdots, i_n\}$. The counit is defined by $\epsilon (x )=0$ for $x\in V$ and the antipode   is given by $S (x_{j_1}\otimes \cdots \otimes x_{j_k})= (-1)^k \, x_{j_k}\otimes \cdots \otimes x_{j_1}$.  

It is a connected graded cocommutative locality Hopf algebra of finite type.
\end{ex}

The usual result for the existence of an antipode in a graded connected bialgebra also holds in the locality setup.
\begin{prop}\label{prop:locantipodeforfree}\cite[Proposition 5.5]{CGPZ1}
Let $(B,\top,m,u,\Delta,\epsilon)$ be a graded connected locality bialgebra. There exists an antipode $S:B\to B$ such that $(B,\top,m,u,\Delta,\epsilon, S)$ is a locality Hopf algebra.
\end{prop}
 
We end this paragraph with a transposition to the locality setup of the known Lie algebra structure of the space $\text{Prim}(B)$ of primitive elements of a bialgebra $B$ (see for example \cite[Theorem 2.1.3]{Abe} for the usual result). Recall that a element $x\in B$ is called {\bf  primitive} if, and only if $\Delta x=x\otimes 1+1\otimes x$ and that a {\bf  graded locality Lie algebra} is a locality Lie algebra which is also a graded algebra for the Lie product (see Definitions \ref{defn:localityLieAlgebra} and \ref{defn:localg}).

\begin{prop} \label{prop:primitive_Lie}
	The space  $ (\text{\rm Prim}(B),\top_{\text{\rm Prim}(B)},[\cdot;\cdot ])$  of a (resp. graded) locality bialgebra $(B,\top)$  equipped with $\top_{\text{\rm Prim}(B)}$ the restriction of the locality relation $\top$ to primitive elements and   the usual commutator  
	$[x,y]=xy-yx$ (resp. $[x,y]=xy-(-1)^{|x|.|y|}yx$), is a (resp. graded) locality Lie algebra.
\end{prop}
\begin{proof} 
		We carry out the proof for the graded case, since the ungraded case can be obtained by setting all degrees to zero. 
		
	Let $m$ be the locality product of the locality bialgebra $(B,\top)$. Since for any $U\subseteq B$, $m(U^\top\otimes_\top U^\top)\subset U^\top$, 
	we have that $[[x,y],z]$ (and its permutations) is well-defined for any triplet $(x,y,z)\in B^{\times_\top 3}$. The rest of the proof goes 
	exactly as for the non locality case. In particular, for any $(x,y)\in\text{Prim}(B)\times_\top\text{Prim}(B)$ since 
	$\Delta$ is a locality algebra morphism, we have:
	\begin{align*}
	\Delta(xy) & = (m\otimes m)\circ\tau_{23}\circ(\Delta\otimes\Delta)(x,y) \\
	& = (m\otimes m)\circ\tau_{23}((x\otimes 1+1\otimes x)\otimes(y\otimes 1+1\otimes y)) \quad\text{since $x$ and $y$ are primitive elements}\\
	& = xy\otimes 1+(-1)^{|x||y|}y\otimes x + x\otimes y + 1\otimes xy. 
	\end{align*}
	$\Delta(yx)$ is obtained by exchanging $x$ and $y$. Putting everything together we obtain  {by linearity of $\Delta$}
	\begin{equation*}
	\Delta([x,y]) = [x,y]\otimes1 + 1\otimes[x,y]\in B\otimes_\top B
	\end{equation*}
	as needed. This, together with the fact that $\Delta$ is linear, implies that $(\text{Prim}(B),\top,[;])$ is a locality algebra.
\end{proof}

\subsection{The reduced coproduct and primitive elements} \label{subsect:technical_results}

This paragraph reviews preliminary well-known technical results, which we transpose to the locality setup.  As  before,  $\K$   is a commutative field of  characteristic zero.
\begin{lem}
Let $H$ be a locality graded, connected $\K$-bialgebra. The projection onto  ker$(\epsilon)$ along $\K\, (1_H)$
 	\begin{equation}
	\begin{array}{rlcl}
	\rho:&H&\longrightarrow&H\\
	&x&\longmapsto&x-\epsilon(x)1_H 
	\end{array}
	\end{equation}  is a locality linear map, which is independent of $\Delta$ in the following sense:
	 \begin{equation} \label{eq:localityrho}
	     (x,y)\in \top \Longrightarrow \left(\rho(x), \Delta(y)\right)\in\top_\otimes
	    \end{equation}
        with $\top_\otimes$ the locality relation on $\mathcal{T}_\otimes(H)\supseteq H$ of Definition \ref{defn:loc_rel_tensor_alg}.
\end{lem} 
\begin{proof} 
	Let  $x\top y$. Since $\K \subset H^{\top}$ we have $\epsilon(x)1_H\top y$ implying by linear locality that  $(x-\epsilon(x)1_H)\top y$. 
	Since $\epsilon(y)1_H\top (x-\epsilon(x)1_H)$, again by linear locality, we deduce that $x-\epsilon(x)1_H\top y-\epsilon(y)1_H$ and  conclude 
	that $\rho(x)\top\rho(y)$.
	
	To check the mutual independence of $\rho$ and $\Delta$, we consider again $x\in H$ and $y\in H$ such that $x\top y$.  {WE write $\Delta y=\sum y_i\otimes y_i^\prime$, so that Equation \eqref{eq:localityrho} amounts} to show that $\rho(x)\top  y_i$ and $\rho(x)\top  y_i^\prime$. Since $\rho(x)=x-\epsilon(x)\, 1_H$ and $\K\, 1_H \top H$, by linearity, it suffices to show that $x\top y_i$ and $x\top y_i^\prime$. But this follows from the fact that  $\Delta$ maps $ \{x\}^\top$ to $\{x\}^\top\otimes_{\top}\{x\}^\top $.
\end{proof} 
	On  a locality graded, connected $\K$-bialgebra  $H$, equipped  with a locality coproduct $\Delta$, we consider the coassociative locality linear map $\tilde{\Delta}:H\to \ker(\epsilon)\otimes_{\top} \ker(\epsilon)$ defined as 
	$\tilde{\Delta}(1)=0$, and for $x\in{\rm ker}(\epsilon)$ as $\tilde{\Delta}(x)=\Delta(x)-1\otimes  x - x\otimes 1$. 
	For   $n\geq 0$, we then define inductively $\tilde{\Delta}^{(n)}:H\to H^{\otimes_{\top}(n+1)}$   by:
	\begin{itemize}
		\item $\tilde{\Delta}^{(0)}=\rho$.
		\item $\tilde{\Delta}^{(1)}=\tilde{\Delta}$.
		\item $\tilde{\Delta}^{(n+1)}=(\tilde{\Delta}\otimes Id^{\otimes n})\circ \tilde{\Delta}^{(n)}$.
	\end{itemize}
	Since $\Delta$ and $\rho$ are locality linear maps, $\tilde{\Delta}$ and $\tilde{\Delta}^{(n)}$ are  also locality maps. Moreover, $\rho$ and $\Delta$ are mutually independent, and therefore $\rho$ and $\tilde{\Delta}$ are also mutually independent in the sense of  (\ref{eq:localityrho}). 
	\begin{prop}\label{prop:deltilasrho} 
		Let $(H,\top)$ with $H=\oplus_{k\in \Z_{\geq 0}}H_k$,  {be} a locality graded, connected $\K$-bialgebra. For every $n\geq1$ 
		\[ \tilde{\Delta}^{(n)}=(\rho\otimes\cdots\otimes\rho)\circ\Delta^{(n)}.\]
	\end{prop}
	\begin{proof}Recall that $H= {u}(\K)\oplus \ker(\epsilon)\simeq \K\oplus\ker(\epsilon) $.  
			We proceed by induction over $n$. For $n=1$
			we have $(\rho\otimes\rho)\circ\Delta(1)=\rho(1)\otimes\rho(1)=0=\tilde{\Delta}(1)$. By linearity, this extends replacing $1$ by
				any  $x\in\K$. Let $x\in \ker(\epsilon)$. Since  $\tilde{\Delta}(x)\in \ker(\epsilon) {\otimes_\top}\ker(\epsilon)$ and $(\rho\otimes \rho)\circ\tilde{\Delta}=\tilde{\Delta}$, it follows that 
			$(\rho\otimes \rho)\circ\Delta(x)=\rho(1)\otimes \rho(x)+\rho(x)\otimes \rho(1)+(\rho\otimes \rho)\circ\tilde{\Delta}(x)=\tilde{\Delta}(x)$, so that  the proposition holds for $n=1$.
			
			Suppose now that  the proposition holds true for $n-1$. We note that $\tilde{\Delta}(1)=0$ implies that $\tilde{\Delta}\circ\rho=\tilde{\Delta}$. Hence, 
			\begin{align*}
			(\rho\otimes \cdots\otimes \rho)\circ\Delta^{(n)} &= ((\rho\otimes\rho)\circ\Delta\otimes\rho\otimes\cdots\otimes\rho)\circ\Delta^{(n-1)}\\
			&=(\tilde{\Delta}\otimes\rho\otimes\cdots\otimes\rho)\circ\Delta^{(n-1)}\\
			&=(\tilde{\Delta}\otimes Id\otimes\cdots\otimes Id)\circ(\rho\otimes\cdots\otimes\rho)\circ\Delta^{(n-1)}\quad \text{by the initial induction step}\\
			&= (\tilde{\Delta}\otimes Id\otimes\cdots\otimes Id)\circ \tilde{\Delta}^{(n-1)}\quad \text{ by the induction step}\\
			&= \tilde{\Delta}^{(n)}. \qedhere
			\end{align*}
\end{proof}

The subsequent proposition combines elementary known results, which we recall for the sake of completeness. If $H=\bigoplus_{n\in\Z}H_n$ is a graded algebra, we say $|x|=n$ if $x\in H$ is a homogeneous component of degree $n$, i.e. $x\in H_n$.
\begin{prop}\label{prop:useful}
Let $(H,\top)$ with    $H=\oplus_{n\in \Z_{\geq 0}}H_n$, be a locality graded, connected  $\K$-bialgebra. Let $k\geq1$. 
\begin{enumerate}

	\item \label{lem:allesprimitn}For any $x\in H$ such that $\tilde{\Delta}^{(k-1)}(x)\neq 0$, we have 
	$\tilde{\Delta}^{(k)}(x)=0\Longrightarrow\tilde{\Delta}^{(k-1)}(x)\in{\rm Prim}(H)^{\otimes_{\top}k}$. 

	 \item \label{lem:lem69foissy}
	For every $x\in H$, $\tilde{\Delta}^{(k)}(x)=0$ if $k\geq \vert x\vert$ and  $\tilde{\Delta}^{(k-1)}(x)\in {\rm Prim}(H)^{\otimes_{\top}k}$ if $\vert x\vert=k$.

 \item \label{lem:deltcopprim}  
Let $n\geq 2$.
For any   $(v_1,\dots,v_n)\in {\rm Prim}(H)^{\times_{\top}n}$, we have 
$\tilde{\Delta}^{(k)}(v_1\cdots v_n)=0$ for any $k\geq n$ and the following refinement of the first item holds:
\begin{equation}\label{eq:deltan-1sigma}
 \tilde{\Delta}^{(n-1)}(v_1\cdots v_n) =\sum_{\sigma\in \lf{\mathfrak{S}_n}}v_{\sigma(1)}\otimes \cdots \otimes v_{\sigma(n)}~\in  {\rm Prim}(H)^{\otimes_\top^n},
\end{equation}
where $\lf{\mathfrak{S}_n}$ is the $n$-th symmetric group. 
\end{enumerate}
\end{prop}
\begin{proof}
	 \begin{enumerate} 
	 	\item 
	 	The coassociativity of $\tilde{\Delta}$ implies 
	 	\[(Id^{\otimes_{\top}(i-1)}\otimes\tilde{\Delta}\otimes Id^{\otimes_{\top}(k-i)})\circ\tilde{\Delta}^{(k-1)}(x)= \tilde{\Delta}^{(k)}(x)=0 \quad \forall 1\leq i\leq k.\]
	We infer that
	 	\[\tilde{\Delta}^{(k-1)}(x)\in \text{ker}(Id^{\otimes_{\top}(i-1)}\otimes\tilde{\Delta}\otimes Id^{\otimes_{\top}(k-i)})=H^{\otimes_{\top}(i-1)}\otimes_{\top}(\K\oplus {\rm Prim}(H))\otimes_{\top} H^{\otimes_{\top}(k-i)}.\]
	 Assuming that $\tilde{\Delta}^{(k-1)}(x)\neq 0$, since $\tilde{\Delta}^{(k-1)}(x)\in \text{ker}(\epsilon)^{\otimes_{\top}k}$  (which is a consequence of Im$(\rho)\subseteq\ker(\epsilon)$ and Proposition \ref{prop:deltilasrho}), it follows that
	 	\[\tilde{\Delta}(x)\in H^{\otimes_{\top}(i-1)}\otimes_{\top}{\rm Prim}(H)\otimes_{\top} H^{\otimes_{\top}(k-i)} \quad \forall 1\leq i\leq k.\]
	 	Thus,  $\tilde{\Delta}^{(k-1)}(x)$ lies in the intersection of all such spaces, which yields the statement.
	 	
	 	\item Using the coassociativity of $\widetilde \Delta$,  for any $x\in H$ with $\vert x\vert=n$,   we have (by induction on $k$)  $\widetilde\Delta^{(k)} x=\sum_{x} x^{(1)}\otimes \cdots \otimes x^{(k+1)}$, with $(x^{(1)},\dots, x^{(k+1)})\in H^{\times_{\top}k+1}$,  $\sum_{j=1}^{k+1} \vert x^{(j)}\vert=\vert x\vert=n$. If   $k\geq n$, this imposes that $\widetilde\Delta^{(k)} x=0$. In particular, $\widetilde\Delta^{(n)} x=0$. It then follows from the previous item that $\tilde{\Delta}^{(n-1)}(x)\in{\rm Prim}(H)^{\otimes_{\top}n}$.

 \item Let $(v_1, \cdots, v_n)\in {\rm Prim}(H)^{\times_\top^n}$ and $x:= v_1\cdots v_n$. To compute $\tilde{\Delta}^{(k)}(x)$ for any $k\leq n-1$, we proceed by induction on $ {k}$.
 	We prove first that $\Delta(v_1\cdots v_n)=\sum_{I\subset [n]}v_I\otimes v_{I^C}$ by induction over $n$, setting 
   $v_{\emptyset}=1$. For $n=1$, \[\Delta(v_1)=v_1\otimes 1+ 1\otimes v_1=\sum_{I\subset[1]}v_I\otimes v_{I^C}~\in\text{Prim}(H)^{\otimes_\top2}.\]
 	
 	Now assume it is true for $n-1$. The compatibility of the product and the coproduct yields
 	\begin{align*}
 	\Delta(v_1\cdots v_{n-1}v_n)&=\Delta(v_1\cdots v_{n-1})\Delta(v_n)\\
 	&=\left(\sum_{I\subset[n-1]}v_I\otimes  v_{I^C}\right)(v_n\otimes 1+1\otimes v_n)\\
 	&=\sum_{I\subset[n-1]}v_I {\cdot} v_n\otimes v_{I^C} + \sum_{I\subset[n-1]}v_I\otimes v_{I^C} {\cdot} v_n\\
 	&=\sum_{I\subset [n]}v_I\otimes v_{I^C}~\in\text{Prim}(H)^{\otimes_\top^n}.
 	\end{align*}
 	It then follows by induction on $k$ that \[\Delta^{(k)}(v_1\cdots v_n)=\sum_{I_1\sqcup\dots\sqcup I_{k+1}=[n]}v_1\otimes\cdots\otimes v_{I_{k+1}}~\in\text{Prim}(H)^{\otimes_\top^n}.\]
 	
Using  Proposition \ref{prop:deltilasrho}, we then easily derive the expression  of $\tilde{\Delta}^k(v_1\cdots v_n)$   composing with $\rho^{\otimes}(k+1)$. Note that for $I=\emptyset$, $\rho(v_I)=0$, otherwise $\rho(v_I)=v_I$. We have
 	\[\Delta^{(k)}(v_1\cdots v_n)=\sum_{{{I_1\sqcup\dots\sqcup I_{k+1}=[n]}\atop{I_1,\dots,I_{k+1}\neq\emptyset}}}v_1\otimes\cdots\otimes v_{I_{k+1}}~\in\text{Prim}(H)^{\otimes_\top^{k+1}}.\]
 	
 	For $k=n-1$ each of the sets $I_j$   only contains one element so that we get the expected formula. For  {$k\geq n$} this expression vanishes since some of the sets are empty. This ends the proof of the statement.
 \end{enumerate}
 \end{proof}
The following result will be useful in the sequel.
\begin{prop}\label{prop:lem70foissy} 
Let $H$ be a filtered connected (with respect to the grading induced by the filtration) locality $\K$-bialgebra and $J\neq\{0\}$ a locality left, right or two-sided coideal of $H$ such that $J\cap H^0=\{0\}$. Then $J$ contains non zero primitive elements of $H$.
\end{prop}
\begin{proof}
Consider the filtration  $J^n=J\cap H^n, n\in \Z_{\geq 0}$ on $J$ induced by the one on $H$. Since   $J\neq\{0\}$, there is some element $0\neq x\in J$   of minimum degree $k$ among the elements of $J$. Explicitly,  $x\in J^k$ and $J^n=\{0\}$ for every $n<k$. The existence of $x$ is guaranteed. Indeed, let us write 
\[\Delta(x)=x\otimes 1+1\otimes x+ \sum x'\otimes x''\]
with $\sum x'\otimes x''\in \ker(\epsilon)\otimes_{\top}\ker(\epsilon)$. Since $\Delta$ respects the filtration, if $x'$ and $x"$ are non zero, there are   integers $0<m, n<k$   such that $ x'\in H^n$ and $x''\in H^m$ with $n+m\leq k$. If  $J$ is a left, resp. right coideal, then $x'$, resp.  $x''$  lies in $J$, which contradicts the minimality of $k$. Therefore at least one of the two elements $x'$ or $x''$ vanishes, which implies $\sum x'\otimes x''=0$ so that $x$ is primitive.
\end{proof}
  
\section{The Milnor-Moore theorem: locality version}  \label{section:eight}

 The results of Subsection \ref{subsect:technical_results} hold  independently of  the validity of the statements \ref{conj:main2} and \ref{conj:univenvalg}. However, assuming these statements hold true, they will be used in an essential way in the proof of  a locality version of a central theorem of Hopf algebra theory, namely the Milnor-Moore theorem. 
 
 In the sequel we  therefore assume that  the conjectural statements  {\ref{conj:main2} and \ref{conj:univenvalg}} hold true and transpose to the locality setup the well-known  the Milnor-Moore theorem, also known as Cartier-Quillen-Milnor-Moore.  The original theorem was proven by Milnor and Moore in 1965 \cite{MM}  using a rather categorical language. The idea of the proof we  follow stems from Cartier \cite{Car1} and Patras \cite{Pat}. We found a nice explanation of this proof in   lecture notes by Loïc Foissy \cite{Foi1}, on which most of  the sequel is based.

\subsection{The locality universal enveloping algebra as a Hopf algebra}  

The Milnor-Moore theorem relates a Hopf algebra $H$ with a Hopf algebra built from the universal envelopping algebra of the Lie algebra of primitive elements of $H$. In order to show the Milnor-Moore theorem in the locality setup, we need to show that the locality universal enveloping algebra of a locality Lie algebra $(\g,\top)$  indeed admits a locality Hopf algebra structure.

\begin{prop} \label{prop:universal_envelopping_alg_is_Hopf}
Given a locality Lie algebra $(\g, \top)$ and assuming the statement \ref{conj:univenvalg} holds true for $(\g, \top)$, the universal enveloping algebra $U_{\top}(\g)$ together with the locality relation $\top_U$ (from Definition \ref{defn:luenvelopingalgebra}) can be equipped with a locality cocommutative Hopf algebra structure.
\end{prop}
\begin{proof}
So far, assuming that the conjectural statement \ref{conj:univenvalg} holds true for $(\g,\top)$,  we  know from Proposition \ref{prop:Utopgloc}, that \[ (U_{\top}(\g),\top_U,\bigotimes,u)\] is an associative unital locality algebra. In order to equip it with a coproduct, we consider the map $\delta:\g\to U_{\top}(\g)\otimes_{\top} U_{\top}(\g)$ defined by $\delta(x)= {\iota_\g(x)}\otimes 1+1\otimes  {\iota_\g(x)}$,  {where $\iota_\g:\g\to U_\top(\g)$ is the canonical map}. One can check that it is locality linear and $\delta([x,y])=\delta(x)\delta(y)-\delta(y)\delta(x)$. Hence, Theorem \ref{thm:univproplocunivenvalg} (which applies since we have assumed that the conjectural statement \ref{conj:univenvalg} holds  true for $(\g,\top)$) gives the existence and uniqueness of a locality algebra morphism $\Delta:U_{\top}(\g)\to U_{\top}(\g)\otimes_{\top}U_{\top}(\g)$ which extends $\delta$. Note that by construction the elements in $ {\iota_\g(\g)}$ are primitive and $\Delta$ is cocommutative. For the counit we consider the zero map from $\g$ to $\K$. This is indeed a locality Lie algebra morphism and once again by Theorem \ref{thm:univproplocunivenvalg}, there is a unique locality algebra morphism $\epsilon:U_{\top}(\g)\to \K$ which vanishes identically on $ {\iota_\g(\g)}$. Therefore $U_{\top}(\g)$ with this coproduct and counit is a filtered connected bialgebra over $\K$. Consider the locality Lie algebra morphism $\sigma:\g\to U_\top(\g)$ defined by $\sigma(x)=- {\iota_\g(x)}$. Once more, Theorem  \ref{thm:univproplocunivenvalg} gives the existence and uniqueness of a locality algebra morphism $S:U_\top(\g)\to U_\top(\g)$ which extends $\sigma$. To prove that it is an antipode, for $\alpha\in\K\subset U_\top(\g)$, we see that $S\star I(\alpha)=\alpha\, (S(1)I(1))=\alpha$ and for $x_i\in {\iota_\g(\g)}$, we have $S\star I(x_1\cdots x_k)=\sum_{J\subset [k]}(-1)^{|J|}x_1\cdots x_k=0$ 
	which shows that $S\star I=u\circ\epsilon$. Similarly, one shows that $I\star S=u\circ\epsilon$, so that $S$ is an antipode and $U_\top(\g)$ is a locality Hopf algebra.
\end{proof}

\begin{prop} \label{prop:sub_Lie_envelopping}
     Let $(\g,\top,[,]_\top)$ be locality Lie algebra and $(\mathfrak{g}',\top',[,]|_{\top'})$ a sub-locality Lie algebra of $\mathfrak{g}$. Assuming that the conjectural statement \ref{conj:univenvalg} holds true for $\g'$ and $\g$, then $U_{\top'}(\g')$ is a sub-locality Hopf algebra of $U_{\top}(\g)$.
    \end{prop}
    \begin{proof} 
    By means of Proposition \ref{prop:CompareUnivAlgebras}, $U_{\top'}(\g')$ is a sub-locality algebra of $U_{\top}(\g)$. \sy{The inclusion map $ \g'\hookrightarrow \g$ induces an injective map}  { $\iota: U_{\top'}(\g')\to U_\top(\g)$ it is easy to see that}  {the counits relate by} $\epsilon'=\epsilon\circ\iota$. Moreover, since the coproduct on the universal enveloping algebra are completely determined by the coproduct on $ {\iota_\g(\g)}$ (resp $ {\iota_{\g'}(\g')}$) (see the proof of Proposition \ref{prop:universal_envelopping_alg_is_Hopf}), it is easy to see that   {$U_\top(\g')$} is a sub-locality bialgebra of $U_\top(\g)$. 
    
    { Let $S_{U'}$ be the antipode on $U_\top(\g')$, then $$S_{U'}=S_U\circ\iota$$ follows from the fact the antipodes are completely determined by their action on $\iota_g(\g)$ (resp. $\iota_{\g'}(\g')$), and for $x\in\iota_{\g'}(\g')$,
\[S_{U'}(x)=-x=-\iota(x)=S_U(\iota(x)).\]
This proves that $U_\top(\g')$ is indeed a sub-locality Hopf algebra of $U_\top(\g)$.}   
\end{proof}

For the proof of  locality version of the Milnor-Moore theorem, we use 
a description of the primitive elements of the universal enveloping algebra, which  usually comes as a corollary of the Poincar\'e-Birkhof-Witt theorem.
Since there is no locality version of this theorem available yet,   we provide an  alternative proof using Zorn's lemma, of the following classical statement extended to a locality set up.
\begin{prop}\label{prop:pbwcor}
    Given a locality Lie algebra $(\g, \top)$ over   $\K$,  the conjectural statement \ref{conj:univenvalg} holds true,   and assuming    the axiom of choice, the  set of primitive elements of the locality universal enveloping algebra $U_\top(\mathfrak{g})$ coincides with $ {\iota_\g(\mathfrak{g})}$: 
    \[{\rm Prim}(U_\top(\mathfrak{g}))= {\iota_\g(\mathfrak{g})},\]
  {where $\iota_\g$ is the canonical map from $\g$ to $U_\top(\g)$.}
 
\end{prop}
\begin{proof}
By the very construction of the coproduct on $U_\top(\mathfrak{g})$, we have $ {\iota_\g(\g)}\subset {\rm Prim}(U_\top(\mathfrak{g}))$. 
For the other inclusion, for any $n\in \N$ we consider the set
\begin{equation}\label{eq:grading_U_g}
 G_n:=\{x_1^{k_1}\cdots x_m^{k_m}\in U_\top(\mathfrak{g})|\ (x_1,\dots,x_m)\in {\iota_\g(\g)}^{\times_{\top_{U}}^m}\ \wedge\ k_i\in\N\  \wedge\sum_{i=[m]} k_i\leq n\}.
\end{equation}
which by construction, generates the space $(U_\top(\mathfrak{g}))^n$ of filtration degree $n$ in the natural filtration of $U_\top(\g)$.  Lemma \ref{lem:compbasis} applied to $E=\K\, 1$ yields the existence of a vector space basis $\{1\}\subset B_1\subset G_1$   of $(U_\top(\mathfrak{g}))^1$, so for filtration degree $1$. Since
$B_1 \subset G_2$, Lemma \ref{lem:compbasis} yields a basis $B_2$ of $(U_\top(\g))^2$ such that $B_1\subset B_2\subset G_2$. 
We proceed inductively to build  $B:=\bigcup_{n\in\N}B_n$ which is a Hammel (vector space) basis of $U_\top(\g)$. We use the simplified notation 
$\vec{x}^{\vec{k}}:=x_1^{k_1}\cdots x_n^{k_n}$ and $|\vec{k}|:=k_1+\cdots+k_n$. Note that for $\vec{x}^{\vec{k}}\in B$ with $|\vec{k}|=n$, 
it is linearly independent of $B_{n-1}$.

A primitive element $y$ of $U_\top(\mathfrak{g})$  can be expressed in terms of the basis $B$ as 
\[y=\sum_{\vec{x}^{\vec{k}}\in B}\alpha_{\vec{x}^{\vec{k}}}\, \vec{x}^{\vec{k}}\] 
where only finitely many $\alpha_{\vec{x}^{\vec{k}}}$ are non zero. Let $N={\rm  max}\{|\vec{k}|:\alpha_{\vec{x}^{\vec{k}}}\neq0\}$. If $N=1$ we have $y\in {\iota_\g(\g)}$ as required.  Let us now assume that $N>1$. Then $\tilde{\Delta}^{(N-1)}(y)=0$ since $y$ is primitive. By (\ref{eq:deltan-1sigma}), one can write  
\[0=\tilde{\Delta}^{(N-1)}(y)=\sum_{|\vec{k}|=N}\alpha_{\vec{x}^{\vec{k}}}\sum_{\sigma\in S_n}x_{\sigma(1)}\otimes\cdots\otimes x_{\sigma(N)}.\] 
with  $x_{\sigma(1)}\otimes\cdots\otimes x_{\sigma(N)}\in {\rm Prim}(H)^{\otimes_\top^N}$.

Applying the   product $m$ yields
\[0=m^{(N-1)}(\tilde{\Delta}^{(N-1)}(y))=\sum_{|\vec{k}|=N}\alpha_{\vec{x}^{\vec{k}}}\sum_{\sigma\in S_n}x_{\sigma(1)}\cdots x_{\sigma(N)}\in U_{\top}(\g).\] 
Since  
${\iota_\g(\g)}=(U_\top(\g))^1\ni[x_i,x_j]=x_i x_j-x_j x_i$   for every $i$ and $j$, we may reorder the $x_i$'s to get the original elements of $B$ at the cost 
of adding some lower order terms (l.o.t.) (with respect to the natural filtration of $U_{\top}(\g)$ given by the sets \eqref{eq:grading_U_g}). The resulting products arising in the new  linear combination are linearly independent of the leading term due to the very manner the basis $B$ was constructed. Hence, we have  
\[0=\sum_{|\vec{k}|=N}\frac{\alpha_{\vec{x}^{\vec{k}}}}{N!}\vec{x}^{\vec k}+ {\rm l.o.t.}\] 
Since the elements of the basis $B$ are linearly independent, we may conclude that all $\alpha_{\vec{x}^{\vec{k}}}=0$ except if $N=1$. 
Therefore ${\rm Prim}(U_\top(\mathfrak{g}))\subset\di{\iota_\g(\g)}$. Thus ${\rm Prim}(U_\top(\mathfrak{g}))=\di{\iota_\g(\g)}$.
\end{proof}

 {\begin{rk}
In the non-locality case, the \sy{injectivity of} the canonical map $\iota_\g:\g\to U(\g)$,  \sy{see (\ref{eq:iotag})},  follows from the Poincar\'e-Birkhof-Witt theorem. However, this is not necessarily true in the locality setup. In general $\iota_\g(\g)$ is a quotient of $\g$, with the property that $U_\top(\g)\simeq U_\top(\iota_\g(\g))$.

A case which will be of particular interest in the following section, is when $\g$ is the Lie algebra of primitive elements of a locality bialgebra $(B,\top)$. In that case, the universal property of the locality universal enveloping algebra (Theorem \ref{thm:univproplocunivenvalg-preloc}) yields the existence of a pre-locality algebra morphism $\phi:U_\top(\g)\to B$ which extends the canonical injection $\iota:\g\to B$. Therefore, the canonical map $\iota_\g:\g\to U_\top(B)$ is indeed injective in such cases.
\end{rk}
}

\subsection{A locality Cartier-Quillen-Milnor-Moore theorem}   

Assuming  that the conjectural statement \ref{conj:main2} holds true, we can now prove a locality version of the  Cartier-Quillen-Milnor-Moore theorem.

\begin{thm}[Locality Cartier-Quillen-Milnor-Moore theorem] \label{thm:CQMMTheorem} Let $(H,\top)$ be a graded, connected, cocommutative locality Hopf algebra. $H$ is generated by its primitive elements ${\rm Prim}(H)$ as a locality algebra, which we assume satifies    the conjectural statements \ref{conj:univenvalg} and \ref{conj:main2}. In that case, we   have the following isomorphism of    
locality Hopf algebras:
\begin{equation}\label{eq:iso}(H,\top)\underset{{\rm loc}}{\simeq}(U_\top({\rm Prim}(H)),\top_U)
\end{equation}
with $\top_U$ the locality relation of Definition \ref{defn:luenvelopingalgebra} in the case $\g={\rm Prim}(H)$.
\end{thm}

\begin{proof}
\begin{itemize}
	\item We   first prove that $H$ is generated by its primitive elements ${\rm Prim}(H)$ as a locality algebra. Let $H'$ be the locality subalgebra 
of $H$ generated by ${\rm Prim}(H)\cup\{1\}$ and let  $0\neq x\in H$. By Proposition \ref{prop:useful} \ref{lem:lem69foissy}., for $k$ big enough, $\tilde{\Delta}^{(k)}(x)=0$. 
Set ${\rm deg}_p(x)$ to be the minimum of all such integers $k$. Using induction over ${\rm deg}_p(x)$, we show that $H\subset H'$. If ${\rm deg}_p(x)=0$, then by 
definition of $\tilde{\Delta}^{(0)}$, $\rho(x)=0$. Hence $x\in\K\subset H'$. If ${\rm deg}_p(x)=1$, then $x$ is a primitive element  so that $x\in H'$. If $n={\rm deg}_p(x)>0$, by means of   Proposition \ref{prop:useful}   \ref{lem:allesprimitn}. We have
\[\tilde{\Delta}^{(n-1)}(x)=\sum_ix_1^{(i)}\otimes\cdots\otimes x_n^{(i)}\]
where all the $x_j^{(i)}$ are primitive elements. The cocommutativity of $H$ implies invariance under the natural action of the symmetric group 
so we have:
\[\tilde{\Delta}^{(n-1)}(x)=\frac{1}{n!}\sum_{\sigma\in S_n}\sum_i x_{\sigma(1)}^{(i)}\otimes\cdots\otimes x_{\sigma(n)}^{(i)} \]
Exchanging the two sums (both over finite sets), we then recognize   in the innermost sum the expression of 
$\tilde\Delta^{(n-1)}(x^{(i)}_1\cdots x^{(i)}_n)$ given by   Proposition \ref{prop:useful}  \ref{lem:deltcopprim}. The linearity of $\tilde{\Delta}^{(n-1)}$ then yields 
\[\tilde{\Delta}^{(n-1)}(x)=\tilde{\Delta}^{(n-1)}\left(\frac{1}{n!}\sum_ix_1^{(i)}\cdots x_n^{(i)}\right)~\Longleftrightarrow~\tilde{\Delta}^{(n-1)}\left(x-\frac{1}{n!}\sum_ix_1^{(i)}\cdots x_n^{(i)}\right)=0.\]
Hence,
by definition of the degree, 
\[{\rm deg}_p\left(x-\frac{1}{n!}\sum_ix_1^{(i)}\cdots x_n^{(i)}\right)<n.\]
By the induction hypothesis, this element lies in $H'$, and so does $x\in H'$. We infer that $H'=H$ and $H$ is generated by its primitive elements ${\rm Prim}(H)$ as a locality algebra.
 \item 
To build the isomorphism   (\ref{eq:iso}) we assume  that the conjectural statement \ref{conj:main2} holds true.  By Proposition \ref{prop:primitive_Lie}, the set Prim$(H)$ of primitive elements of $H$ has a graded locality Lie algebra structure. 
We can therefore build its enveloping algebra $U_{\top}({\rm Prim}(H))$. Assuming   that the conjectural statement \ref{conj:main2} holds true, we can apply the universal property  (Theorem 
\ref{thm:univproplocunivenvalg})  to extend  the locality Lie 
algebra morphism given by the injection $i:{\rm Prim}(H)\to H$ to a locality algebra morphism 
\[\phi:U_{\top}({\rm Prim}(H))\to H,\] 
which stabilizes
the elements in ${\rm Prim}(H)$. 
%
Moreover, since $\phi$ is a locality algebra morphism, again by \cy{Lemma \ref{lem:Kerlocideal}} the range of $\phi$ is a locality subalgebra of $H$ which contains $H'$. Hence, by the 
first part of this proof, $\phi$ is surjective. 

\cy{\item Let us show that $\phi$ is a coalgebra morphism. The coproducts $\Delta_U$ on $U_\top({\rm Prim}(H))$ and $\Delta$ on $H$ are by definition algebra morphisms. Since they coincide on the set ${\rm Prim}(H)$,  $(\phi\otimes\phi)\circ\Delta_U=\Delta\circ\phi$ on ${\rm Prim}(H)$. This identity  extends  everywhere since ${\rm Prim}(H)$ generates both $U_{\top}({\rm Prim})$ and $H$ as locality algebras. We still need to show $\epsilon_U=\epsilon\circ\phi$, with $\epsilon_U$ the counit of $U_\top({\rm Prim}(H))$ and $\epsilon$ the counit of $H$. Again, since ${\rm Prim}(H)$ generates both $U_{\top}({\rm Prim})$ and $H$ as locality algebras, it is enough to show that these maps coincide on Prim$(H)$. On the one hand, by definition of $\epsilon_U$ (see the proof of Proposition \ref{prop:universal_envelopping_alg_is_Hopf}), $\epsilon_U$ vanishes on Prim$(H)$. On the other hand, for any $h\in{\rm Prim}(H)$, using the property of $\epsilon$ and the canonical identification $\K\otimes H\simeq H$, we have:
\begin{equation*}
 h=(\epsilon\otimes {\rm Id}_H)\circ\Delta(h)=\epsilon(h)\otimes 1_H+\epsilon(1)\otimes h=\epsilon(h)1_H+h.
\end{equation*}
Therefore $\epsilon$ vanishes on Prim$(H)$ as required and $\phi$ is a coalgebra morphism.
}

\item We prove the injectivity of $\phi$ ad absurdum. The fact that $\phi$ is a locality coalgebra morphism,  implies by Lemma \ref{lem:Kerloccoideal},  that its 
kernel ${\rm Ker}(\phi)$ is a locality coideal of $U_{\top}({\rm Prim}(H))$. 
Since $\phi$ is an algebra morphism, $\phi(1)=1$ so   $\rm{Ker}(\phi)\cap\K\, 1=\{0\}$. By  Proposition  \ref{prop:lem70foissy} 
applied to $J=\ker(\phi)$, assuming the latter is non trivial,  it must contain  a primitive element of 
$U_{\top}({\rm Prim}(H))$. 
However this leads to  a contradiction since, by Proposition \ref{prop:pbwcor}, ${\rm Prim}(U_{\top}({\rm Prim}(H)))={\rm Prim}(H)$ which is 
fixed by $\phi$, 
therefore none of them lies in the kernel. Hence $\phi$ is injective.

\item To show that $\phi$ is an locality isomorphism  of Hopf algebras, we still need to prove that $\phi^{-1}:H\longrightarrow U_\top({\rm Prim}(H))$ is a locality algebra morphism .The previous items give the existence of the inverse map $\phi^{-1}$. By definition of $\phi$, on any element $h$ of $H$, $\phi^{-1}$ acts as:
\begin{equation*}
 \phi^{-1}(h) = x_1\otimes\cdots\otimes x_n
\end{equation*}
for some $n\in\N^*$ and primitive elements $x_i$ such that $h=x_1\cdots x_n$. Here the right-hand-side  actually stands for an equivalence class of tensor products, which we write as a tensor to simplify notations. We now distinguish two cases.

If  $n=1$, then $h$ is a primitive element of $H$ and $\phi^{-1}$ restricted to primitive elements is simply the projection from $H$ to $U_\top({\rm Prim}(H))$. This map is a locality map by definition of the locality on the quotient space $U_\top({\rm Prim}(H))$.

If $n\geq2$,  using Equation \eqref{eq:deltan-1sigma} of Proposition \ref{prop:useful},  we have
\begin{equation*}
 \tilde{\Delta}^{(n-1)}(h) = \tilde{\Delta}^{(n-1)}(x_1\cdots x_n) = \sum_{\sigma\in S_n}x_{\sigma(1)}\otimes \cdots \otimes x_{\sigma(n)}.
\end{equation*}
The fact that $\phi^{-1}$ is a locality map then follows from combining together the  fact that $\tilde{\Delta}^{(n-1)}$   is a locality map (which, as already pointed out in Section \ref{subsect:technical_results}, follows from  $\rho$ and $\Delta$ being locality maps),   that the identity permutation is one term of the right-hand-side of the above equation, and the definitions of the locality relations on $\mathcal{T}_\top({\rm Prim}(H))$ and $U_\top({\rm Prim}(H))$.

 \item Then by Proposition \ref{prop:universal_envelopping_alg_is_Hopf}, $U_\top({\rm  Prim}(H))$ is a Hopf algebra, and since $\phi$ is a morphism of graded locality bialgebras and an isomorphism, it is an isomorphism of locality Hopf algebras.  \qedhere
 
\end{itemize}
\end{proof}

\subsection{Consequences of the Milnor-Moore theorem}

We end the paper with some useful consequences of the locality Milnor-Moore Theorem.

Throughout this paragraph, we assume that for any locality Hopf algebra $H$ under consideration,  the set ${\rm Prim}(H)$ of its primitive elements, viewed as a locality Lie algebra, satifies    the conjectural statements \ref{conj:univenvalg} and \ref{conj:main2}.

It follows from  the locality Milnor-Moore Theorem \ref{thm:CQMMTheorem}, that the locality relation on a graded, connected, cocommutative Hopf algebra is entirely determined  by the locality relation on its primitive elements.

More precisely:
\begin{cor}
 Let $(H_1,m_1,\Delta_1,\top_1)$ and $(H_2,m_2,\Delta_2,\top_2)$ be two graded, connected, cocommutative locality Hopf algebras. Then if
 \begin{equation*}
  ({\rm Prim}(H_1),[,]_1,\tilde\top_1)\simeq({\rm Prim}(H_2),[,]_2,\tilde\top_2)
 \end{equation*}
 as locality Lie algebras (heere $[,]_i$ is the antisymmetrisation of the product $m_i$) and wehe we have set   $\tilde\top_i:=\top_i\cap({\rm Prim}(H_i)\times{\rm Prim}(H_i))$), then 
 \begin{equation*}
  (H_1,m_1,\Delta_1,\top_1)\simeq(H_2,m_2,\Delta_2,\top_2)
 \end{equation*}
 as locality Hopf algebras.
\end{cor}
\begin{proof}
By  the universality property of universal locality algebras, the isomorphism $({\rm Prim}(H_1),[,]_1,\tilde\top_1)\simeq({\rm Prim}(H_2),[,]_2,\tilde\top_2)$ implies that the universal envelopping algebras of these Lie algebras are isomorphic as locality algebras. As a result of the construction of their respective coproducts (presented in Proposition \ref{prop:universal_envelopping_alg_is_Hopf}), they are isomorphic as locality Hopf algebras. The result then follows from Theorem \ref{thm:CQMMTheorem}.
\end{proof}
This further leads to the observation that  locality Hopf algebras are not generally speaking  ordinary locality Hopf algebras with an "added" locality relation. In other words, one cannot simply ``turn on'' locality, at least when the locality Milnor-Moore theorem applies.

\begin{cor}\label{cor:Nogocorollary}
Let $(H,m,\Delta)$ be a graded, connected, cocommutative Hopf algebra. The   trivial locality relation  $\top=H\times H$ is the only locality relation $\top$ on $H$ such that  such that $(H,\top,m|_\top,\Delta)$ is a locality Hopf algebra.
\end{cor}
\begin{proof}
We proceed by contradiction, assuming  such a non-trivial locality relation $\top$ exists. 

 Since we saw that  locality on $H$ is determined by the locality on primitive elements,  if $\top$ is not the trivial locality relation, there are   primitive elements $a$ and $b$ such that $(a,b)\notin\top$. 
 Indeed, if this were not true, then $\top|_{{\rm Prim}(H)\times{\rm Prim}(H)}={\rm Prim}(H)\times{\rm Prim}(H)$ would  be the trivial locality on ${\rm Prim}(H)$ and hence on $H$ thanks to the locality Milnor-Moore theorem.  

Let $a$ and $b$ be primitive elements such that $(a,b)\notin \top$. Setting  $y=m(a,b)$, on the one hand $y\in(H,\top)$ implies that $\Delta(y)\subset H\otimes H$. On the other hand \[\Delta(y)=\Delta(a)\Delta(b)=(a\otimes1+1\otimes a)(b\otimes1+1\otimes b)=y\otimes 1+a\otimes b+b\otimes a+1\otimes y\notin H\otimes_\top H.\]
But this contradicts the inclusion $\Delta(H)\subset H\otimes_\top H$, which follows from the fact that the  coproduct $\Delta$ of the Hopf algebra coincides with the locality coproduct of the locality Hopf algebra.
\end{proof}

Theorem \ref{thm:CQMMTheorem}  \cy{also} allows to describe examples of locality Hopf algebras.
\begin{ex}
With the notations of Example	\ref{ex:tensorHopf}, let $P_V$ be the subspace of primitive elements of $\mathcal{T}_\top(V)$. Endowed with the locality Lie bracket induced by the commutator, namely $[a,b]=a\otimes b-b\otimes a$ whenever $(a,b)\in\top_\otimes$, $P_V$ is a locality lie algebra. It then follows  from Theorem \ref{thm:CQMMTheorem} that
\[\left({\mathcal T}_\top(V), \top_\otimes\right)\simeq\left( {\mathcal U}_\top( P_V), \top_U\right), \]
this last isomorphism is an isomorphism of locality Hopf algebras.
\end{ex}
Theorem \ref{thm:CQMMTheorem} allows to build sub-locality Hopf algebras from sub-locality Lie algebras.
	\begin{cor}\label{cor:subalgsubHop}   Let $(H,\top)$ be a graded, connected, cocommutative locality Hopf algebra whose set  ${\mathfrak g}:={\rm Prim}(H)$ of primitive elements obeys    the conjectural statements \ref{conj:main2} and \ref{conj:univenvalg}.
	 There is    a one-to-one correspondence
		\begin{equation} \label{eq:correspondence_Hg}\left((H, \top)\supset (H', \top')\right)\quad \longleftrightarrow \quad \left((\g', \top')\subset (\g, \top)\right) 
		\end{equation}
between    graded, connected, cocommutative sub-locality Hopf algebras of $H$ and  sub-locality Lie algebras of $(\g={\rm Prim}(H), \top)$ where the localities $\top$ and $\top'$ on the r.h.s. are the restrictions of the ones of the l.h.s.
	\end{cor}
	\begin{proof}
		\begin{itemize}
			\item Let $\g'$ be a subset of $\g=$Prim$(H)$  with locality $\top'\subset\top$ so that $(\g',\top',[, ]|_{\top'})$ is sub-locality Lie algebra of $(\g,\top,[,]_\top)$. Then by Proposition \ref{prop:sub_Lie_envelopping}, $U_{\top'}(\g')$ is a sub-locality algebra of $U_\top(\g)$. It is therefore  isomorphic to a sub-locality Hopf algebra $(H', \top')\simeq U_{\top'}(\g')$ of $(H, \top)\simeq U_\top(\g)$. \
			
			\item Conversely, let $H'$ be a sub-locality Hopf algebra of $H$ with locality $\top'\subseteq\top$. Then $\g:={\rm Prim}(H')\subseteq\g={\rm Prim}(H)$. By Proposition \ref{prop:primitive_Lie}, setting $\top'':=\top'\cap(\g'\times\g')$, then $(\g',\top'',[,]|_{\top''})$ is  a sub-locality Lie algebra of $\g$. Thus we have a map from sub-locality Hopf algebras of $H$ and sub-locality Lie algebras of $\g$.
			
			\item The two maps built above are inverse of each other by construction, which proves the corollary.
		\end{itemize}
		
	\end{proof}
	
 \begin{rk} \label{rk:important}
		It follows from  Corollary \ref{cor:Nogocorollary} that  $H=H'$ and ${\mathfrak g}={\mathfrak g}'$  implies $\top=\top'$. Note that this statement is trivially satisfied in the usual (non-locality) setup where $\top=\top'$ is the trivial locality. The subsequent Proposition \ref{prop:locality_GL} illustrates the case when $\mathfrak g=\mathfrak g'$ but $H\neq H'$, which is specific to the locality setup, \cy{since} it cannot occur in the non-locality setup due to the Milnor-Moore theorem.
\end{rk}
We apply Corollary \ref{cor:subalgsubHop} to the Hopf  algebra of rooted forests. Let us start by recalling some classical  combinatorial concepts  { \cite{Foi2, Kre}}. 
\begin{defn} \label{defn:rooted_forests_GL}
 \begin{itemize}
  \item An {\bf admissible cut} of a rooted forest $F$ is a subset $c$ of edges of $F$ such that any path from one the roots of $F$ to a leaf of $F$ meet $c$ at most once. We write Adm$(F)$ the set of admissible cuts of $F$. For $c\in{\rm Adm}(F)$, we write $R_c(F)$ the subforest of $F$ below the  $c$ and $T_c(F)$ the subforest of $F$ above the cut $c$. Notice that we have $|V(F)|=V(R_c(F))|+|V(T_c(F))|$.
  \item For $F_1$, $F_2$ and $F$ three rooted forests, we set 
  \begin{equation*}
   n(F_1,F_2,F):=|\{c\in{\rm Adm}(F)|R_c(F)=F_1~\wedge~T_c(F)=F_2\}|.
  \end{equation*}
  Notice that for $F_1$ and $F_2$ two given rooted forests, $n(F_1,F_2,F)=0$ except for a finite number of rooted forests $F$.
  \item We define the product of two rooted forests $F_1$ and $F_2$ by
  \begin{equation*}
   F_1*F_2 = \sum_{F\in\mathcal{F}}n(F_1,F_2,F)F.
  \end{equation*}
It is well-defined since the sum contains finitely many non-zero terms.
  \item We further define the coproduct $\Delta_*$ by its action on rooted forests $F=T_1\cdots T_n$ defined as
  \begin{equation*}
   \Delta_*(F=T_1\cdots T_n) = \sum_{I\subseteq [n]}T_I\otimes T_{[n]\setminus I}
  \end{equation*}
  with, for $I\subseteq[n]$, we have $T_I:=\prod_{i\in I}T_i$.
 \end{itemize}
\end{defn}
  $(\mathcal{F},*,\Delta_*)$ is a Hopf algebra, which is the dual of the Connes-Kreimer Hopf algebra, sometimes  called the Grossman-Larson Hopf algebra to which it is isomophic, see \cite{Pan}  (with corrections in \cite{Hof,Foi2}). The Grossman-Larson algebra, whose primitive elements are rooted trees, was introduced in \cite{GL1,GL2,GL3}.

In order to introduce locality, we first decorate the rooted forests. Recall that, for a set $\Omega$, a {\bf $\Omega$-decorated} forest is a pair $(F,d_F)$ with $F$ a forest and $d_F:V(F)\longrightarrow\Omega$ a map. We often omit the $d_F$ to lighten the notation. We also write $\mathcal{F}_\Omega$ the set of $\Omega$-decorated forests. The notions of Definition \ref{defn:rooted_forests_GL} easily generalise to decorated forests. 
\begin{defn}\cite[Definition 3.1]{CGPZ2}
 Let $(\Omega,\top)$ be a locality set. A {\bf properly} $\Omega$-decorated forest is a $\Omega$-decorated forest $(F,d_F)$ such that any disjoint pair of vertices of $F$ are decorated by independent elements of $\Omega$:
 \begin{equation*}
  \forall (v_1,v_2)\in V(F)\times V(F),~v_1\neq v_2~\Longrightarrow~d_F(v_1)\top d_F(v_2).
 \end{equation*}
 We write $\mathcal{F}_\Omega^{\rm prop}$ the set of properly $\Omega$-decorated forests. We endow $\mathcal{F}_\Omega^{\rm prop}$ with a locality relation $\top_{\mathcal{F}_\Omega}$ induced from the relation $\top$ on $\Omega$:
 \begin{equation*}
  (F_1,d_1)\top_{\mathcal{F}_\Omega}(F_2,d_2)~:\Longleftrightarrow~\forall(v_1,v_2)\in V(F_1)\times V(F_2),~d_1(F_1)\top d_2(F_2).
 \end{equation*}
\end{defn}
The Hopf algebra $(\mathcal{F},*\vert_{\top_{\mathcal{F}_\Omega}} ,\Delta_*)$ induces a locality Hopf algebra structure on $\mathcal{F}_\Omega^{\rm prop}$:
\begin{prop} \label{prop:locality_GL}
Given a locality set $(\Omega,\top)$, the quadruple $(\mathcal{F}_\Omega^{\rm prop},\top_{\mathcal{F}_\Omega},*_{\top_{\mathcal{F}_\Omega}} ,\Delta_*)$ is a graded, connected, cocommutative locality Hopf algebra equipped with the product $*_{\top_{\mathcal{F}_\Omega}} $ given by the restriction of the $*$ product of Definition \ref{defn:rooted_forests_GL} to the graph  $\top_{\top_{\mathcal{F}_\Omega}} $ of the locality relation.
\end{prop}
\begin{proof} 
Applying Corollary \ref{cor:subalgsubHop}   to the locality Lie algebra $\left(\mathfrak g:={\rm Prim}({\mathcal F}_\Omega^{\rm prop}), \top_{\rm triv}\right)$ equipped with the trivial locality relation $\top_{\rm triv}:=\mathfrak g\times \mathfrak g$  and $\left(\mathfrak g':={\rm Prim}({\mathcal F}_\Omega^{\rm prop}), \top_{{\mathcal F}_\Omega}\right)$ shows that  $\mathcal{F}_\Omega^{\rm prop}:= {\mathcal U}_{\top_{\mathcal{F}_\Omega }}\left( \mathfrak g'\right)$ is a locality sub-Hopf algebra of $\left(\mathcal{F}_\Omega={\mathcal U}_{\top_{\rm triv}}\left( \mathfrak g\right), \top_{\rm triv}\right)$.
\end{proof}

The locality Milnor-Moore theorem \ref{thm:CQMMTheorem} also provides refinements of properties in the ordinary setup, here a decomposition  involving mutually independent arguments.
\begin{cor}
Given  a locality set $(\Omega,\top)$,   any properly $\Omega$-decorated rooted forest $F=T_1\cdots T_n$  can be expressed as a linear combination of $*$-products of  finitely many pairwise independent properly $\Omega$-decorated rooted trees $t^{(i)}_j$: 
 \begin{equation*}
  F=T_1\cdots T_n = \sum_{n=1}^N \alpha_n t_1^{(n)}*\cdots *t_{p_n}^{(n)}
 \end{equation*}
 for   $\alpha_1, \cdots, \alpha_N\in\R$.
\end{cor} \begin{proof}  For the sake of simplicity, throughout the proof, we set $\top:=\top_{\mathcal{F}_\Omega}$.

 Let $\phi:U_\top({\rm Prim}(\mathcal{F}_\Omega^{\rm prop}))\longrightarrow\mathcal{F}_\Omega^{\rm prop}$ be the isomorphism of locality Hopf algebra given by Theorem \ref{thm:CQMMTheorem}. Since $U_\top({\rm Prim}(\mathcal{F}_\Omega^{\rm prop}))$ is the set $\Omega$-decorated rooted trees, for any such tree $t$ we have $\phi([t])=t$.
 
The map $\phi$ being an isomorphism, for any properly $\Omega$-decorated rooted forest $F=T_1\cdots T_n$ there  is an element $\sum_{n=1}^N\alpha_n t_1^{(n)}\cdot\cdots \cdot t_{p_n}^{(n)}\in U_\top({\rm Prim}(\mathcal{F}_\Omega^{\rm prop})$ (where we write $\cdot$ for the product of $U_\top({\rm Prim}(\mathcal{F}_\Omega^{\rm prop})$) such that 
 \begin{equation*}
  F=T_1\cdots T_n = \phi\left(\sum_{n=1}^N\alpha_n t_1^{(n)}\cdot\cdots \cdot t_{p_n}^{(n)}\right) = \sum_n\alpha_{n=1}^N \phi([t_1^{(n)}])*\cdots *\phi([t_{p_n}^{(n)}]) = \sum_{n=1}^N\alpha_n t_1^{(n)}*\cdots *t_{p_n}^{(n)},
 \end{equation*}
 where we have used that $\phi$ is a morphism of algebras.
\end{proof}


\vspace{0.5cm}

\newpage

\part*{Appendix: Alternative locality tensor product}

\addcontentsline{toc}{part}{Appendix: Alternative locality tensor product}

\renewcommand\thesection{\Alph{section}}

\setcounter{section}{1}

\setcounter{subsection}{0}

\setcounter{theorem}{0}

For the sake of completeness we provide an alternative construction of the locality tensor product  already suggested on \cite{CGPZ1}. As in \cite{CGPZ1}, in practice, we choose to work with Definition \ref{defn:loctensprod1} since the alternative tensor product is in general not a subspace of   the usual (non-locality) tensor product. Notice that this alternative locality tensor product induces an alternative definition of $\top_\times$-bilinearity as  mentioned in the sequel.

Alternatively to $I_{\rm bil}(V)$ defined in Equations (\ref{eq:bilforms1}) to (\ref{eq:bilforms4}), we can consider the linear subspace ${\rm I_{bil,\top}}\subseteq \K( V\times_\top W)$ generated by all 
elements in $V\times_\top V$ of the forms (\ref{eq:bilforms1}) to (\ref{eq:bilforms4}) such that each argument in the linear combinations (\ref{eq:bilforms1})--(\ref{eq:bilforms4}), lies  in $\top$. In some cases ${\rm I_{bil,\top}}$ will coincide with ${\rm I_{bil}}\cap\K(\top)$ but as we see in the following, this does not generally hold.

\begin{ex}\label{ex:countex2} 
Consider the locality vector space $V=\R^2$, 
\begin{equation*}
 \top=\R^2\times\{0\}\cup\{0\}\times\R^2\cup\K( e_1+e_2)\times\K(e_1)\cup\K(e_1+2e_2)\times\K(e_2)\cup\K(e_1)\times\K(e_1+e_2)\cup\K(e_2)\times\K(e_1+2e_2)
\end{equation*}
and the element of $\K(V\times_\top V)$ \[y=(-e_1-e_2,e_1)+(-e_1-2e_2,e_2)+(e_1,e_1+e_2)+(e_2,e_1+2e_2).\]
Now we can write $(-e_1-e_2,e_1)=\Big((-e_1-e_2,e_1)+(e_1+e_2,e_1)\Big) - (e_1+e_2,e_1) =: y_1 - (e_1+e_2,e_1)$ with $y_1\in{\rm I_{bil}}$. Writing 
the same type of expansions for $(-e_1-2e_2,e_2)$, $(e_1,e_1+e_2)$ and $
 y = y_1+y_2+y_3+y_4  - (e_1+e_2,e_1) - (e_1+2e_2,e_2) + (e_1,e_1)+(e_1,e_2) + (e_2,e_1) + (e_2,2e_2)$
with $y_i\in{\rm I_{bil}}$ for $i=1,\cdots,4$. Writing further 
$- (e_1+e_2,e_1) = \Big(- (e_1+e_2,e_1) + (e_2,e_1)+(e_1,e_1)\Big) - (e_2,e_1)-(e_1,e_1)=:y_5- (e_2,e_1)-(e_1,e_1)$ with $y_5\in{\rm I_{bil}}$. 
Writing the same type of expansions for $- (e_1+2e_2,e_2)$ and $(e_2+2e_2)$ we obtain
\begin{align*}
 y & = y_1+y_2+y_3+y_4+y_5- (e_2,e_1)-(e_1,e_1) + y_6 - (e_1,e_2) - (2e_2,e_2) \\
 &+ (e_1,e_1)+(e_1,e_2) + (e_2,e_1) + y_7 + (e_2,2e_2) \\
   & = y_1+y_2+y_3+y_4+y_5+ y_6 + y_7 - (2e_2,e_2) + (e_2,2e_2)\\
   & = y_1+y_2+y_3+y_4+y_5+ y_6 + y_7 +y_8
\end{align*}
with $y_i\in{\rm I_{bil}}$ for $i=1,\cdots,8$. Thus $y\in{\rm I_{bil}}$. 
Note that we left $\K(V\times_\top V)$ in the last step to get to this conclusion: $y_i\notin{\rm I_{bil,\top}}$ for $i=1,\cdots,8$, thus 
$y\notin{\rm I_{bil,\top}}$.
\end{ex}

In general only the inclusion ${\rm I_{bil,\top}}\subset{\rm I_{bil}}\cap\K(V\times_\top V)$ holds, which leaves us with another possible definition of the locality tensor product.

\begin{defn}\label{defn:loctensprod2} 				
Let $(V,\top)$ be a locality vector space, the alternative locality tensor product is defined as \begin{equation}\label{eq:loctensprod2}
V\otimes^{\top}V:=\sfrac{\K(V\times\top V)}{{\rm I_{bil,\top}}}
\end{equation} 
\end{defn}

\begin{ex}\label{ex:countex3} 
Going back to Example (\ref{ex:countex2}), the alternative locality tensor product is \[V\otimes^{\top}V=\K\{(e_1+e_2)\otimes e_1, (e_1+2e_2)\otimes e_2, e_1\otimes(e_1+e_2),e_2\otimes(e_1+2e_2)\}\nsubseteq V\otimes V\] It has dimension $4$ like the usual tensor product since we already saw that all those elements that span it are linearly independent as elements in $V\otimes^{\top}V$ but is not the same as the usual one, because in the usual tensor product those $4$ elements mentioned above are not linearly independent. However, the locality tensor product in this case is \[V\otimes_{\top}V=\K\{(e_1+e_2)\otimes e_1, (e_1+2e_2)\otimes e_2, e_1\otimes(e_1+e_2)\}\subset V\otimes W\] it has dimension 3 since some of the tensor products involve arguments that are not linearly independent as elements of $V\otimes_{\top}V$, and this is a vector subspace of the usual tensor product.
\end{ex}
 {Using the same methods as in the paper, one can show that this alternative tensor product also enjoys universal properties (under similar conjectural assumptions as above) using an alternative definition} of $\top_\times$-bilinearity, namely requiring that $\bar f(I_{bil,\top})=\{0\}$.

 \end{document}